\documentclass[11pt]{amsart}
\usepackage{amsmath,amssymb,amsthm,mathtools}
\usepackage[shortlabels]{enumitem}
\usepackage[margin=1.18in]{geometry}
\usepackage[protrusion=true,expansion=false]{microtype}
\usepackage{booktabs}
\usepackage{xcolor}
\usepackage{hyperref}
\hypersetup{colorlinks=true,citecolor=black,linkcolor=black,urlcolor=black,pdfauthor={Hamed Ebadi},pdftitle={Cohomology of Linearly Constrained Kloosterman Families}}
\allowdisplaybreaks[2]
\theoremstyle{plain}
\newtheorem{mainthm}{Theorem}

\newtheorem{theorem}{Theorem}[section]
\newtheorem{proposition}[theorem]{Proposition}
\newtheorem{lemma}[theorem]{Lemma}
\newtheorem{corollary}[theorem]{Corollary}
\theoremstyle{definition}
\newtheorem{definition}[theorem]{Definition}
\theoremstyle{remark}
\newtheorem{remark}[theorem]{Remark}
\numberwithin{equation}{section}

\newcommand{\Qlb}{\overline{\mathbb{Q}}_\ell}
\newcommand{\Fq}{\mathbb{F}_q}
\newcommand{\Fp}{\mathbb{F}_p}
\newcommand{\Fpbar}{\overline{\mathbb{F}}_p}
\newcommand{\Fqbar}{\overline{\mathbb{F}}_q}
\newcommand{\Gm}{\mathbb{G}_m}
\newcommand{\Gmn}[1]{\mathbb{G}_{m,#1}}
\newcommand{\Lpsi}[1]{\mathcal{L}_\psi(#1)}
\newcommand{\Kl}{\mathrm{Kl}}
\newcommand{\Ggeom}{G_{\mathrm{geom}}}
\newcommand{\Iw}{I^{\mathrm{w}}}
\newcommand{\Up}[1]{\mathrm{U}(#1)}
\DeclareMathOperator{\Tr}{Tr}
\DeclareMathOperator{\Frob}{Frob}
\DeclareMathOperator{\Swan}{Swan}
\DeclareMathOperator{\SL}{SL}
\DeclareMathOperator{\GL}{GL}
\DeclareMathOperator{\Sym}{Sym}
\DeclareMathOperator{\conv}{conv}
\DeclareMathOperator{\Vol}{Vol}
\DeclareMathOperator{\rk}{rank}
\DeclareMathOperator{\End}{End}
\DeclareMathOperator{\Hom}{Hom}
\DeclareMathOperator{\gr}{gr}
\DeclareMathOperator{\FT}{FT}
\DeclareMathOperator{\pr}{pr}

\title[Cohomology of constrained Kloosterman families]{Cohomology of Linearly Constrained Kloosterman Families: Newton Polytopes, Weights, and Boundary Monodromy}
\author{Hamed Ebadi}
\address{K. N. Toosi University of Technology, Tehran, Iran}
\email{h.ebadi1@email.kntu.ac.ir}
\date{October 2026}
\subjclass[2020]{Primary 11L05, 14F20; Secondary 11T23, 14G15}
\keywords{Kloosterman sums, exponential sums, Newton polyhedra, weights, Swan conductors, middle convolution, geometric monodromy, $\ell$-adic cohomology}

\begin{document}

\begin{abstract}
Let $p$ be an odd prime, $\psi$ a nontrivial additive character of $\mathbb F_p$, $\psi_q=\psi\circ\mathrm{Tr}_{\mathbb F_q/\mathbb F_p}$ and $\lambda\in\mathbb F_p^\times$. For $h\in(\mathbb F_q^\times)^n$ let $\Sigma_n(h)$ be the sum of $\psi_q(u_1^{-1}+\cdots+u_n^{-1})$ over $u\in(\mathbb F_q^\times)^n$ with $h_1u_1+\cdots+h_nu_n=-\lambda$, an additively twisted average of products of Kloosterman sums. For every $n\ge2$, the cohomology computing $\Sigma_n(h)$ is concentrated in one degree and has dimension $2^n-1$. It has $r(n)=2^n-\binom n{\lfloor n/2\rfloor}$ eigenvalues of weight $n-1$, and its other eigenvalues are the integers $q^{k-1}$, $1\le k\le\lfloor n/2\rfloor$, with ballot numbers as multiplicities. The part of top weight forms a lisse sheaf $\mathcal W_n$ of rank $r(n)$ on the torus $\mathbb G_m^n$. We determine its local monodromy at the coordinate boundary and its determinant: $\det\mathcal W_n$ is the Artin--Schreier sheaf of the sum $-2^{n-1}\lambda^{-1}(h_1+\cdots+h_n)$ of the critical values of the phase, twisted by $p^{(n+1)r(n)/2}$. The geometric monodromy group of $\mathcal W_n$, and the arithmetic one after normalization to weight $0$, are equal to $\{g\in\mathrm{GL}_{r(n)}:\det(g)^p=1\}$. Consequently the normalized sums become equidistributed like traces of Haar-random elements of $\{g\in\mathrm U(r(n)):\det(g)^p=1\}$. The proofs use Laumon's Fourier transform, middle convolution, the Cayley cubic surface and Larsen's alternative, but no classification of finite groups.
\end{abstract}

\maketitle

\section{Introduction}\label{sec:intro}

Let $p$ be an odd prime, let $\psi$ be a nontrivial additive character of $\Fp$, and for a power $q$ of $p$ put $\psi_q=\psi\circ\Tr_{\Fq/\Fp}$. The Kloosterman sums
\[
K(a)=\sum_{x\in\Fq^\times}\psi_q(x^{-1}+ax)\qquad(a\in\Fq)
\]
are, up to sign, the trace function of Deligne's Kloosterman sheaf $\Kl$. This is a lisse sheaf of rank two on $\Gm$, pure of weight one, tamely and unipotently ramified at $0$ and totally wild at $\infty$, whose geometric monodromy group is $\SL_2$ \cite{SGA45,Katz-Gauss}. This paper studies the family obtained by restricting the phase $u_1^{-1}+\cdots+u_n^{-1}$ to hyperplane sections of the torus. For $n\ge2$, a fixed $\lambda\in\Fp^\times$ and $h\in\Fq^n$ put
\begin{equation}\label{eq:intro-Sigma}
\Sigma_n(h)=\sum_{\substack{u\in(\Fq^\times)^n\\ h_1u_1+\cdots+h_nu_n=-\lambda}}\psi_q\big(u_1^{-1}+\cdots+u_n^{-1}\big).
\end{equation}
Orthogonality of additive characters gives
\begin{equation}\label{eq:intro-factor}
q\,\Sigma_n(h)=\sum_{t\in\Fq}\psi_q(\lambda t)\prod_{i=1}^nK(th_i),
\end{equation}
so $\Sigma_n(h)$ is an additively twisted average of products of Kloosterman sums at proportional arguments. The size of $\Sigma_n(h)$, and the statistics of the family as $h$ varies, are governed by two things: the cohomology of the hyperplane section $Y_h=\{u\in\Gm^n:\langle h,u\rangle=-\lambda\}$ with coefficients in $\Lpsi{u_1^{-1}+\cdots+u_n^{-1}}$, and the monodromy of the resulting sheaves on the parameter torus. For the general framework of trace functions see \cite{FKM}. Our aim is to determine both: the cohomology of each fibre, with its weights and the exact values of its eigenvalues of lower weight, and the geometric and arithmetic monodromy groups of the resulting sheaves.

Put
\[
r(n)=2^n-\binom{n}{\lfloor n/2\rfloor},\qquad m_{n,k}=\binom nk-\binom n{k-1}\quad(0\le k\le\lfloor n/2\rfloor),
\]
with $\binom n{-1}=0$. The $m_{n,k}$ are ballot numbers, and $\sum_km_{n,k}=\binom n{\lfloor n/2\rfloor}$. Weights are taken in the sense of Deligne \cite{Deligne}.

\begin{mainthm}\label{thmA}
Let $n\ge2$ and $h\in(\Fq^\times)^n$. Then $H^i_c\big(Y_h\otimes\Fqbar,\Lpsi{\sum_iu_i^{-1}}\big)=0$ for $i\ne n-1$, and $H^{n-1}_c$ has dimension $2^n-1$. Its Frobenius eigenvalues are $r(n)$ eigenvalues of weight $n-1$ together with, for each $1\le k\le\lfloor n/2\rfloor$, the eigenvalue $q^{k-1}$ with multiplicity $m_{n,k}$. Consequently
\[
\Sigma_n(h)=(-1)^{n-1}\Big(\tau_q(h)+\sum_{k=1}^{\lfloor n/2\rfloor}m_{n,k}\,q^{k-1}\Big),\qquad|\tau_q(h)|\le r(n)\,q^{(n-1)/2},
\]
where $\tau_q(h)$ is the sum of the eigenvalues of weight $n-1$.
\end{mainthm}

For $n=2,3,4,5$ this reads $|\Sigma_2(h)+1|\le2q^{1/2}$, $|\Sigma_3(h)-2|\le5q$, $|\Sigma_4(h)+2q+3|\le10q^{3/2}$ and $|\Sigma_5(h)-5q-4|\le22q^2$. If $h\in\Fq^n$ has exactly $a$ nonzero coordinates, the cohomology is still concentrated in degree $n-1$, and its dimension is $2^a-1$ (Corollary~\ref{cor:all-h}).

\begin{mainthm}\label{thmB}
Let $S=\Gm^n$ over $\Fp$, with coordinates $h=(h_1,\dots,h_n)$. There is a lisse sheaf $\mathcal G_n$ of rank $2^n-1$ on $S$ whose stalk at a geometric point $h$ is $H^{n-1}_c\big(Y_h,\Lpsi{\sum_iu_i^{-1}}\big)(-1)$. It carries a filtration by lisse subsheaves whose graded pieces are the geometrically constant sheaves $\Qlb(-k)^{\oplus m_{n,k}}$ for $1\le k\le\lfloor n/2\rfloor$, and a lisse sheaf $\mathcal W_n$, pure of weight $n+1$ and of rank $r(n)$, as top quotient. For $h\in S(\Fq)$,
\[
\Tr(\Frob_h\mid\mathcal W_n)=(-1)^{n-1}q\,\Sigma_n(h)-\sum_{k=1}^{\lfloor n/2\rfloor}m_{n,k}\,q^k .
\]
\end{mainthm}

In particular $\mathcal G_n$ is lisse across the divisors of $S$ along which a critical point of the phase escapes to infinity (Remark~\ref{rem:collision}). The boundary of $S$ in $(\mathbb P^1)^n$ consists of the divisors $h_i=0$ and $h_i=\infty$, and by symmetry it suffices to consider $h_1$. Fix $c=(c_2,\dots,c_n)\in(\Fpbar^\times)^{n-1}$ and the line $C_c=\{(h_1,c_2,\dots,c_n):h_1\in\Gm\}\subset S$. For $j\ge1$ let $\kappa_j$ be the number of Jordan blocks of size $j$ in the tensor product of $n-1$ unipotent Jordan blocks of size two; thus $\kappa_{n-2k}=m_{n-1,k}$, and $\kappa_j=0$ for all other $j$.

\begin{mainthm}\label{thmC}
\begin{enumerate}[(i)]
\item $\mathcal W_n|_{C_c}$ is tamely ramified at $h_1=0$, and its local monodromy there is unipotent, with $r(n-1)$ Jordan blocks of size $2$, $\kappa_1$ blocks of size $1$, and no block of size $\ge3$.
\item The local monodromy of $\mathcal W_n|_{C_c}$ at $h_1=\infty$ is the direct sum of a tame part and a part of dimension $2^{n-1}$ that is isoclinic of slope $1$. The tame part is unipotent, with $\kappa_j$ Jordan blocks of size $j-2$ for every $j\ge3$. The part of slope $1$ is of the form $\Lpsi{-h_1/\lambda}\otimes X$, where all slopes of $X$ are $<1$. In particular $\Swan_\infty(\mathcal W_n|_{C_c})=2^{n-1}$.
\item $\det(\mathcal W_n|_{C_c})$ is geometrically isomorphic to $\Lpsi{-2^{n-1}\lambda^{-1}h_1}$.
\item The perverse sheaf $\mathcal W_n|_{C_c}[1]$ is isomorphic to the multiplicative middle convolution of $\Kl[1]$ with $\iota^*\mathcal M_c[1]$, where $\iota(t)=t^{-1}$ and $\mathcal M_c=\bigotimes_{i=2}^n[\times c_i]^*\Kl\otimes\Lpsi{\lambda t}$.
\end{enumerate}
\end{mainthm}

\begin{mainthm}\label{thmD}
\begin{enumerate}[(i)]
\item For every $n\ge2$ and every $q$,
\[
\sum_{h\in\Fq^n}\Sigma_n(h)=(-1)^nq^{n-1},\qquad\sum_{h\in\Fq^n}|\Sigma_n(h)|^2=q^{n-2}\Big((q-1)^{n+1}+1-(-1)^n(q-2)\Big).
\]
\item For every $n\ge3$,
\[
\sum_{h\in\Fq^n}|\Sigma_n(h)|^4=2q^{3n-2}+O_n\big(q^{3n-2-\delta(n)}\big),\qquad\delta(n)=\min\Big(1,\frac{n-2}2\Big).
\]
\item For $n=2$ and fixed $p$, $\sum_{h\in\Fq^2}|\Sigma_2(h)|^4=2q^4+O_p\big(q^{7/2}\big)$ as $q\to\infty$ through powers of $p$.
\end{enumerate}
\end{mainthm}

\begin{mainthm}\label{thmE}
Let $n\ge2$ and let $p$ be an odd prime.
\begin{enumerate}[(i)]
\item $\det\mathcal W_n\cong\Lpsi{-2^{n-1}\lambda^{-1}(h_1+\cdots+h_n)}\otimes\alpha^{\deg}$ with $\alpha=p^{(n+1)r(n)/2}$, where $\alpha^{\deg}$ is the geometrically constant sheaf on which $\Frob_p$ acts by $\alpha$. Equivalently, by Theorem~\ref{thmB} and the Chebotarev density theorem, for every power $q$ of $p$ and every $h\in(\Fq^\times)^n$, the product of the $r(n)$ eigenvalues of weight $n-1$ in Theorem~\ref{thmA} is $q^{(n-1)r(n)/2}\psi_q\big(-2^{n-1}\lambda^{-1}(h_1+\cdots+h_n)\big)$, and
\[
\det\Big(\Frob_q\mid H^{n-1}_c\big(Y_h\otimes\Fqbar,\Lpsi{\textstyle\sum_iu_i^{-1}}\big)\Big)=q^{(n-2)2^{n-1}+1}\,\psi_q\big(-2^{n-1}\lambda^{-1}(h_1+\cdots+h_n)\big).
\]
\item The geometric monodromy group of $\mathcal W_n$ is
\[
\Ggeom(\mathcal W_n)=\big\{g\in\GL_{r(n)}(\Qlb):\det(g)^p=1\big\}.
\]
In particular $\Ggeom(\mathcal W_n)^0=\SL_{r(n)}$, and the component group is cyclic of order $p$.
\item Let $\mathcal W_n^0=\mathcal W_n\otimes\gamma^{\deg}$, where $\gamma\in\Qlb$ corresponds to the positive real number $p^{-(n+1)/2}$. Then $\mathcal W_n^0$ is pure of weight $0$, and its arithmetic monodromy group is equal to $\Ggeom(\mathcal W_n)$.
\end{enumerate}
\end{mainthm}

The function $-2^{n-1}\lambda^{-1}(h_1+\cdots+h_n)$ is the sum of the values of the phase at its $2^{n-1}$ critical points (Remark~\ref{rem:critical-values}). For the equidistribution statement, put $c_n(q)=\sum_{k=1}^{\lfloor n/2\rfloor}m_{n,k}q^{k-1}$; by Theorem~\ref{thmA}, $q^{-(n-1)/2}|\Sigma_n(h)-(-1)^{n-1}c_n(q)|\le r(n)$ for $h\in(\Fq^\times)^n$.

\begin{mainthm}\label{thmF}
Let $n\ge2$, let $p$ be an odd prime, and let $K_n=\{g\in\mathrm U(r(n)):\det(g)^p=1\}$. As $q\to\infty$ through powers of $p$, the $(q-1)^n$ complex numbers
\[
q^{-(n-1)/2}\big((-1)^{n-1}\Sigma_n(h)-c_n(q)\big),\qquad h\in(\Fq^\times)^n,
\]
become equidistributed with respect to the image under $g\mapsto\Tr g$ of the Haar probability measure on $K_n$. In particular, the numbers $q^{-(n-1)/2}\big|\Sigma_n(h)-(-1)^{n-1}c_n(q)\big|$ become equidistributed with respect to the law of $|\Tr g|$ for $g$ a Haar-random element of $\mathrm{SU}(r(n))$.
\end{mainthm}

Theorems~\ref{thmA}--\ref{thmF} are proved as Theorems~\ref{thm:fibre}, \ref{thm:lisse}, \ref{thm:boundary} and~\ref{thm:infinity}, Proposition~\ref{prop:low-moments} with Theorem~\ref{thm:fourth} and Corollary~\ref{cor:fourth-n2}, Theorems~\ref{thm:det}, \ref{thm:alpha}, \ref{thm:Ggeom} and~\ref{thm:Garith}, and Corollary~\ref{cor:equidistribution}, respectively.

\subsection*{On the proofs}
The Newton polytope of the phase $F_n(u,t;h)=\sum_iu_i^{-1}+t(\lambda+\langle h,u\rangle)$ on the torus $\Gm^n\times\Gm$ is a pyramid over an $n$-dimensional cross-polytope. Its Newton polyhedron at infinity is the bipyramid obtained by adding the origin as a second apex (Section~\ref{sec:polytope}). This explains two numerical facts. For generic $h$ the phase has $2^{n-1}$ critical points, whereas the cohomology of the torus has dimension $2^n$, the normalized volume of the bipyramid. And since the origin is a vertex rather than an interior point, nothing forces the cohomology to be pure; in fact it is mixed.

To prove Theorem~\ref{thmA} we integrate over $u$ first. The cohomology of the torus becomes the cohomology on the $t$-line of the tensor product $\bigotimes_i[\times h_i]^*\Kl\otimes\Lpsi{\lambda t}$. The unipotent monodromy of $\Kl$ at $0$, decomposed by the Clebsch--Gordan rule, produces the lower weights, and the fibre over $t=0$ removes exactly one line of weight zero. The lower eigenvalues themselves come from the Frobenius on the nearby fibre of $\Kl$ at $0$, which is $\operatorname{diag}(1,q)$ in a basis adapted to the monodromy; on the tensor power, the invariants of the monodromy are the highest-weight vectors of the Clebsch--Gordan decomposition, and their degree determines the eigenvalue. Theorem~\ref{thmB} follows from the semicontinuity theorem of Deligne and Laumon, because all Swan conductors involved are independent of $h$.

Along a coordinate line the family is a multiplicative convolution of $\Kl$ with a sheaf that is totally wild at $0$ and unipotent at $\infty$. We first show, using Laumon's local Fourier transform, that the top-weight sheaf is tame at $h_1=0$. The Jordan structure is then read off from a product formula, due to Rojas-Le\'on, for the unipotent parts of a middle convolution. We deduce it from Deligne's fibre functor on Katz's Tannakian category of perverse sheaves on $\Gm$, through a dictionary between Jordan blocks and weights. The identification with the middle convolution follows from a comparison of generic ranks, together with the absence of punctual subobjects. At $h_1=\infty$, Laumon's principle of stationary phase, applied to the same Fourier model, shows that all slopes are at most $1$ and that the part of slope $1$ comes from a single singular point of the Fourier model, at which the sheaf is twisted by $\Lpsi{-h_1/\lambda}$. A count of dimensions and Swan conductors then gives Theorem~\ref{thmC}(ii) and~(iii).

The first two moments in Theorem~\ref{thmD} are evaluated exactly by orthogonality. The fourth moment rests on an exact four-parameter identity. On the generic stratum its terms are expressed through plane sections of the Cayley cubic surface $e_3(u_1,\dots,u_4)=0$, which are irreducible plane cubics away from the strata carrying the main term. For $n=2$ the error terms of this computation are as large as the main term, and Theorem~\ref{thmD}(iii) is deduced instead from Theorem~\ref{thmE}.

For Theorem~\ref{thmE}, the boundary supplies a nontrivial unipotent element, and the determinant has finite image. For $n\ge3$ the fourth moment gives $M_4=2$, and Larsen's alternative, whose short proof we include, shows that the identity component of the monodromy group is $\SL_{r(n)}$; for $n=2$ a direct argument in dimension two replaces the fourth moment. The determinant is then identified on every coordinate line by Theorem~\ref{thmC}(iii). An Artin--Schreier argument shows that a character of $p$-power order of $\pi_1(\Gm^n\otimes\Fpbar)$ that is trivial on all coordinate lines is trivial, which gives Theorem~\ref{thmE}(i) up to the value of $\alpha$, and hence the component group. The constant $\alpha$ is then computed by induction on $n$. At the boundary point $h_1=0$ of a coordinate line, the Fourier model above is the Fourier transform of a perverse sheaf, and its nearby cycles there sit in an exact sequence of Frobenius modules. The subobject is the cohomology of the complement of a hyperplane section in an $(n-1)$-dimensional torus. The quotient is formed by the vanishing cycles, which come from the point at infinity of the source of the Fourier transform, where the sheaf is unramified; it is a Tate twist of the cohomology of a hyperplane section in $n-1$ variables. Comparing determinants of Frobenius with the identity of Theorem~\ref{thmE}(i) on the line gives $\alpha=p^{(n+1)r(n)/2}$, and with it the arithmetic monodromy group in Theorem~\ref{thmE}(iii). Theorem~\ref{thmF} then follows from Deligne's equidistribution theorem. Neither the classification of finite simple groups nor that of finite primitive linear groups is used.

\subsection*{Logical structure}
Theorem~\ref{thmA} rests on the facts (K1)--(K5) about $\Kl$, Lemmas~\ref{lem:integrate-t}, \ref{lem:breaks} and~\ref{lem:CG}, Propositions~\ref{prop:fibration}, \ref{prop:weights-torus} and~\ref{prop:frob-invariants}, and Lemma~\ref{lem:t0}. Theorem~\ref{thmB} rests on Theorem~\ref{thmA}, on the theorem of Deligne and Laumon (Theorem~\ref{thm:DL}) and on Lemma~\ref{lem:averages}. Theorem~\ref{thmC}(i) and~(iv) rest on Theorem~\ref{thmB}; on Proposition~\ref{prop:tame}, which uses Laumon's local Fourier transform; and on the product formula of Theorem~\ref{thm:RL}, which we deduce from Deligne's fibre functor \cite[Th.~3.1 and Cor.~3.2]{Katz-Conv}, two facts on Katz's arithmetic Tannakian category \cite[Ch.~4 and~5]{Katz-Conv}, and the weight dictionary of Lemma~\ref{lem:weight-dictionary}. Theorem~\ref{thmC}(ii) and~(iii) rest in addition on Laumon's principle of stationary phase, through Proposition~\ref{prop:slopes-infinity}. Theorem~\ref{thmD}(i) and~(ii) are proved in Section~\ref{sec:fourth} from the identity~\eqref{eq:intro-factor} by elementary means: orthogonality, Weil's bound, Jacobsthal's evaluation of a quadratic character sum, the nine lines of the Cayley cubic surface, B\'ezout's theorem and Hasse's theorem. Theorem~\ref{thmD}(iii) is deduced from Theorem~\ref{thmE}. The statement $\Ggeom(\mathcal W_n)^0=\SL_{r(n)}$ in Theorem~\ref{thmE} combines four inputs:
\begin{itemize}
\item Theorem~\ref{thmB}, through Lemma~\ref{lem:lower};
\item Theorem~\ref{thmD}(ii) for $n\ge3$, through Lemma~\ref{lem:averages} and Corollary~\ref{cor:M4};
\item Theorem~\ref{thmC}(i), through Propositions~\ref{prop:unipotent} and~\ref{prop:det};
\item Larsen's alternative (Theorem~\ref{thm:larsen}).
\end{itemize}
Theorem~\ref{thmE}(ii), and Theorem~\ref{thmE}(i) up to the value of $\alpha$, combine this with Theorem~\ref{thmC}(iii), Proposition~\ref{prop:det} and the Artin--Schreier Lemma~\ref{lem:AS-lines}. The value of $\alpha$, and Theorem~\ref{thmE}(iii), rest in addition on Theorems~\ref{thmA} and~\ref{thmB} for $n-1$ variables (by induction on $n$), on Lemma~\ref{lem:R-global}, on the Fourier model~\eqref{eq:FT-model}, on Lemmas~\ref{lem:integrate-t} and~\ref{lem:t0}, on Theorem~\ref{thm:lisse}(vi), on Theorem~\ref{thm:det}, on Laumon's results on the Fourier transform \cite[Cor.~(1.2.2.3) and Th.~(1.3.2.3)]{Laumon}, on the formalism of vanishing cycles \cite[Exp.~XIII]{SGA7} together with Deligne's local acyclicity theorem \cite[Th.~finitude, Cor.~2.16]{SGA45}, and on the elementary identities of Lemma~\ref{lem:numerology}. Theorem~\ref{thmF} combines Theorem~\ref{thmE}(ii),(iii) with Deligne's equidistribution theorem \cite[Th.~3.5.3]{Deligne}. Theorem~\ref{thmD}(i) is not used in the proofs of the other theorems; through Corollary~\ref{cor:M2} it gives a second proof of the irreducibility of the geometric monodromy, valid for every $n\ge2$. The analysis of the Newton polytope in Section~\ref{sec:polytope} explains the numerology of Theorem~\ref{thmA}; it is used only in Remarks~\ref{rem:critical-values}, \ref{rem:AS} and~\ref{rem:collision}.

\subsection*{Relation to other work}
This article and the companion preprint \cite{EbadiMoments} grew out of the author's longer manuscript \cite{EbadiFull}, which studies the same family and develops analytic applications, including short-sum estimates and correlations with multiplicative functions. In the range $p>\max(2,n)$, the longer manuscript also obtains the identity component $\SL_{r(n)}$.

The present article treats every $n\ge2$ and every odd prime $p$. Its emphasis is on the compactly supported cohomology of the fibres $Y_h$, the exact multiplicities and values of the lower-weight eigenvalues, the construction of the associated lisse sheaves, and the detailed local monodromy along coordinate lines. These results lead to the determination of the geometric determinant, its arithmetic normalization, the full geometric and normalized arithmetic monodromy groups, and the corresponding equidistribution theorem.

The current companion preprint \cite{EbadiMoments} treats every $n\ge2$ in every characteristic. In odd characteristic, the two articles overlap in the cohomological description, the fourth-moment calculation through plane sections of the Cayley cubic surface, the determination of the identity component by Larsen's alternative, the full monodromy groups, and an equidistribution theorem. In characteristic two, the companion preprint establishes self-duality, identifies the orthogonal or symplectic monodromy group, and develops the associated moment and equidistribution results. It also gives explicit refinements of the moment calculations, including an exact fourth moment for $n=2$ and formulas for moment constants beyond the Gaussian range.

The proofs in the present article do not use either \cite{EbadiMoments} or \cite{EbadiFull} as an input. Conversely, for odd $p$ and $n\ge3$, the companion preprint uses the unipotence at the origin of coordinate lines proved in Theorem~\ref{thm:boundary}(ii) of the present article to complete its determination of the full monodromy groups; it proves the corresponding assertion for $n=2$ directly. Thus this use of the present article introduces no circular dependence between the two papers.

For a fixed $h\in(\Fq^\times)^n$, an upper bound of the right order already follows from~\eqref{eq:intro-factor}, Proposition~\ref{prop:fibration} and Deligne's theorem \cite[Th.~3.3.1]{Deligne}: the terms with $t\ne0$ in~\eqref{eq:intro-factor} add up to $\pm\Tr\big(\Frob_q\mid H^1_c(\Gm\otimes\Fqbar,\mathcal N_h)\big)$, where $\mathcal N_h$ is the sheaf of Proposition~\ref{prop:fibration}, and this group has dimension $2^n$ and weights at most $n+1$; hence $|\Sigma_n(h)|\le2^nq^{(n-1)/2}+q^{-1}$. For $q=p$ such bounds are also instances of the general estimates of Fouvry, Kowalski and Michel for sums of products of trace functions \cite{FKM2}. Theorem~\ref{thmA} determines the cohomology exactly: its dimension, the multiplicities of all of its weights, and the eigenvalues of the lower weights. The strategy behind the equality $\Ggeom(\mathcal W_n)^0=\SL_{r(n)}$, which combines a unipotent element coming from the boundary, the finiteness of the determinant, a fourth moment and Larsen's alternative, is the one developed by Katz \cite{Katz-Larsen,Katz-Conv}. The function whose Artin--Schreier sheaf is $\det\mathcal W_n$ is the sum of the critical values of the phase (Remark~\ref{rem:critical-values}). The identity holds on all of $S$, although the phase degenerates along the divisors where critical points escape to infinity (Remark~\ref{rem:collision}); it is proved through the local monodromy at infinity on coordinate lines, and not through the critical points. The constant $p^{(n+1)r(n)/2}$, which involves no Gauss sums, is computed instead at the finite boundary point, from the unramified case of Laumon's local Fourier transform. Theorem~\ref{thmF} is an instance of Deligne's equidistribution theorem \cite[Th.~3.5.3]{Deligne}, in the form given by Katz and Sarnak \cite[Th.~9.2.6]{KatzSarnak}.

\subsection*{Organization}
Section~\ref{sec:prelim} fixes notation and collects facts on Kloosterman sheaves, breaks, unipotent Jordan blocks, semisimplicity and averages of traces. Section~\ref{sec:polytope} analyses the Newton polytope. Theorem~\ref{thmA} is proved in Section~\ref{sec:fibre} and Theorem~\ref{thmB} in Section~\ref{sec:family-S}. Section~\ref{sec:boundary} treats the coordinate boundary, at $h_1=0$ and at $h_1=\infty$, and proves Theorem~\ref{thmC}. The moments are computed in Section~\ref{sec:fourth}. Section~\ref{sec:monodromy} proves Larsen's alternative in the form needed, determines the geometric monodromy group, the determinant and its constant, and the arithmetic monodromy group, and proves Theorems~\ref{thmD}(iii), \ref{thmE} and~\ref{thmF}; it also records conductor bounds for tensor constructions. Section~\ref{sec:numerics} reports numerical checks, which are not used in the proofs.

\section{Notation and preliminaries}\label{sec:prelim}

\subsection{Conventions}\label{sec:conventions}
Throughout, $p$ is an odd prime, $q$ is a power of $p$, and $\Fpbar$ is a fixed algebraic closure of $\Fp$. We fix a nontrivial additive character $\psi$ of $\Fp$ and an element $\lambda\in\Fp^\times$, and write $\langle h,u\rangle=\sum_ih_iu_i$. The letter $\ell$ denotes a prime different from $p$. Sheaves are constructible $\Qlb$-sheaves, and we fix a field isomorphism $\iota_{\mathbb C}\colon\Qlb\to\mathbb C$, with respect to which absolute values and weights are taken; through it we regard traces of Frobenius as complex numbers. An eigenvalue $\alpha$ of the geometric Frobenius of $\Fq$ has weight $w$ if $|\alpha|=q^{w/2}$. A sheaf is pure of weight $w$ if, at every closed point, all Frobenius eigenvalues on the stalk have weight $w$, and mixed if it is a successive extension of pure sheaves \cite[1.2]{Deligne}. The Tate twist $\Qlb(1)$ is pure of weight $-2$. For a separated morphism $f$ of finite type, \emph{proper base change} refers to the base change isomorphism for $Rf_!$ \cite[Exp.~XVII]{SGA4}; it holds without any properness assumption on $f$, because $Rf_!$ is defined through a compactification of $f$.

For a morphism $g\colon X\to\mathbb A^1$, $\Lpsi g$ denotes the inverse image of the Artin--Schreier sheaf. It is lisse of rank one and pure of weight $0$, with trace $\psi_q(g(x))$ at $x\in X(\Fq)$. For a character $\chi$ of finite order prime to $p$ of the tame fundamental group of $\Gm\otimes\Fpbar$, $\mathcal L_\chi$ is the corresponding Kummer sheaf.

Let $F$ be a lisse sheaf on a dense open subset of a smooth curve $C$ over $\Fpbar$, and let $x$ be a closed point of $C$. We write $F(x)$ for the representation of the inertia group $I_x$ on the geometric generic stalk, $\Iw_x\subset I_x$ for the wild inertia subgroup, and $\Swan_x(F)$ for the Swan conductor. Since $\Iw_x$ is a pro-$p$ group acting through a finite quotient, and $\ell\ne p$, its action is semisimple. We say that $F$ is \emph{tame} at $x$ if $\Iw_x$ acts trivially, and \emph{totally wild} at $x$ if $F(x)^{\Iw_x}=0$. The tame part $F(x)^{\Iw_x}$ is a direct summand of $F(x)$. All sheaves in this paper are defined over finite fields, so by Grothendieck's local monodromy theorem a topological generator $\gamma_x$ of tame inertia acts on the tame part quasi-unipotently, and the tame part decomposes canonically as $\bigoplus_\rho\mathcal L_\rho\otimes U_\rho$ over characters $\rho$ of finite order prime to $p$ of the tame inertia, with $U_\rho$ unipotent. We call $U_{\mathbf 1}$ the \emph{unipotent part} of $F(x)$, and write $\Up{k}$ for a unipotent Jordan block of size $k$.

When $F$ is lisse on all of $\Gm\otimes\Fpbar$, the Grothendieck--Ogg--Shafarevich formula \cite[Exp.~X]{SGA5} reads
\begin{equation}\label{eq:GOS}
\chi_c(\Gm\otimes\Fpbar,F)=-\Swan_0(F)-\Swan_\infty(F).
\end{equation}
Moreover $H^0_c(\Gm\otimes\Fpbar,F)=0$, and $H^2_c(\Gm\otimes\Fpbar,F)$ is the space of geometric coinvariants of $F$, twisted by $(-1)$. It vanishes as soon as $F$ is totally wild at $0$ or at $\infty$, since a geometrically trivial quotient is tame everywhere.

For $a\in\Fpbar^\times$ let $[\times a]$ be the automorphism $t\mapsto at$ of $\Gm$, and let $\iota$ be the inversion $t\mapsto t^{-1}$, which exchanges $0$ and $\infty$. Since $[\times a]$ fixes $0$ and $\infty$, the local monodromies of $F$ and $[\times a]^*F$ at each of these points have the same breaks and Swan conductors and the same unipotent Jordan structure. Indeed, the induced automorphism of the punctured disc respects the ramification filtration, and it induces the identity on the tame quotient of inertia, canonically isomorphic to $\prod_{\ell'\ne p}\mathbb Z_{\ell'}(1)$, because $(at)^{1/m}=a^{1/m}t^{1/m}$ with $a^{1/m}\in\Fpbar$.

\subsection{The sums and the phase}\label{sec:family}
For $h\in\Fq^n$ the sum $\Sigma_n(h)$ is defined by \eqref{eq:intro-Sigma}. On $\Gm^n\times\mathbb A^1_t$ we use the phase
\[
F_n(u,t;h)=f(u)+t\big(\lambda+\langle h,u\rangle\big),\qquad f(u)=\sum_{i=1}^nu_i^{-1},
\]
and for $h\in(\Fq^\times)^n$ we write
\[
Y_h=\{u\in\Gm^n:\langle h,u\rangle=-\lambda\},\quad X_h=\Gm^n\times\mathbb A^1_t,\quad X_h^\circ=\Gm^n\times\Gmn{t},\quad Z_h=\Gm^n\times\{0\}.
\]
Since every $h_i$ is nonzero, $Y_h$ is the complement of $n$ hyperplanes in an affine hyperplane of $\mathbb A^n$; it is a smooth, affine, geometrically irreducible variety of dimension $n-1$.

\begin{lemma}\label{lem:factor}
For every $h\in\Fq^n$ the identity \eqref{eq:intro-factor} holds. If exactly $a$ coordinates of $h$ are nonzero and $h'\in(\Fq^\times)^a$ is the vector they form, then $\Sigma_n(h)=0$ if $a=0$, $|\Sigma_n(h)|=1$ if $a=1$, and $\Sigma_n(h)=(-1)^{n-a}\Sigma_a(h')$ if $a\ge2$.
\end{lemma}

\begin{proof}
Insert $\mathbf 1_{\langle h,u\rangle=-\lambda}=q^{-1}\sum_{t\in\Fq}\psi_q\big(t(\lambda+\langle h,u\rangle)\big)$ into \eqref{eq:intro-Sigma}. The sum over $u$ factors, and the factor indexed by $i$ is $\sum_{u_i\in\Fq^\times}\psi_q(u_i^{-1}+th_iu_i)=K(th_i)$. If $h_i=0$, the variable $u_i$ does not occur in the constraint, and summing over it contributes the factor $\sum_{x\in\Fq^\times}\psi_q(x^{-1})=-1$. If $a=0$ the constraint reads $0=-\lambda$, which is impossible. If $a=1$ it determines the remaining variable, and the sum has a single term.
\end{proof}

\begin{lemma}[Integration over $t$]\label{lem:integrate-t}
Let $V$ be a scheme of finite type over a field of characteristic $p$, $g\colon V\to\mathbb A^1$ a morphism, $i\colon Y=g^{-1}(0)\hookrightarrow V$ the inclusion, and $q_V\colon V\times\mathbb A^1_t\to V$ the projection. Then $Rq_{V!}\Lpsi{tg}\cong i_*\Qlb(-1)[-2]$. Hence, for every constructible complex $K$ on $V$,
\[
Rq_{V!}\big(q_V^*K\otimes\Lpsi{tg}\big)\cong i_*(K|_Y)(-1)[-2],
\]
and the formation of these isomorphisms commutes with base change $V'\to V$.
\end{lemma}

\begin{proof}
Let $p_s\colon\mathbb A^1_s\times\mathbb A^1_t\to\mathbb A^1_s$ be the projection. For a geometric point $s\ne0$, the sheaf $\Lpsi{st}$ on $\mathbb A^1_t$ is lisse of rank one and geometrically nontrivial, with Swan conductor $1$ at $\infty$. The Grothendieck--Ogg--Shafarevich formula gives Euler characteristic $1-1=0$; as $H^0_c$ and $H^2_c$ vanish, all its compactly supported cohomology vanishes. For $s=0$ one obtains $H^*_c(\mathbb A^1,\Qlb)=\Qlb(-1)[-2]$. By proper base change for $Rp_{s!}$, the stalk of $Rp_{s!}\Lpsi{st}$ at a geometric point $s$ is $R\Gamma_c(\mathbb A^1_t\otimes k(s),\Lpsi{st})$. So the cohomology sheaves of $Rp_{s!}\Lpsi{st}$ are supported at $s=0$, with stalk $\Qlb(-1)$ in degree $2$, and $Rp_{s!}\Lpsi{st}\cong i_{0*}\Qlb(-1)[-2]$. Since $\Lpsi{tg}=(g\times\mathrm{id})^*\Lpsi{st}$, proper base change along $g$ gives the first isomorphism, and the projection formula gives the second.
\end{proof}

\subsection{Kloosterman sheaves}\label{sec:Kl}
Let $\sigma\colon\Gm^2\to\Gm$ be the map $(x,y)\mapsto xy$, and let $\Kl=R^1\sigma_!\Lpsi{x+y}$ over $\Fp$. We use the following facts.
\begin{itemize}
\item[(K1)] $R^i\sigma_!\Lpsi{x+y}=0$ for $i\ne1$. The sheaf $\Kl$ is lisse of rank $2$ on $\Gm$ and pure of weight $1$, and $\Tr(\Frob_a\mid\Kl)=-K(a)$ for $a\in\Fq^\times$ \cite[Sommes trig., Th.~7.4 and~7.8]{SGA45}, \cite[4.1.1]{Katz-Gauss}.
\item[(K2)] $\Kl$ is tame at $0$, and $\Kl(0)\cong\Up{2}$.
\item[(K3)] $\Kl$ is totally wild at $\infty$, with $\Swan_\infty(\Kl)=1$ \cite[Sommes trig., Th.~7.4 and~7.8]{SGA45}, \cite[4.1.1]{Katz-Gauss}. Both breaks of $\Kl$ at $\infty$ lie in the open interval $(0,1)$.
\item[(K4)] If $\beta\colon\Gmn{u}\times\Gmn{b}\to\Gmn{b}$ is the projection, then $R\beta_!\Lpsi{u^{-1}+bu}\cong\Kl[-1]$.
\item[(K5)] Let $j\colon\Gm\hookrightarrow\mathbb A^1$ be the inclusion. For every power $q$ of $p$, the geometric Frobenius of $\Fq$ acts trivially on the one-dimensional stalk $(j_*\Kl)_0=\Kl(0)^{I_0}$.
\end{itemize}
For (K2), recall that $\Kl$ is tame at $0$, with unipotent local monodromy consisting of a single Jordan block \cite[Sommes trig., Th.~7.4 and~7.8]{SGA45}. For completeness we note why this unipotent automorphism of a two-dimensional space is not the identity, hence is a single Jordan block of size~$2$: if $I_0$ acted trivially, $\Kl$ would extend to a lisse sheaf on $\mathbb A^1$, and excision together with~\eqref{eq:GOS} would give $\chi_c(\mathbb A^1\otimes\Fpbar,\Kl)=\chi_c(\Gm\otimes\Fpbar,\Kl)+2=-\Swan_\infty(\Kl)+2=1$. But $H^0_c(\mathbb A^1\otimes\Fpbar,\Kl)=0$, and $H^2_c(\mathbb A^1\otimes\Fpbar,\Kl)$, the space of geometric coinvariants twisted by $(-1)$, vanishes because $\Kl$ is totally wild at $\infty$; so $\chi_c(\mathbb A^1\otimes\Fpbar,\Kl)\le0$, a contradiction. For (K3), the two breaks are positive because $\Kl$ is totally wild at $\infty$, and their sum is $\Swan_\infty(\Kl)=1$. For (K4), note that $(u,b)\mapsto(u^{-1},bu)$ is an automorphism of $\Gm^2$, with inverse $(x,y)\mapsto(x^{-1},xy)$. It transforms $\beta$ into $\sigma$ and the phase $u^{-1}+bu$ into $x+y$. The trace formula in (K1) follows from the Grothendieck--Lefschetz trace formula \cite[Rapport]{SGA45}, which we use throughout, and $\sum_{xy=a}\psi_q(x+y)=\sum_{z\in\Fq^\times}\psi_q(z^{-1}+az)=K(a)$.

For (K5), consider $j_*\Kl$ on $\mathbb A^1$ over $\Fq$; by (K2) its stalk at $0$ is a line. Its group $H^0_c$ vanishes, since $j_*\Kl$ has no nonzero section supported on a finite set. Excision for $\{0\}\subset\mathbb A^1$ shows that $H^2_c(\mathbb A^1\otimes\Fqbar,j_*\Kl)=H^2_c(\Gm\otimes\Fqbar,\Kl)$, which vanishes because $\Kl$ is totally wild at $\infty$. The Euler characteristic is $\chi_c(\Gm\otimes\Fqbar,\Kl)+1=-\Swan_\infty(\Kl)+1=0$, by~\eqref{eq:GOS}. So all compactly supported cohomology of $j_*\Kl$ vanishes, and the Grothendieck--Lefschetz formula gives
\[
0=\Tr\big(\Frob_q\mid(j_*\Kl)_0\big)-\sum_{a\in\Fq^\times}K(a).
\]
Finally $\sum_{a\in\Fq}K(a)=\sum_{x\in\Fq^\times}\psi_q(x^{-1})\sum_{a\in\Fq}\psi_q(ax)=0$ and $K(0)=\sum_{x\in\Fq^\times}\psi_q(x^{-1})=-1$, so $\sum_{a\ne0}K(a)=1$, and the eigenvalue of $\Frob_q$ on $(j_*\Kl)_0$ is~$1$.

\subsection{Breaks of tensor products}

\begin{lemma}\label{lem:breaks}
Let $x$ be a closed point of a smooth curve over $\Fpbar$, and let $V,V'$ be finite-dimensional continuous representations of $I_x$.
\begin{enumerate}[(i)]
\item Every break of $V\otimes V'$ is at most the largest of the breaks of $V$ and $V'$.
\item If all breaks of $V$ are $<s'$ and $V'$ is isoclinic of break $s'>0$, then $V\otimes V'$ is isoclinic of break $s'$. In particular $\Swan_x(V\otimes V')=s'\dim V\dim V'$.
\end{enumerate}
\end{lemma}

\begin{proof}
Let $(I^{(r)})_{r\ge0}$ be the upper numbering filtration of $I_x$, and let $I^{(r+)}$ be the closure of $\bigcup_{r'>r}I^{(r')}$. A representation $W$ has all breaks $\le r$ if and only if $I^{(r+)}$ acts trivially on $W$, and it is isoclinic of break $s>0$ if and only if $I^{(s+)}$ acts trivially on $W$ and $W^{I^{(r+)}}=0$ for $0\le r<s$ \cite[1.1]{Katz-Gauss}. Statement~(i) is clear, because $I^{(r+)}$ acts trivially on $V\otimes V'$ as soon as it acts trivially on both factors. For (ii), let $s<s'$ be the largest break of $V$. By (i), $I^{(s'+)}$ acts trivially on $V\otimes V'$. If $s\le r<s'$, then $I^{(r+)}$ acts trivially on $V$, so $(V\otimes V')^{I^{(r+)}}=V\otimes V'^{I^{(r+)}}=0$. If $0\le r<s$, then $I^{(r+)}\supseteq I^{(s+)}$, and therefore $(V\otimes V')^{I^{(r+)}}\subseteq(V\otimes V')^{I^{(s+)}}=0$ by the previous case. Hence $V\otimes V'$ is isoclinic of break $s'$, and its Swan conductor is $s'\dim(V\otimes V')$.
\end{proof}

\subsection{Unipotent Jordan blocks}

\begin{lemma}\label{lem:CG}
Let $g_1,\dots,g_n$ be unipotent automorphisms of two-dimensional vector spaces over a field of characteristic $0$, each consisting of a single Jordan block. Then $g_1\otimes\cdots\otimes g_n$ is unipotent, and
\[
\Up{2}^{\otimes n}\cong\bigoplus_{k=0}^{\lfloor n/2\rfloor}m_{n,k}\,\Up{n-2k+1}.
\]
In particular $g_1\otimes\cdots\otimes g_n$ has $\binom n{\lfloor n/2\rfloor}$ Jordan blocks, which is the dimension of its fixed space. The number of blocks of size $1$ is the Catalan number $m_{n,n/2}$ if $n$ is even, and $0$ if $n$ is odd.
\end{lemma}

\begin{proof}
For a unipotent automorphism $g$, the logarithm $N=\log g$ is a nilpotent endomorphism with the same Jordan type: $g-1=N\,e(N)$, where $e(N)=1+N/2+\cdots$ is invertible and commutes with $N$. The factors $1\otimes\cdots\otimes g_i\otimes\cdots\otimes1$ commute, so $\log(g_1\otimes\cdots\otimes g_n)=\sum_iN_i$ with $N_i=1\otimes\cdots\otimes\log g_i\otimes\cdots\otimes1$. In bases in which every $\log g_i$ is $\left(\begin{smallmatrix}0&1\\0&0\end{smallmatrix}\right)$, the endomorphism $\sum_iN_i$ is the action of $e=\left(\begin{smallmatrix}0&1\\0&0\end{smallmatrix}\right)\in\mathfrak{sl}_2$ on the $n$-th tensor power of the standard representation $\mathrm{St}$. On an irreducible $\mathfrak{sl}_2$-module of dimension $d$ the element $e$ is a single Jordan block of size $d$. So the Jordan type is the multiset of dimensions of the irreducible constituents of $\mathrm{St}^{\otimes n}$. By the Clebsch--Gordan rule, $\mathrm{St}\otimes\Sym^{d}\cong\Sym^{d+1}\oplus\Sym^{d-1}$ for $d\ge1$ and $\mathrm{St}\otimes\Sym^0=\mathrm{St}$. Hence the multiplicity $\mu_{n,k}$ of $\Sym^{n-2k}$ in $\mathrm{St}^{\otimes n}$ satisfies $\mu_{1,0}=1$ and $\mu_{n,k}=\mu_{n-1,k}+\mu_{n-1,k-1}$ for $0\le k\le\lfloor n/2\rfloor$, where $\mu_{n-1,-1}=0$ and, for even $n$, $\mu_{n-1,n/2}=0$. The numbers $m_{n,k}$ satisfy the same recursion and initial value. Indeed, by Pascal's rule both $m_{n-1,k}+m_{n-1,k-1}$ and $m_{n,k}$ equal $\binom{n-1}k-\binom{n-1}{k-2}$, and $m_{n-1,n/2}=0$ for even $n$. Hence $\mu_{n,k}=m_{n,k}$. The number of blocks telescopes to $\binom n{\lfloor n/2\rfloor}$, and the blocks of size $1$ correspond to $k=n/2$.
\end{proof}

We shall also use the kernel of $\sum_iN_i$ with its grading. In the notation of the proof, let $e_1,e_2$ be the basis of $\mathrm{St}$ with $e\cdot e_1=0$ and $e\cdot e_2=e_1$. For $J=(j_1,\dots,j_n)\in\{1,2\}^n$ put $e_J=e_{j_1}\otimes\cdots\otimes e_{j_n}$, and let $|J|$ be the number of indices $i$ with $j_i=2$. The operator $\sum_iN_i$ maps $e_J$ to a combination of the $e_{J'}$ with $|J'|=|J|-1$, so its kernel $E$ is graded: $E=\bigoplus_kE_k$, where $E_k$ is spanned by the elements of $E$ that are combinations of the $e_J$ with $|J|=k$. The vector $e_J$ has weight $n-2|J|$ for the diagonal element $\left(\begin{smallmatrix}1&0\\0&-1\end{smallmatrix}\right)$ of $\mathfrak{sl}_2$, and $E$ is the space of highest-weight vectors of $\mathrm{St}^{\otimes n}$. Each irreducible constituent $\Sym^{n-2k}$ contains a single line of highest-weight vectors, of weight $n-2k$. Hence
\begin{equation}\label{eq:Ek}
\dim E_k=m_{n,k}\quad(0\le k\le\lfloor n/2\rfloor),\qquad E_k=0\quad(k>n/2).
\end{equation}

\begin{lemma}\label{lem:frob-block}
Let $x$ be an $\Fq$-rational point of a smooth curve over $\Fq$, let $D_x$ be a decomposition group at $x$, with inertia group $I_x$ and wild inertia group $\Iw_x$, and let $\gamma\in I_x$ be an element whose image topologically generates $I_x/\Iw_x$. Let $\phi\in D_x$ induce the geometric Frobenius $z\mapsto z^{1/q}$ on $\Fqbar$. Let $V$ be a two-dimensional continuous representation of $D_x$ on which $\Iw_x$ acts trivially and $\gamma$ acts by a unipotent automorphism $T\ne1$, and suppose that $\phi$ acts trivially on the line $V^{I_x}$. Put $N=\log T$. Then there is a basis $e_1,e_2$ of $V$ with $Ne_1=0$, $Ne_2=e_1$, $\phi e_1=e_1$ and $\phi e_2=qe_2$.
\end{lemma}

\begin{proof}
Let $\varpi$ be a uniformizer at $x$. The tame quotient $I_x/\Iw_x$ is identified with $\varprojlim_{p\nmid m}\mu_m(\Fqbar)$ by $\sigma\mapsto\big(\sigma(\varpi^{1/m})/\varpi^{1/m}\big)_m$; this does not depend on the choice of the roots $\varpi^{1/m}$. Let $\sigma\in I_x$, let $m$ be prime to $p$, and put $\zeta=\sigma(\varpi^{1/m})/\varpi^{1/m}$. Write $\phi^{-1}(\varpi^{1/m})=\xi\varpi^{1/m}$ with $\xi\in\mu_m(\Fqbar)$, so that $\phi(\varpi^{1/m})=\phi(\xi)^{-1}\varpi^{1/m}$. Since $\sigma$ acts trivially on $\Fqbar$,
\[
\phi\sigma\phi^{-1}(\varpi^{1/m})=\phi(\xi\zeta\varpi^{1/m})=\phi(\zeta)\,\varpi^{1/m}=\zeta^{1/q}\varpi^{1/m}.
\]
Thus conjugation by $\phi$ induces on $I_x/\Iw_x$ the inverse of the $q$-th power map, which is multiplication by $q^{-1}$ in the procyclic group $I_x/\Iw_x\cong\prod_{\ell'\ne p}\mathbb Z_{\ell'}$. The action of $I_x$ on $V$ factors through $I_x/\Iw_x$, and the closure of the group generated by $T$ is $\exp(\mathbb Z_\ell N)$; so $\gamma^a$ acts by $\exp(a_\ell N)$, where $a_\ell$ is the component of $a\in\prod_{\ell'\ne p}\mathbb Z_{\ell'}$ in $\mathbb Z_\ell$. Hence $F=\phi|_V$ satisfies $FTF^{-1}=\exp(q^{-1}N)$, and taking logarithms, $FNF^{-1}=q^{-1}N$, that is, $NF=qFN$.

As $T\ne1$ is unipotent on a space of dimension two, $N=T-1$ is nonzero with $N^2=0$, and $\ker N=V^{I_x}$ is a line, spanned by a vector $e_1$ with $Fe_1=e_1$. Choose $v$ with $Nv=e_1$. Then $N(Fv)=qFNv=qe_1=N(qv)$, so $Fv=qv+ce_1$ for some scalar $c$. Since $q\ne1$, the vector $e_2=v+(q-1)^{-1}ce_1$ satisfies $Fe_2=qe_2$ and $Ne_2=e_1$.
\end{proof}

\subsection{Semisimplicity}

\begin{lemma}\label{lem:semisimple}
Let $k$ be an algebraically closed field of characteristic $0$.
\begin{enumerate}[(i)]
\item If $H$ is a linear algebraic group over $k$ with $H^0$ reductive, every finite-dimensional representation of $H$ is semisimple.
\item If $H\subseteq\GL(V)$ is a Zariski-closed subgroup and $V$ is a semisimple representation of $H$, then $H^0$ is reductive.
\end{enumerate}
\end{lemma}

\begin{proof}
(i) In characteristic $0$, every finite-dimensional representation of a connected reductive group is semisimple. Let $U$ be a finite-dimensional representation of $H$, let $W\subseteq U$ be an $H$-stable subspace, and let $\varpi\colon U\to W$ be an $H^0$-equivariant projection. For $g\in H$ the map $g\varpi g^{-1}$ is again a projection onto $W$, and it depends only on the coset $gH^0$. The average of the $g\varpi g^{-1}$ over the finite group $H/H^0$ is therefore an $H$-equivariant projection onto $W$, and its kernel is an $H$-stable complement of $W$.

(ii) The unipotent radical $R_u(H^0)$ is a normal unipotent subgroup of $H$. It has a nonzero fixed vector in each simple summand of $V$, and these fixed vectors form an $H$-stable subspace of the summand, hence the whole summand. So $R_u(H^0)$ acts trivially on $V$, and it is trivial because $V$ is a faithful representation of $H$.
\end{proof}

\subsection{Averages of traces}

\begin{lemma}\label{lem:averages}
Let $T$ be a smooth geometrically connected variety of dimension $d$ over a finite field $\mathbb F_{q_0}$, and let $\mathcal E$ be a lisse sheaf on $T$, pointwise pure of weight $0$, with geometric generic stalk $V_{\mathcal E}$. Let $\pi^{\mathrm g}=\pi_1(T\otimes\overline{\mathbb F}_{q_0})$, and let $M$ be the dimension of the space of invariants $V_{\mathcal E}^{\pi^{\mathrm g}}$. Then $\Frob_{q_0}$ acts on $V_{\mathcal E}^{\pi^{\mathrm g}}$ with eigenvalues $\beta_1,\dots,\beta_M$ of absolute value $1$, and for $q=q_0^k$
\[
q^{-d}\sum_{x\in T(\Fq)}\Tr(\Frob_x\mid\mathcal E)=\sum_{m=1}^M\beta_m^k+O\big(q^{-1/2}\big),
\]
where the implied constant depends only on $\mathcal E$. Consequently, if the left-hand side converges to a limit $L$ as $k\to\infty$, then $L=M$.
\end{lemma}

\begin{proof}
By the Grothendieck--Lefschetz trace formula \cite[Rapport]{SGA45}, the left-hand side is $q^{-d}\sum_i(-1)^i\Tr\big(\Frob_q\mid H^i_c(T\otimes\Fqbar,\mathcal E)\big)$. Here $H^i_c(T\otimes\Fqbar,\mathcal E)=0$ for $i>2d$. For $i<2d$ the group $H^i_c$ is mixed of weights at most $i\le2d-1$ \cite[Th.~3.3.1]{Deligne}, so every eigenvalue of $\Frob_q$ on it has absolute value at most $q^{d-1/2}$; its dimension does not depend on $k$. Hence these terms contribute at most $q^{-1/2}\sum_{i<2d}\dim H^i_c(T\otimes\Fqbar,\mathcal E)$. Moreover $H^{2d}_c(T\otimes\Fqbar,\mathcal E)\cong(V_{\mathcal E})_{\pi^{\mathrm g}}(-d)$. The fundamental group $\pi_1(T)$ acts on $V_{\mathcal E}$, and its normal subgroup $\pi^{\mathrm g}$ acts semisimply, because $\mathcal E$ is pure and hence geometrically semisimple \cite[3.4.1(iii)]{Deligne}. So the natural map $V_{\mathcal E}^{\pi^{\mathrm g}}\to(V_{\mathcal E})_{\pi^{\mathrm g}}$ is an isomorphism, compatible with the action of $\pi_1(T)/\pi^{\mathrm g}=\mathrm{Gal}(\Fqbar/\mathbb F_{q_0})$. By Poincar\'e duality, $H^{2d}_c(T\otimes\Fqbar,\mathcal E)$ is dual to $H^0(T\otimes\Fqbar,\mathcal E^\vee)(d)$, and $H^0(T\otimes\Fqbar,\mathcal E^\vee)$ embeds into the stalk of $\mathcal E^\vee$ at a geometric point above a closed point, compatibly with the Frobenius of that point. Hence $H^{2d}_c$ is pure of weight $2d$, and the $\beta_m$ have absolute value $1$. This proves the formula. If its left-hand side converges to $L$, then $\sum_m\beta_m^k\to L$. By Dirichlet's simultaneous approximation theorem there are infinitely many $k$ for which every $\beta_m^k$ is arbitrarily close to $1$; along them $\sum_m\beta_m^k\to M$, and hence $L=M$.
\end{proof}

\section{The Newton polytope}\label{sec:polytope}

In this section $n\ge1$. We work in $\mathbb R^{n+1}=\mathbb R^n\times\mathbb R$, with coordinates $(x,y)$ and standard basis $e_1,\dots,e_{n+1}$, and put
\[
\mathbf a_i=-e_i,\qquad\mathbf b_i=e_i+e_{n+1}\quad(1\le i\le n),\qquad\mathbf c=e_{n+1}.
\]
These are the exponents of the monomials $u_i^{-1}$, $h_iu_it$ and $\lambda t$ of $F_n$ in the torus coordinates $(u_1,\dots,u_n,t)$. Let
\[
\begin{aligned}
P_n&=\conv\{\mathbf a_1,\dots,\mathbf a_n,\mathbf b_1,\dots,\mathbf b_n,\mathbf c\},\qquad\Delta_n=\conv(P_n\cup\{0\}),\\
\sigma^*&=\conv\{\mathbf a_1,\dots,\mathbf a_n,\mathbf b_1,\dots,\mathbf b_n\}.
\end{aligned}
\]
Thus $P_n$ is the Newton polytope of $F_n$, and $\Delta_n$ is its Newton polyhedron at infinity in the sense of Kouchnirenko and of Adolphson--Sperber \cite{Kouchnirenko,AS}. For a nonzero linear form $\varphi(x,y)=\sum_i\alpha_ix_i+\beta y$ let $M_\varphi=\max_{P_n}\varphi$, and let $P_n^\varphi=\{v\in P_n:\varphi(v)=M_\varphi\}$ be the face it cuts out. As $P_n$ is the convex hull of the $2n+1$ points above, $M_\varphi$ is attained among them, and $P_n^\varphi$ is the convex hull of those points at which $\varphi=M_\varphi$. Put $\xi_i=\varphi(\mathbf a_i)=-\alpha_i$; then $\varphi(\mathbf b_i)=\beta-\xi_i$ and $\varphi(\mathbf c)=\beta$. Finally let $\varphi^*(x,y)=\sum_ix_i-2y$.

\begin{theorem}[Faces of $P_n$]\label{thm:faces}
Let $\varphi\ne0$ and $M=M_\varphi$.
\begin{enumerate}[(a)]
\item $\beta-M\le\xi_i\le M$ for every $i$, $\beta\le M$, and $\beta\le2M$.
\item If $\beta=2M$, then $M<0$, $\varphi=-M\varphi^*$, and $P_n^\varphi=\sigma^*$.
\item If $\beta\ne2M$, each index $i$ is of exactly one of the types
\[
\mathrm A:\ \xi_i=M,\qquad\mathrm B:\ \xi_i=\beta-M,\qquad\mathrm N:\ \beta-M<\xi_i<M.
\]
The face $P_n^\varphi$ is the convex hull of the $\mathbf a_i$ with $i$ of type $\mathrm A$, the $\mathbf b_i$ with $i$ of type $\mathrm B$, and, if and only if $\beta=M$, the point $\mathbf c$. If $\beta<M$, some index is not of type $\mathrm N$.
\item Conversely, let $J_{\mathrm A},J_{\mathrm B}\subseteq\{1,\dots,n\}$ be disjoint and $\eta\in\{0,1\}$, with $J_{\mathrm A}\cup J_{\mathrm B}\ne\varnothing$ if $\eta=0$. Let $\tau(J_{\mathrm A},J_{\mathrm B},\eta)$ be the convex hull of $\{\mathbf a_i\}_{i\in J_{\mathrm A}}\cup\{\mathbf b_j\}_{j\in J_{\mathrm B}}$, together with $\mathbf c$ if $\eta=1$. Then there is a $\varphi$ with $M_\varphi=1$ and $P_n^\varphi=\tau(J_{\mathrm A},J_{\mathrm B},\eta)$.
\item The vertex set of $\tau(J_{\mathrm A},J_{\mathrm B},\eta)$ is linearly independent. In particular $\tau(J_{\mathrm A},J_{\mathrm B},\eta)$ is a simplex of dimension $|J_{\mathrm A}|+|J_{\mathrm B}|+\eta-1$.
\end{enumerate}
\end{theorem}

\begin{proof}
(a) The values $\xi_i$, $\beta-\xi_i$ and $\beta$ are at most $M$; adding the first two for $i=1$ gives $\beta\le2M$.

(b) If $\beta=2M$, then $\beta-M=M$, so (a) forces $\xi_i=M$ for all $i$. Thus $\alpha=-M(1,\dots,1)$ and $\beta=2M$, that is, $\varphi=-M\varphi^*$. Since $\varphi\ne0$ we have $M\ne0$, and $2M=\beta\le M$ gives $M<0$. Every $\mathbf a_i$ and every $\mathbf b_i$ takes the value $M$, whereas $\varphi(\mathbf c)=2M<M$; hence $P_n^\varphi=\sigma^*$.

(c) If $\beta\ne2M$, then $\beta<2M$ by (a), so $\beta-M<M$. By (a), $\xi_i$ lies in $[\beta-M,M]$ and equals at most one of its endpoints. Since $\mathbf a_i\in P_n^\varphi$ if and only if $\xi_i=M$, $\mathbf b_i\in P_n^\varphi$ if and only if $\xi_i=\beta-M$, and $\mathbf c\in P_n^\varphi$ if and only if $\beta=M$, the description of the face follows. If $\beta<M$ and every index were of type $\mathrm N$, no generator would attain the value $M$.

(d) For $\eta=0$ take $\beta=0$, and $\alpha_i=-1$, $1$ or $0$ according as $i\in J_{\mathrm A}$, $i\in J_{\mathrm B}$ or $i\notin J_{\mathrm A}\cup J_{\mathrm B}$. Then $\varphi(\mathbf a_i)$ is $1$, $-1$ or $0$, $\varphi(\mathbf b_i)$ is $-1$, $1$ or $0$, and $\varphi(\mathbf c)=0$. Since $J_{\mathrm A}\cup J_{\mathrm B}\ne\varnothing$, the maximum is $1$, attained exactly at the vertices of $\tau(J_{\mathrm A},J_{\mathrm B},0)$. For $\eta=1$ take $\beta=1$, and $\alpha_i=-1$, $0$ or $-\frac12$ according as $i\in J_{\mathrm A}$, $i\in J_{\mathrm B}$ or $i\notin J_{\mathrm A}\cup J_{\mathrm B}$. Then $\varphi(\mathbf a_i)$ is $1$, $0$ or $\frac12$, $\varphi(\mathbf b_i)$ is $0$, $1$ or $\frac12$, and $\varphi(\mathbf c)=1$, so the maximum $1$ is attained exactly at the vertices of $\tau(J_{\mathrm A},J_{\mathrm B},1)$.

(e) In a linear relation $\sum_{i\in J_{\mathrm A}}c_i\mathbf a_i+\sum_{j\in J_{\mathrm B}}d_j\mathbf b_j+c_0\mathbf c=0$, the coordinate $x_i$ with $i\in J_{\mathrm A}$ receives only the contribution $-c_i$, and the coordinate $x_j$ with $j\in J_{\mathrm B}$ only the contribution $d_j$. Hence all $c_i$ and $d_j$ vanish, and then $c_0=0$. Linearly independent points are affinely independent.
\end{proof}

\begin{proposition}\label{prop:structure}
\begin{enumerate}[(i)]
\item $\dim P_n=n+1$, and $0\notin P_n$.
\item $\sigma^*$ is a facet of $P_n$; it is cut out by $\varphi^*$, which takes the value $-1$ on it. With $\mathbf m=\frac12\mathbf c$ and $w_i=e_i+\frac12e_{n+1}$ one has $\sigma^*=\mathbf m+\conv\{\pm w_1,\dots,\pm w_n\}$, and $w_1,\dots,w_n$ are linearly independent. Thus $\sigma^*$ is an $n$-dimensional cross-polytope; for $n\ge2$ it has $2n>n+1$ vertices and is not a simplex.
\item $P_n=\conv(\sigma^*\cup\{\mathbf c\})$ and $\Delta_n=P_n\cup\conv(\sigma^*\cup\{0\})$. These two pyramids over $\sigma^*$ have disjoint interiors and equal volumes, so $\Vol(\Delta_n)=2\Vol(P_n)$.
\item The proper faces of $P_n$ are $\sigma^*$ and the pairwise distinct simplices $\tau(J_{\mathrm A},J_{\mathrm B},\eta)$ of Theorem~\ref{thm:faces}(d). Consequently, for $0\le j\le n$,
\[
f_j(P_n)=\binom n{j+1}2^{j+1}+\binom nj2^j+[\,j=n\,],
\]
where $[\,j=n\,]$ is $1$ if $j=n$ and $0$ otherwise. For $n=4$ the $f$-vector $(f_0,\dots,f_4)$ of the $5$-dimensional polytope $P_4$ is $(9,32,56,48,17)$.
\item The faces of $\Delta_n$ not containing $0$ are exactly the proper faces of $P_n$ other than $\sigma^*$.
\end{enumerate}
\end{proposition}

\begin{proof}
(i) The differences $\mathbf b_i-\mathbf c=e_i$ ($1\le i\le n$) and $(\mathbf c-\mathbf a_1)-(\mathbf b_1-\mathbf c)=e_{n+1}$ span $\mathbb R^{n+1}$, so the affine hull of $P_n$ is $\mathbb R^{n+1}$. The coordinate $y$ vanishes at the $\mathbf a_i$ and equals $1$ at the $\mathbf b_i$ and at $\mathbf c$. If $0$ were a convex combination of the generators, the coefficients of the $\mathbf b_i$ and of $\mathbf c$ would vanish, and $0$ would be a convex combination of $-e_1,\dots,-e_n$, which is absurd.

(ii) For $\varphi^*$ we have $\xi_i=-1$ and $\beta=-2$; so $\max_{P_n}\varphi^*=-1$ and, by Theorem~\ref{thm:faces}(b), $\varphi^*$ cuts out $\sigma^*$. All segments $[\mathbf a_i,\mathbf b_i]$ have midpoint $\mathbf m$, and $\mathbf b_i-\mathbf m=w_i=-(\mathbf a_i-\mathbf m)$. The $w_i$ are linearly independent, since their first $n$ coordinates form the identity matrix. Hence the affine hull of $\sigma^*$ is $\mathbf m+\sum_i\mathbb Rw_i$, of dimension $n=\dim P_n-1$, and $\sigma^*$ is the image of the standard cross-polytope $\conv\{\pm e_1,\dots,\pm e_n\}\subset\mathbb R^n$ under the affine embedding $z\mapsto\mathbf m+\sum_iz_iw_i$.

(iii) Every generator of $P_n$ other than $\mathbf c$ lies in $\sigma^*$, so $P_n=\conv(\sigma^*\cup\{\mathbf c\})$. Since $\mathbf c=2\mathbf m$, a point of $\conv(\sigma^*\cup\{0,\mathbf c\})$ can be written $z=(1-a-b)s+b\,\mathbf c$ with $s\in\sigma^*$, $a,b\ge0$ and $a+b\le1$. If $a\ge b$, then $z=(1-a-b)s+2b\,\mathbf m+(a-b)\cdot0$ is a convex combination of points of $\sigma^*$ and of $0$. If $b\ge a$, then $z=(1-a-b)s+2a\,\mathbf m+(b-a)\mathbf c$ is a convex combination of points of $\sigma^*$ and of $\mathbf c$. Hence $\Delta_n$ is the union of the two pyramids. They lie on the two sides $\{\varphi^*\le-1\}$ and $\{\varphi^*\ge-1\}$ of the hyperplane spanned by $\sigma^*$, so their interiors are disjoint. Since $\varphi^*(\mathbf c)=-2$ and $\varphi^*(0)=0$, the two apices are at the same distance from that hyperplane, and the pyramids have the same volume.

(iv) By Theorem~\ref{thm:faces}(b),(c), every proper face is $\sigma^*$ or of the form $\tau(J_{\mathrm A},J_{\mathrm B},\eta)$, and by Theorem~\ref{thm:faces}(d) each such $\tau$ is a face. Distinct data $(J_{\mathrm A},J_{\mathrm B},\eta)$ give distinct vertex sets, and $\sigma^*$ contains both $\mathbf a_1$ and $\mathbf b_1$, which no $\tau$ does. By Theorem~\ref{thm:faces}(e) the simplex $\tau(J_{\mathrm A},J_{\mathrm B},\eta)$ has dimension $j=|J_{\mathrm A}|+|J_{\mathrm B}|+\eta-1$. For $\eta=0$ there are $\binom n{j+1}2^{j+1}$ choices with $|J_{\mathrm A}|+|J_{\mathrm B}|=j+1$ (choose the indices, then a type for each), and for $\eta=1$ there are $\binom nj2^j$; the facet $\sigma^*$ contributes $[\,j=n\,]$.

(v) Let $\tau$ be a face of $\Delta_n$ with $0\notin\tau$, cut out by a linear form $\varphi$ with maximum $M'$ on $\Delta_n$. Then $M'>\varphi(0)=0$, and since $\Delta_n=\conv(P_n\cup\{0\})$, we get $M'=M_\varphi$ and $\tau=P_n^\varphi$. This is a proper face of $P_n$ (as $\varphi$ is not constant on $P_n$) cut out by a form with positive maximum. By Theorem~\ref{thm:faces}(b) every form cutting out $\sigma^*$ has negative maximum, so $\tau\ne\sigma^*$. Conversely, every proper face of $P_n$ other than $\sigma^*$ is cut out by a form $\varphi$ with $M_\varphi=1$ (Theorem~\ref{thm:faces}(d)), and this form cuts out the same face of $\Delta_n$, because $\varphi(0)=0<1$.
\end{proof}

\begin{proposition}[A unimodular triangulation]\label{prop:volume}
For $J\subseteq\{1,\dots,n\}$ let
\[
\tau_J=\conv\big(\{0,\mathbf c\}\cup\{\mathbf a_i\}_{i\in J}\cup\{\mathbf b_j\}_{j\notin J}\big).
\]
Each $\tau_J$ is a unimodular $(n+1)$-simplex, $\Delta_n=\bigcup_J\tau_J$, and the interiors of the $\tau_J$ are pairwise disjoint. Consequently
\[
(n+1)!\,\Vol(\Delta_n)=2^n,\qquad(n+1)!\,\Vol(P_n)=2^{n-1}.
\]
\end{proposition}

\begin{proof}
The edge vectors of $\tau_J$ at the vertex $0$ are $-e_i$ ($i\in J$), $e_j+e_{n+1}$ ($j\notin J$) and $e_{n+1}$. Subtracting the last vector from those of the second group turns them into a signed permutation of the standard basis, so $\tau_J$ is a unimodular simplex. Write a point of $\tau_J$ as $\sum_{i\in J}\beta_i\mathbf a_i+\sum_{j\notin J}\delta_j\mathbf b_j+\gamma\mathbf c+\alpha\cdot0$, with nonnegative coefficients of sum $1$. Then $x_i=-\beta_i$ ($i\in J$), $x_j=\delta_j$ ($j\notin J$), $y=\gamma+\sum_{j\notin J}\delta_j$ and $\alpha=1+\sum_{i\in J}x_i-y$. Hence
\begin{equation}\label{eq:tauJ}
\tau_J=\Big\{(x,y):\ x_i\le0\ (i\in J),\ \ x_j\ge0\ (j\notin J),\ \ \sum_{j\notin J}x_j\le y\le1+\sum_{i\in J}x_i\Big\}.
\end{equation}
Let $K=\{(x,y):\sum_ix_i^+\le y\le1-\sum_ix_i^-\}$, where $x^{\pm}=\max(\pm x,0)$. The set $K$ is convex, because $x\mapsto\sum_ix_i^{\pm}$ are convex functions. It contains $0$, $\mathbf c$, every $\mathbf a_i$ (for which $0\le0\le1-1$) and every $\mathbf b_i$ (for which $1\le1\le1$); hence $\Delta_n\subseteq K$. Conversely, given $(x,y)\in K$ put $J=\{i:x_i<0\}$; then \eqref{eq:tauJ} shows $(x,y)\in\tau_J\subseteq\Delta_n$. Thus $\Delta_n=K=\bigcup_J\tau_J$. If $J\ne J'$ and $i$ lies in exactly one of them, then the interior of one of $\tau_J,\tau_{J'}$ lies in $\{x_i<0\}$ and that of the other in $\{x_i>0\}$. A unimodular $(n+1)$-simplex has volume $1/(n+1)!$, which gives the first formula; the second follows from Proposition~\ref{prop:structure}(iii).
\end{proof}

\begin{remark}\label{rem:euclidean}
The volumes can also be computed from the cross-polytope. The Gram matrix of the vectors $w_1,\dots,w_n$ of Proposition~\ref{prop:structure}(ii) is $\mathrm I_n+\frac14\mathrm J_n$, where $\mathrm J_n$ is the all-ones matrix. Its eigenvalues are $1+\frac n4$, once, and $1$, with multiplicity $n-1$, so its determinant is $\frac{n+4}4$. The standard cross-polytope in $\mathbb R^n$ has volume $2^n/n!$. Hence the $n$-dimensional Euclidean volume of the facet $\sigma^*$ is
\[
\Vol_n(\sigma^*)=\frac{2^n}{n!}\sqrt{\frac{n+4}4}.
\]
The facet lies in the hyperplane $\varphi^*=-1$, whose normal vector $(1,\dots,1,-2)$ has length $\sqrt{n+4}$, and $\varphi^*(\mathbf c)=-2$. So the apex $\mathbf c$ is at distance $1/\sqrt{n+4}$ from this hyperplane, and
\[
\Vol(P_n)=\frac1{n+1}\cdot\frac{2^n}{n!}\sqrt{\frac{n+4}4}\cdot\frac1{\sqrt{n+4}}=\frac{2^{n-1}}{(n+1)!},
\]
in agreement with Proposition~\ref{prop:volume}. For $n=4$, $\Vol_4(\sigma^*)=\frac23\sqrt2$.
\end{remark}

\begin{proposition}[Non-degeneracy]\label{prop:nondeg}
Let $h\in(\Fpbar^\times)^n$. For every face $\tau$ of $\Delta_n$ not containing $0$, the partial derivatives of the face polynomial $F_{n,\tau}$, the sum of the monomials of $F_n$ whose exponents lie in $\tau$, have no common zero on the torus $\Gm^{n+1}$ over $\Fpbar$. Thus $F_n$ is non-degenerate with respect to $\Delta_n$ in the sense of \cite{AS}, and $\dim\Delta_n=n+1$.
\end{proposition}

\begin{proof}
By Proposition~\ref{prop:structure}(iv),(v), $\tau=\tau(J_{\mathrm A},J_{\mathrm B},\eta)$ with $J_{\mathrm A}\cap J_{\mathrm B}=\varnothing$. If $i\in J_{\mathrm A}$, then $i\notin J_{\mathrm B}$, so the only monomial of $F_{n,\tau}$ involving $u_i$ is $u_i^{-1}$, and $\partial F_{n,\tau}/\partial u_i=-u_i^{-2}$ does not vanish on the torus. If $J_{\mathrm A}=\varnothing$ and $j\in J_{\mathrm B}$, then $\partial F_{n,\tau}/\partial u_j=h_jt\ne0$. If $J_{\mathrm A}=J_{\mathrm B}=\varnothing$, then $F_{n,\tau}=\lambda t$, and $\partial F_{n,\tau}/\partial t=\lambda\ne0$. Finally $\dim\Delta_n=n+1$ by Proposition~\ref{prop:structure}(i).
\end{proof}

\begin{proposition}[Critical points]\label{prop:crit}
Let $h\in(\Fpbar^\times)^n$, fix $s_i\in\Fpbar$ with $s_i^2=h_i$, and for $\varepsilon\in\{\pm1\}^n$ put $L_\varepsilon=\sum_i\varepsilon_is_i$. The critical points of $F_n(\cdot,\cdot;h)$ on $\Gm^n\times\Gm$ are in bijection with the pairs $\{\varepsilon,-\varepsilon\}$ such that $L_\varepsilon\ne0$; the pair $\{\pm\varepsilon\}$ corresponds to the point
\[
u_i=-\frac{\lambda\varepsilon_i}{s_iL_\varepsilon},\qquad t=\frac{L_\varepsilon^2}{\lambda^2}.
\]
At each critical point, $\det\operatorname{Hess}F_n=2^{n-1}\lambda/(t\,u_1^3\cdots u_n^3)\ne0$. In particular, on the dense open subset $\Omega_n\subset\Gm^n$ where $\prod_\varepsilon L_\varepsilon\ne0$, the phase has exactly $2^{n-1}$ critical points, all of them non-degenerate.
\end{proposition}

\begin{proof}
The critical equations are $u_i^{-2}=th_i$ ($1\le i\le n$) and $\lambda+\langle h,u\rangle=0$. Let $(u,t)$ be a critical point and $\tau$ a square root of $t$. Then $(s_i\tau u_i)^2=1$, so $\varepsilon_i:=s_i\tau u_i\in\{\pm1\}$ and $u_i=\varepsilon_i/(s_i\tau)$. The last equation becomes $\lambda+L_\varepsilon/\tau=0$, so $L_\varepsilon\ne0$ and $\tau=-L_\varepsilon/\lambda$. Replacing $\tau$ by $-\tau$ replaces $\varepsilon$ by $-\varepsilon$, so the critical point determines the pair $\{\pm\varepsilon\}$. Conversely, if $L_\varepsilon\ne0$, the displayed formulas define a critical point, and $\varepsilon$ and $-\varepsilon$ give the same point. If $\varepsilon$ and $\varepsilon'$ give the same point, the corresponding square roots satisfy $\tau'=\pm\tau$, and hence $\varepsilon'=\pm\varepsilon$.

The Hessian in $(u_1,\dots,u_n,t)$ is $\left(\begin{smallmatrix}D&h\\ h^{T}&0\end{smallmatrix}\right)$ with $D=\operatorname{diag}(2u_i^{-3})$. By the Schur complement formula its determinant is $-\det(D)\,h^{T}D^{-1}h=-\frac{2^n}{\prod_iu_i^3}\cdot\frac12\sum_ih_i^2u_i^3$. At a critical point $h_i=u_i^{-2}t^{-1}$, so $h_i^2u_i^3=t^{-2}u_i^{-1}$ and $u_i^{-1}=th_iu_i$. Therefore $\sum_ih_i^2u_i^3=t^{-1}\langle h,u\rangle=-\lambda/t$, and the determinant equals $2^{n-1}\lambda/(t\prod_iu_i^3)$, which is nonzero because $p$ is odd. The open set $\Omega_n$ is dense: the product $\prod_\varepsilon L_\varepsilon$ is invariant under all sign changes of the $s_i$, hence a polynomial in $h$, and it is nonzero because its coefficient of $h_1^{2^{n-1}}$ is $\pm1$.
\end{proof}

\begin{remark}\label{rem:critical-values}
At the critical point attached to $\{\pm\varepsilon\}$ the constraint $\lambda+\langle h,u\rangle=0$ holds, so the critical value is $f(u)=\sum_iu_i^{-1}=-\lambda^{-1}L_\varepsilon\sum_i\varepsilon_is_i=-L_\varepsilon^2/\lambda$. Since $\sum_{\varepsilon\in\{\pm1\}^n}L_\varepsilon^2=\sum_{i,j}s_is_j\sum_\varepsilon\varepsilon_i\varepsilon_j=2^n\sum_ih_i$, the sum of the critical values over the $2^{n-1}$ critical points, for $h\in\Omega_n$, is
\[
-\frac1\lambda\sum_{\{\pm\varepsilon\}}L_\varepsilon^2=-\frac{2^{n-1}}\lambda\,(h_1+\cdots+h_n).
\]
This is the function in Theorem~\ref{thmE}(i). The identity of Theorem~\ref{thmE}(i) holds on all of $S$, although critical points escape to infinity along the divisors $L_\varepsilon=0$ (Remark~\ref{rem:collision}); along them the critical value of the escaping point, $-L_\varepsilon^2/\lambda$, tends to $0$. The proof of Theorem~\ref{thmE}(i) does not use critical points.
\end{remark}

\begin{remark}\label{rem:AS}
(a) By Propositions~\ref{prop:nondeg} and~\ref{prop:volume}, the theorems of Adolphson--Sperber \cite{AS} and Denef--Loeser \cite{DL} predict that $H^i_c(\Gm^{n+1}\otimes\Fpbar,\Lpsi{F_n})$ vanishes for $i\ne n+1$ and has dimension $(n+1)!\Vol(\Delta_n)=2^n$; we prove this directly in Proposition~\ref{prop:fibration}. The origin is a vertex of $\Delta_n$, not an interior point, so the purity criterion of \cite{AS,DL} does not apply, and indeed the cohomology is mixed (Proposition~\ref{prop:weights-torus}).

(b) The face $\sigma^*$ of $P_n$ is not a face of $\Delta_n$, and by Proposition~\ref{prop:nondeg} it is the only proper face of $P_n$ at which non-degeneracy can fail. Its face polynomial is $\sum_i(u_i^{-1}+h_itu_i)$. The computation in the proof of Proposition~\ref{prop:crit} shows that the partial derivatives of this polynomial have a common zero on the torus if and only if $L_\varepsilon=0$ for some $\varepsilon$. On $\Omega_n$ the phase is therefore non-degenerate with respect to every proper face of $P_n$. The count $2^{n-1}$ of Proposition~\ref{prop:crit} is the one predicted by Bernstein's theorem \cite{Bernstein}. Indeed, the critical points are the common zeros on the torus of the $n+1$ Laurent polynomials
\[
u_i\frac{\partial F_n}{\partial u_i}=-u_i^{-1}+h_itu_i\quad(1\le i\le n),\qquad t\frac{\partial F_n}{\partial t}=t\big(\lambda+\langle h,u\rangle\big),
\]
whose Newton polytopes are the segments $[\mathbf a_i,\mathbf b_i]$ and the simplex $Q=\conv\{\mathbf c,\mathbf b_1,\dots,\mathbf b_n\}$. Normalize mixed volumes in $\mathbb R^{n+1}$ so that $n+1$ copies of a polytope $P$ have mixed volume $(n+1)!\Vol(P)$. As the first $n$ polytopes are segments with directions $v_i=\mathbf b_i-\mathbf a_i=2e_i+e_{n+1}$, their mixed volume with $Q$ is $\max_{y,y'\in Q}|\det(v_1,\dots,v_n,y-y')|$. The linear form $w\mapsto\det(v_1,\dots,v_n,w)$ and the form $\varphi^*$ both vanish on the hyperplane spanned by the linearly independent vectors $v_i$, so they are proportional. Expanding $\det(v_1,\dots,v_n,e_1)=-2^{n-1}$ along its last column, and using $\varphi^*(e_1)=1$, gives $\det(v_1,\dots,v_n,w)=-2^{n-1}\varphi^*(w)$. Since $\varphi^*$ takes the value $-2$ at $\mathbf c$ and $-1$ at every $\mathbf b_i$, the mixed volume is $2^{n-1}$, which also equals $(n+1)!\Vol(P_n)$ (Proposition~\ref{prop:volume}). The cohomology of the torus is twice as large because $\Delta_n$ contains the second pyramid $\conv(\sigma^*\cup\{0\})$. Nothing in this remark is used later.
\end{remark}

\begin{remark}\label{rem:purity-example}
The failure of purity in Remark~\ref{rem:AS}(a) can be seen already in dimension one. Let $g(u)=u+u^2$ on $\Gm$. Its Newton polytope is $[1,2]$ and its Newton polyhedron at infinity is $[0,2]$, of normalized volume $2$, with $0$ as a vertex. The only face not containing $0$ is $\{2\}$, with face polynomial $u^2$, whose derivative $2u$ does not vanish on $\Gm$ because $p$ is odd; so $g$ is non-degenerate. The sheaf $\Lpsi g$ on $\mathbb A^1$ is lisse, with $\Swan_\infty=2$, so its compactly supported cohomology is concentrated in degree $1$ and is one-dimensional; it is pure of weight $1$ by \cite[3.2.3]{Deligne}, since $\Lpsi g$ is totally wild at $\infty$. As $g(0)=0$, excision for $\{0\}\subset\mathbb A^1$ gives an exact sequence
\[
0\to\Qlb\to H^1_c(\Gm\otimes\Fqbar,\Lpsi g)\to H^1_c(\mathbb A^1\otimes\Fqbar,\Lpsi g)\to0,
\]
with $\Frob_q$ acting trivially on $\Qlb$. Thus $H^1_c(\Gm\otimes\Fqbar,\Lpsi g)$ has dimension $2$, as predicted by \cite{AS}, and it is mixed, of weights $0$ and $1$. For the Kloosterman phase $u^{-1}+au$, by contrast, the polyhedron at infinity is $[-1,1]$, the origin is an interior point, and the cohomology is pure.
\end{remark}

\section{The cohomology of a single fibre}\label{sec:fibre}

In this section $n\ge2$ and $h\in(\Fq^\times)^n$. Cohomology groups are taken over $\Fqbar$ and carry the action of the geometric Frobenius $\Frob_q$. We write $\mathcal L=\Lpsi{F_n}$ on $X_h$.

\begin{proposition}\label{prop:fibration}
Let $\pr\colon X_h^\circ\to\Gmn{t}$ be the projection and
\[
\mathcal N_h=\Big(\bigotimes_{i=1}^n[\times h_i]^*\Kl\Big)\otimes\Lpsi{\lambda t}.
\]
\begin{enumerate}[(i)]
\item $R\pr_!\mathcal L\cong\mathcal N_h[-n]$, and $\mathcal N_h$ is lisse of rank $2^n$ on $\Gm$ and pure of weight $n$.
\item $\mathcal N_h$ is tame at $0$, with $\mathcal N_h(0)\cong\Up{2}^{\otimes n}$, and totally wild at $\infty$, with all breaks equal to $1$. Hence $\Swan_0(\mathcal N_h)=0$ and $\Swan_\infty(\mathcal N_h)=2^n$.
\item $H^i_c(\Gm,\mathcal N_h)=0$ for $i\ne1$, and $\dim H^1_c(\Gm,\mathcal N_h)=2^n$. Consequently $H^i_c(X_h^\circ,\mathcal L)\cong H^{i-n}_c(\Gm,\mathcal N_h)$ vanishes for $i\ne n+1$ and has dimension $2^n$ for $i=n+1$.
\end{enumerate}
\end{proposition}

\begin{proof}
(i) Since $F_n=\lambda t+\sum_i(u_i^{-1}+h_itu_i)$, we have $\mathcal L\cong\pr^*\Lpsi{\lambda t}\otimes\bigotimes_i\Lpsi{u_i^{-1}+h_itu_i}$, the $i$-th factor being pulled back from $\Gmn{u_i}\times\Gmn{t}$. By (K4) and proper base change along $t\mapsto h_it$, the direct image of $\Lpsi{u_i^{-1}+h_itu_i}$ to $\Gmn{t}$ is $[\times h_i]^*\Kl[-1]$. The K\"unneth formula over $\Gmn{t}$ and the projection formula give $R\pr_!\mathcal L\cong\mathcal N_h[-n]$. The remaining assertions of (i) follow from (K1).

(ii) By (K2) and \S\ref{sec:conventions}, each $[\times h_i]^*\Kl$ is tame at $0$ with local monodromy $\Up{2}$, and $\Lpsi{\lambda t}$ is lisse at $0$; the tame generator acts on the tensor product by the tensor product of its actions. At $\infty$ each $[\times h_i]^*\Kl$ has both breaks in $(0,1)$ by (K3), so their tensor product has all breaks $<1$ by Lemma~\ref{lem:breaks}(i). Since $\lambda\ne0$, the sheaf $\Lpsi{\lambda t}$ is isoclinic of break $1$ at $\infty$, and Lemma~\ref{lem:breaks}(ii) shows that $\mathcal N_h$ is isoclinic of break $1$ there.

(iii) By \eqref{eq:GOS}, $\chi_c(\Gm,\mathcal N_h)=-2^n$. Moreover $H^0_c=0$, and $H^2_c=0$ because $\mathcal N_h$ is totally wild at $\infty$. The last assertion follows from (i) and the Leray spectral sequence of $\pr$, which degenerates.
\end{proof}

\begin{proposition}\label{prop:weights-torus}
Let $j\colon\Gm\hookrightarrow\mathbb P^1$ be the inclusion. There is an exact sequence of $\Frob_q$-modules
\[
0\longrightarrow\mathcal N_h^{I_0}\longrightarrow H^1_c(\Gm,\mathcal N_h)\longrightarrow H^1(\mathbb P^1,j_*\mathcal N_h)\longrightarrow0,
\]
where $\mathcal N_h^{I_0}=(j_*\mathcal N_h)_0$ carries the action of the Frobenius at the rational point $0$. The right-hand term is pure of weight $n+1$ and has dimension $r(n)$. For $0\le k\le\lfloor n/2\rfloor$, exactly $m_{n,k}$ of the Frobenius eigenvalues on $\mathcal N_h^{I_0}$, counted with multiplicity, have weight $2k$, and there are no others. In particular $H^1_c(\Gm,\mathcal N_h)$ contains a unique Frobenius-stable line of weight $0$ and has no other eigenvalue of weight $0$.
\end{proposition}

\begin{proof}
The adjunction map $j_!\mathcal N_h\to j_*\mathcal N_h$ is injective, with cokernel the skyscraper sheaves $\mathcal N_h^{I_0}$ at $0$ and $\mathcal N_h^{I_\infty}$ at $\infty$; the latter vanishes because $\mathcal N_h$ is totally wild at $\infty$. Since $H^0(\mathbb P^1,j_*\mathcal N_h)=H^0(\Gm,\mathcal N_h)$ is contained in $\mathcal N_h^{I_\infty}=0$, and skyscraper sheaves have no $H^1$, the long exact cohomology sequence gives the displayed sequence, with $H^1(\mathbb P^1,j_!\mathcal N_h)=H^1_c(\Gm,\mathcal N_h)$. By Deligne's theorem \cite[3.2.3]{Deligne}, $H^1(\mathbb P^1,j_*\mathcal N_h)$ is pure of weight $n+1$. By Proposition~\ref{prop:fibration}(ii) and Lemma~\ref{lem:CG}, the local monodromy of $\mathcal N_h$ at $0$ is unipotent, with $m_{n,k}$ Jordan blocks of size $n-2k+1$ for each $k$. So $\dim\mathcal N_h^{I_0}=\binom n{\lfloor n/2\rfloor}$, and the right-hand term has dimension $2^n-\binom n{\lfloor n/2\rfloor}=r(n)$. Finally, $\mathcal N_h$ is pure of weight $n$ with unipotent local monodromy at $0$. By Deligne's local monodromy theorem with weights \cite[1.8.4]{Deligne}, the monodromy filtration on the generic stalk, centred at $n$, has graded pieces $\gr^M_i$ pure of weight $n+i$; the invariants of a Jordan block of size $m$ lie in $M_{-(m-1)}$ and have weight $n-(m-1)$. A block of size $n-2k+1$ therefore contributes one eigenvalue of weight $2k$.
\end{proof}

The eigenvalues on $\mathcal N_h^{I_0}$ can be determined exactly.

\begin{proposition}\label{prop:frob-invariants}
Let $h\in(\Fq^\times)^n$. Then $\Frob_q$ acts semisimply on $\mathcal N_h^{I_0}$, and its eigenvalues are the numbers $q^k$, $0\le k\le\lfloor n/2\rfloor$, with multiplicity $m_{n,k}$.
\end{proposition}

\begin{proof}
Let $D_0$ be a decomposition group at the rational point $t=0$ of $\mathbb A^1$ over $\Fq$, and let $I_0$, $\Iw_0$, $\gamma$ and $\phi$ be as in Lemma~\ref{lem:frob-block}, with $x=0$. For each $i$ let $V_i$ be the generic stalk of $[\times h_i]^*\Kl$, as a representation of $D_0$. By (K2) and \S\ref{sec:conventions}, $\Iw_0$ acts trivially on $V_i$ and $\gamma$ acts by a unipotent automorphism $T_i\ne1$. The map $t\mapsto h_it$ extends to an automorphism of $\mathbb A^1$ over $\Fq$ fixing $0$, so $j_*[\times h_i]^*\Kl=[\times h_i]^*j_*\Kl$, and $V_i^{I_0}=(j_*\Kl)_0$ as $\Frob_q$-modules. By (K5), $\phi$ acts trivially on $V_i^{I_0}$. Lemma~\ref{lem:frob-block} provides a basis $e^{(i)}_1,e^{(i)}_2$ of $V_i$ with $N_ie^{(i)}_1=0$, $N_ie^{(i)}_2=e^{(i)}_1$, $\phi e^{(i)}_1=e^{(i)}_1$ and $\phi e^{(i)}_2=qe^{(i)}_2$, where $N_i=\log T_i$.

The sheaf $\Lpsi{\lambda t}$ is lisse at $0$, and $\phi$ acts on its stalk at $0$ by $\psi_q(0)=1$. So the generic stalk of $\mathcal N_h$ is $V=V_1\otimes\cdots\otimes V_n$ as a representation of $D_0$. On it $\gamma$ acts by $T_1\otimes\cdots\otimes T_n=\exp\big(\sum_iN_i\big)$, where $N_i$ acts on the $i$-th factor, and $\phi$ acts on the basis vectors $e_J=e^{(1)}_{j_1}\otimes\cdots\otimes e^{(n)}_{j_n}$ by $\phi e_J=q^{|J|}e_J$. As in the proof of Lemma~\ref{lem:CG}, the invariants of $\exp(\sum_iN_i)$ form the kernel $E$ of $\sum_iN_i$, and in the basis $(e_J)$ the operator $\sum_iN_i$ is the one considered before~\eqref{eq:Ek}. Hence $\mathcal N_h^{I_0}=E=\bigoplus_kE_k$ with $\dim E_k=m_{n,k}$. Since $I_0$ is normal in $D_0$, the element $\phi$ preserves $E$; it also preserves the span of the $e_J$ with $|J|=k$, on which it acts as $q^k$. So $\phi$ acts on $E_k$ as multiplication by $q^k$. As $\phi$ induces $\Frob_q$ on $\mathcal N_h^{I_0}=(j_*\mathcal N_h)_0$, the proposition follows.
\end{proof}

This recovers the last assertion of Proposition~\ref{prop:weights-torus} without \cite[1.8.4]{Deligne}: the eigenvalue $q^k$ has weight $2k$.

\begin{lemma}\label{lem:t0}
$\mathcal L|_{Z_h}\cong\boxtimes_{i=1}^n\Lpsi{u_i^{-1}}$. Its compactly supported cohomology vanishes outside degree $n$, and $H^n_c(Z_h,\mathcal L|_{Z_h})$ is one-dimensional, with $\Frob_q$ acting trivially.
\end{lemma}

\begin{proof}
Setting $t=0$ leaves the phase $\sum_iu_i^{-1}$. On $\Gm$, the sheaf $\Lpsi{u^{-1}}$ is lisse of rank one, with $\Swan_0=1$ and $\Swan_\infty=0$. By \eqref{eq:GOS} its Euler characteristic is $-1$; its $H^0_c$ vanishes, and so does its $H^2_c$, because it is geometrically nontrivial. Hence its $H^1_c$ is one-dimensional, and the Grothendieck--Lefschetz formula $-\Tr(\Frob_q\mid H^1_c)=\sum_{u\in\Fq^\times}\psi_q(u^{-1})=-1$ shows that $\Frob_q$ acts trivially on it. The K\"unneth formula concludes.
\end{proof}

\begin{theorem}\label{thm:fibre}
Let $n\ge2$ and $h\in(\Fq^\times)^n$.
\begin{enumerate}[(i)]
\item $H^i_c(X_h,\mathcal L)=0$ for $i\ne n+1$.
\item The connecting homomorphism $\partial\colon H^n_c(Z_h,\mathcal L|_{Z_h})\to H^{n+1}_c(X_h^\circ,\mathcal L)$ of the excision sequence is injective, and its image is the unique line of weight $0$.
\item $\dim H^{n+1}_c(X_h,\mathcal L)=2^n-1$. The Frobenius eigenvalues are $r(n)$ eigenvalues of weight $n+1$ and, for $1\le k\le\lfloor n/2\rfloor$, the eigenvalue $q^k$ with multiplicity $m_{n,k}$. The subspace of weights $\le n$ is a direct sum of eigenspaces of $\Frob_q$.
\item $H^i_c(Y_h,\Lpsi f)=0$ for $i\ne n-1$, $H^{n-1}_c(Y_h,\Lpsi f)\cong H^{n+1}_c(X_h,\mathcal L)(1)$, and
\[
\Sigma_n(h)=(-1)^{n-1}\Tr\big(\Frob_q\mid H^{n-1}_c(Y_h,\Lpsi f)\big).
\]
The eigenvalues on $H^{n-1}_c(Y_h,\Lpsi f)$ are $r(n)$ of weight $n-1$ and, for $1\le k\le\lfloor n/2\rfloor$, the eigenvalue $q^{k-1}$ with multiplicity $m_{n,k}$. Consequently
\begin{equation}\label{eq:exact-formula}
\Sigma_n(h)=(-1)^{n-1}\Big(\tau_q(h)+\sum_{k=1}^{\lfloor n/2\rfloor}m_{n,k}\,q^{k-1}\Big),
\end{equation}
where $\tau_q(h)$ is the trace of $\Frob_q$ on the part of weight $n-1$, so that $|\tau_q(h)|\le r(n)q^{(n-1)/2}$.
\end{enumerate}
In particular Theorem~\ref{thmA} holds.
\end{theorem}

\begin{proof}
By Lemma~\ref{lem:integrate-t}, applied with $V=\Gm^n$, $g(u)=\lambda+\langle h,u\rangle$ and $K=\Lpsi f$, we have $H^i_c(X_h,\mathcal L)\cong H^{i-2}_c(Y_h,\Lpsi f)(-1)$ for all $i$. The variety $Y_h$ is smooth and affine of pure dimension $n-1$. By Artin's vanishing theorem \cite[Exp.~XIV]{SGA4} and Poincar\'e duality \cite[Exp.~XVIII]{SGA4}, $H^j_c(Y_h,\Lpsi f)=0$ for $j<n-1$. Hence $H^i_c(X_h,\mathcal L)=0$ for $i\le n$.

Consider the excision sequence
\[
\cdots\to H^i_c(X_h^\circ,\mathcal L)\to H^i_c(X_h,\mathcal L)\to H^i_c(Z_h,\mathcal L|_{Z_h})\xrightarrow{\ \partial\ }H^{i+1}_c(X_h^\circ,\mathcal L)\to\cdots.
\]
By Proposition~\ref{prop:fibration}(iii) and Lemma~\ref{lem:t0}, the groups $H^i_c(X_h^\circ,\mathcal L)$ and $H^i_c(Z_h,\mathcal L|_{Z_h})$ vanish except for $H^{n+1}_c(X_h^\circ,\mathcal L)$ and $H^n_c(Z_h,\mathcal L|_{Z_h})$. For $i\ge n+2$ this gives $H^i_c(X_h,\mathcal L)=0$, which proves (i), and the remaining part of the sequence reads
\[
0=H^n_c(X_h,\mathcal L)\to H^n_c(Z_h,\mathcal L|_{Z_h})\xrightarrow{\ \partial\ }H^{n+1}_c(X_h^\circ,\mathcal L)\to H^{n+1}_c(X_h,\mathcal L)\to0.
\]
Thus $\partial$ is injective. It is Frobenius-equivariant, and its source is a line on which Frobenius acts trivially (Lemma~\ref{lem:t0}). Hence its image is a Frobenius-stable line of weight $0$, which by Proposition~\ref{prop:weights-torus} is the unique such line in $H^{n+1}_c(X_h^\circ,\mathcal L)\cong H^1_c(\Gm,\mathcal N_h)$. This proves (ii). For (iii), $H^{n+1}_c(X_h,\mathcal L)$ is the cokernel of $\partial$. By Proposition~\ref{prop:weights-torus}, the part of weight $\le n$ of $H^1_c(\Gm,\mathcal N_h)$ is the image of $\mathcal N_h^{I_0}$, and its quotient is pure of weight $n+1$ and of dimension $r(n)$. By Proposition~\ref{prop:frob-invariants}, $\Frob_q$ acts on $\mathcal N_h^{I_0}$ as $q^k$ on a subspace $E_k$ of dimension $m_{n,k}$, with $\mathcal N_h^{I_0}=\bigoplus_kE_k$, and the image of $\partial$ is the line $E_0$. Hence the part of weight $\le n$ of the cokernel is $\bigoplus_{k\ge1}E_k$, which proves (iii).

For (iv), the vanishing and the isomorphism follow from (i) and the first paragraph. The Grothendieck--Lefschetz trace formula on $Y_h$ gives $\Sigma_n(h)=\sum_i(-1)^i\Tr\big(\Frob_q\mid H^i_c(Y_h,\Lpsi f)\big)$, which reduces to the displayed formula. The Tate twist $(1)$ divides eigenvalues by $q$ and lowers weights by $2$, which gives the eigenvalues and~\eqref{eq:exact-formula}. Since an eigenvalue of weight $w$ has absolute value $q^{w/2}$, the bound on $\tau_q(h)$ follows.
\end{proof}

\begin{corollary}\label{cor:all-h}
Let $h\in\Fq^n$, and let $a$ be the number of nonzero coordinates of $h$. Then $H^i_c(Y_h\otimes\Fqbar,\Lpsi f)=0$ for $i\ne n-1$, $\dim H^{n-1}_c(Y_h\otimes\Fqbar,\Lpsi f)=2^a-1$, and $\Sigma_n(h)=(-1)^{n-1}\Tr\big(\Frob_q\mid H^{n-1}_c(Y_h,\Lpsi f)\big)$. For $a\ge2$ the eigenvalues are those of Theorem~\ref{thm:fibre}(iv) with $n$ replaced by $a$; for $a=1$ and $h_i\ne0$ the single eigenvalue is $\psi_q(-h_i/\lambda)$.
\end{corollary}

\begin{proof}
If $a=0$, then $Y_h$ is empty, since $\lambda\ne0$. Otherwise let $I$ be the set of indices of the nonzero coordinates, $h_I$ the vector they form, and $f_I(u_I)=\sum_{i\in I}u_i^{-1}$. Then $Y_h\cong Y_{h_I}\times\Gm^{n-a}$, where $Y_{h_I}\subset\Gm^I$ is defined by $\langle h_I,u_I\rangle=-\lambda$, and $\Lpsi f\cong\Lpsi{f_I}\boxtimes\boxtimes_{i\notin I}\Lpsi{u_i^{-1}}$. By Lemma~\ref{lem:t0} and the K\"unneth formula,
\[
R\Gamma_c\big(Y_h\otimes\Fqbar,\Lpsi f\big)\cong R\Gamma_c\big(Y_{h_I}\otimes\Fqbar,\Lpsi{f_I}\big)[-(n-a)],
\]
compatibly with $\Frob_q$. For $a\ge2$, Theorem~\ref{thm:fibre} applied to $h_I$ gives the assertions. For $a=1$, $Y_{h_I}$ is the point $u_i=-\lambda/h_i$, at which $\Frob_q$ acts on the stalk of $\Lpsi{u_i^{-1}}$ by $\psi_q(-h_i/\lambda)$. The trace formula concludes.
\end{proof}

By Corollary~\ref{cor:all-h}, the dimension of the cohomology drops from $2^n-1$ to $2^{n-1}-1$ on the coordinate hyperplanes. This is consistent with the nontrivial unipotent monodromy of $\mathcal W_n$ around them (Theorem~\ref{thmC}(i)).

Table~\ref{tab:small} lists the resulting data for small $n$. The proof passes through the following chain of dimensions, which we spell out for $n=4$. The phase has $8=2^{n-1}$ critical points for generic $h$ (Proposition~\ref{prop:crit}), while the cohomology of the torus has dimension $16=2^n$ (Proposition~\ref{prop:fibration}). The local monodromy of $\mathcal N_h$ at $t=0$ is $\Up{2}^{\otimes4}\cong\Up{5}\oplus3\,\Up{3}\oplus2\,\Up{1}$, so the invariants $\mathcal N_h^{I_0}$ have dimension $6$, with eigenvalues $1$, $q,q,q$ and $q^2,q^2$ (Proposition~\ref{prop:frob-invariants}). The connecting map $\partial$ removes the line with eigenvalue $1$, which leaves dimension $15$. The pure quotient $H^1(\mathbb P^1,j_*\mathcal N_h)$ has dimension $10$ and weight $5$. After the Tate twist, $H^3_c(Y_h,\Lpsi f)$ has ten eigenvalues of weight $3$, the eigenvalue $1$ three times and the eigenvalue $q$ twice.

\begin{table}[ht]
\centering
\caption{Theorem~\ref{thm:fibre} for small $n$: the number of critical points for generic $h$, the dimensions of the cohomology of the torus and of the invariants at $t=0$, the dimensions of $H^{n-1}_c(Y_h,\Lpsi f)$ and of its part of weight $n-1$, and the remaining eigenvalues, with multiplicities as exponents. The last column is the exact formula~\eqref{eq:exact-formula}, where $\tau=\tau_q(h)$ is a sum of $r(n)$ Weil numbers of weight $n-1$.}\label{tab:small}
\begin{tabular}{cccccccl}
\toprule
$n$ & $2^{n-1}$ & $2^n$ & $\binom n{\lfloor n/2\rfloor}$ & $2^n-1$ & $r(n)$ & lower eigenvalues & $\Sigma_n(h)$\\
\midrule
$2$ & $2$ & $4$ & $2$ & $3$ & $2$ & $1^{(1)}$ & $-1-\tau$\\
$3$ & $4$ & $8$ & $3$ & $7$ & $5$ & $1^{(2)}$ & $2+\tau$\\
$4$ & $8$ & $16$ & $6$ & $15$ & $10$ & $1^{(3)},\ q^{(2)}$ & $-3-2q-\tau$\\
$5$ & $16$ & $32$ & $10$ & $31$ & $22$ & $1^{(4)},\ q^{(5)}$ & $4+5q+\tau$\\
\bottomrule
\end{tabular}
\end{table}

\section{The family over the parameter torus}\label{sec:family-S}

Let $S=\Gm^n$ over $\Fp$, with coordinates $h$, and let
\[
X=\Gm^n\times\mathbb A^1_t\times S\supset X^\circ=\{t\ne0\},\qquad Z=\{t=0\},
\]
with projections $\pi$, $\pi^\circ$ and $\pi_Z$ to $S$. On $X$ we again write $\mathcal L=\Lpsi{F_n}$, now with $h$ variable, and put
\[
\mathcal G_n=R^{n+1}\pi_!\mathcal L,\qquad\mathcal H_n=R^{n+1}\pi^\circ_!\mathcal L.
\]
By proper base change, the geometric stalks of these sheaves at $h$ are $H^{n+1}_c(X_h,\mathcal L)$ and $H^{n+1}_c(X_h^\circ,\mathcal L)$.

\begin{theorem}[Deligne--Laumon]\label{thm:DL}
Let $T$ be an excellent noetherian scheme, $f\colon\mathbb P^1_T\to T$ the projection, $D\subset\mathbb P^1_T$ a closed subscheme finite and \'etale over $T$, $j\colon V=\mathbb P^1_T\smallsetminus D\hookrightarrow\mathbb P^1_T$ the inclusion, and $F$ a lisse sheaf of constant rank on $V$. Let $\sigma_F$ be the function $\bar t\mapsto\sum_{x\in D_{\bar t}}\Swan_x(F|_{V_{\bar t}})$ on geometric points of $T$.
\begin{enumerate}[(i)]
\item The function $\sigma_F$ is constructible.
\item If $\sigma_F$ is locally constant, then $R^if_*(j_!F)$ is lisse on $T$ for every $i$.
\end{enumerate}
\end{theorem}

This is \cite[Th.~2.1.1 and Cor.~2.1.2]{Laumon-semicont}. Laumon states it for locally constant constructible sheaves of $\mathbb F_\ell$-vector spaces on a smooth relative curve, with a divisor finite and flat over an excellent noetherian base, which covers the situation here: by Th.~2.1.1(i) the Swan function is constructible (we shall apply this with $T$ smooth over $\Fp$); by Th.~2.1.1(ii), if it is locally constant, then $j_!F$ is universally locally acyclic relative to $f$; and then, $f$ being proper, the sheaves $R^if_*(j_!F)$ are locally constant (Cor.~2.1.2). The $\ell$-adic statement follows by d\'evissage. Let $E$ be a finite extension of $\mathbb Q_\ell$, with uniformizer $\varpi$ and residue field $k_E$, such that $F$ has a lisse $\mathcal O_E$-lattice $F_0$, and put $\bar F_0=F_0/\varpi F_0$, a lisse sheaf of $k_E$-vector spaces, hence of $\mathbb F_\ell$-vector spaces. Wild inertia acts through finite $p$-groups and $\ell\ne p$, so taking invariants under the higher ramification groups commutes with reduction modulo~$\varpi$; hence $\sigma_{\bar F_0}=[k_E:\mathbb F_\ell]\,\sigma_F$, and (i) follows. Under the hypothesis of~(ii), the sheaves $R^if_*(j_!\bar F_0)$ are locally constant and constructible. Multiplication by $\varpi^m$ gives exact sequences $0\to\bar F_0\to F_0/\varpi^{m+1}F_0\to F_0/\varpi^mF_0\to0$. Kernels, cokernels and extensions of locally constant constructible sheaves are again locally constant and constructible. So the long exact sequences of the $R^if_*j_!$ and induction on $m$ show that every $R^if_*\big(j_!(F_0/\varpi^mF_0)\big)$ is locally constant and constructible. As $f$ is proper, $R^if_*(j_!F_0)$ is the $\varpi$-adic sheaf defined by the AR-$\varpi$-adic projective system $\big(R^if_*(j_!(F_0/\varpi^mF_0))\big)_{m\ge1}$ \cite[Exp.~V and~VI]{SGA5}. This system is AR-isomorphic to the system of stable images of its transition maps, and these images are again locally constant and constructible. Hence $R^if_*(j_!F_0)$ is a lisse $\varpi$-adic sheaf, and $R^if_*(j_!F)$ is lisse. By proper base change, the stalks of $R^if_*(j_!F)=R^i(fj)_!F$ are the groups $H^i_c(V_{\bar t},F)$.

\begin{remark}\label{rem:constant-rank}
Theorem~\ref{thm:fibre} shows that the stalks of $R^{n+1}\pi_!\mathcal L$ have the same dimension $2^n-1$ at all geometric points of $S$. This alone does not imply that the sheaf is lisse, even over a smooth connected base. For instance, on $\mathbb A^1$ the constructible sheaf $i_*\Qlb\oplus j_!\Qlb$, where $i\colon\{0\}\hookrightarrow\mathbb A^1$ and $j\colon\Gm\hookrightarrow\mathbb A^1$, has one-dimensional stalks everywhere but is not locally constant. This is why we use Theorem~\ref{thm:DL}.
\end{remark}

\begin{theorem}\label{thm:lisse}
\begin{enumerate}[(i)]
\item Let $q_1\colon\Gm^n\times\Gmn{t}\times S\to\Gmn{t}\times S$ be the projection. Then $Rq_{1!}\mathcal L\cong\mathcal N[-n]$, where
\[
\mathcal N=\Big(\bigotimes_{i=1}^n[\times h_i]^*\Kl\Big)\otimes\Lpsi{\lambda t},\qquad[\times h_i](t,h)=h_it,
\]
is lisse of rank $2^n$ on $\Gmn{t}\times S$.
\item $R^i\pi^\circ_!\mathcal L=0$ for $i\ne n+1$, and $\mathcal H_n$ is lisse of rank $2^n$.
\item $R^i\pi_!\mathcal L=0$ for $i\ne n+1$, and there is an exact sequence of lisse sheaves
\[
0\longrightarrow\Qlb\xrightarrow{\ \partial\ }\mathcal H_n\longrightarrow\mathcal G_n\longrightarrow0,
\]
in which $\Qlb$ is the constant sheaf, with trivial Frobenius action. In particular $\mathcal G_n$ is lisse of rank $2^n-1$, and by Lemma~\ref{lem:integrate-t} its stalk at $h$ is $H^{n-1}_c(Y_h,\Lpsi f)(-1)$.
\item $\mathcal G_n$ has a unique increasing filtration $\mathrm W_\bullet$ by lisse subsheaves such that $\gr^{\mathrm W}_w\mathcal G_n$ is pointwise pure of weight $w$. The nonzero graded pieces are $\gr^{\mathrm W}_{2k}\mathcal G_n$, of rank $m_{n,k}$ for $1\le k\le\lfloor n/2\rfloor$, and $\mathcal W_n:=\gr^{\mathrm W}_{n+1}\mathcal G_n$, of rank $r(n)$. The inverse image $\mathcal H_n^{\le n}\subset\mathcal H_n$ of $\mathrm W_{\le n}\mathcal G_n$ is a lisse subsheaf, pointwise mixed of weights $\le n$, and $\mathcal H_n/\mathcal H_n^{\le n}\cong\mathcal W_n$.
\item For $h\in S(\Fq)$ put $A_q(h)=\Tr(\Frob_h\mid\mathcal W_n)$. Then
\[
(-1)^{n-1}q\,\Sigma_n(h)=A_q(h)+B_q(h),\qquad|A_q(h)|\le r(n)\,q^{(n+1)/2},\qquad B_q(h)=\sum_{k=1}^{\lfloor n/2\rfloor}m_{n,k}\,q^k.
\]
\item For $1\le k\le\lfloor n/2\rfloor$ there is an isomorphism $\gr^{\mathrm W}_{2k}\mathcal G_n\cong\Qlb(-k)^{\oplus m_{n,k}}$ of lisse sheaves on $S$.
\end{enumerate}
In particular Theorem~\ref{thmB} holds.
\end{theorem}

\begin{proof}
(i) This is the relative form of Proposition~\ref{prop:fibration}(i). The same decomposition of $\mathcal L$, (K4), proper base change along $(t,h)\mapsto h_it$, the K\"unneth formula over $\Gmn{t}\times S$ and the projection formula apply. Each $[\times h_i]^*\Kl$ is lisse of rank $2$, being the inverse image of $\Kl$ by a morphism to $\Gm$.

(ii) Let $\pr_S\colon\Gmn{t}\times S\to S$ be the projection, and apply Theorem~\ref{thm:DL} with $T=S$, $D=\{0,\infty\}\times S$ and $F=\mathcal N$. At a geometric point $\bar h$ above a closed point of $S$, Proposition~\ref{prop:fibration}(ii) gives $\Swan_0(\mathcal N_{\bar h})=0$ and $\Swan_\infty(\mathcal N_{\bar h})=2^n$, so $\sigma_{\mathcal N}(\bar h)=2^n$. By Theorem~\ref{thm:DL}(i), every fibre of $\sigma_{\mathcal N}$ is a constructible subset of $S$. As $S$ is of finite type over $\Fp$, Hilbert's Nullstellensatz shows that every nonempty constructible subset of $S$ contains a closed point. Hence $\sigma_{\mathcal N}$ is constant, equal to $2^n$, and by Theorem~\ref{thm:DL}(ii) the sheaf $R^i\pr_{S!}\mathcal N$ is lisse for every $i$, and by Proposition~\ref{prop:fibration}(iii) it vanishes for $i\ne1$ and has rank $2^n$ for $i=1$. The Leray spectral sequence for $\pi^\circ=\pr_S\circ q_1$ and (i) give $R\pi^\circ_!\mathcal L\cong R\pr_{S!}\mathcal N[-n]$, whence (ii).

(iii) The excision triangle for the open subscheme $X^\circ$ and its closed complement $Z$ gives a long exact sequence
\[
\cdots\to R^i\pi^\circ_!\mathcal L\to R^i\pi_!\mathcal L\to R^i\pi_{Z!}(\mathcal L|_Z)\xrightarrow{\ \partial\ }R^{i+1}\pi^\circ_!\mathcal L\to\cdots.
\]
On $Z\cong\Gm^n\times S$ the phase is $\sum_iu_i^{-1}$, which does not involve $h$. So $R\pi_{Z!}(\mathcal L|_Z)$ is the constant complex with value $R\Gamma_c\big(\Gm^n\otimes\Fpbar,\boxtimes_i\Lpsi{u_i^{-1}}\big)$, which by Lemma~\ref{lem:t0} is $\Qlb[-n]$ with trivial Frobenius action. Together with (ii), the sequence shows that $R^i\pi_!\mathcal L=0$ for $i\notin\{n,n+1\}$, and it reduces to an exact sequence
\[
0\to R^n\pi_!\mathcal L\to\Qlb\xrightarrow{\ \partial\ }\mathcal H_n\to\mathcal G_n\to0.
\]
Lisse sheaves on the connected normal scheme $S$ form an abelian subcategory of the constructible sheaves, equivalent to the category of continuous representations of $\pi_1(S)$; so $\ker\partial=R^n\pi_!\mathcal L$ and $\operatorname{coker}\partial=\mathcal G_n$ are lisse. By proper base change and Theorem~\ref{thm:fibre}(ii), the map $\partial$ is injective on the stalk at a geometric point above a closed point of $S$. A lisse sheaf on the connected scheme $S$ with one zero stalk vanishes; hence $R^n\pi_!\mathcal L=0$, which proves (iii). The stalk of $\mathcal G_n$ at any geometric point $h$ is $H^{n+1}_c(X_h,\mathcal L)$ by proper base change, and this is $H^{n-1}_c(Y_h,\Lpsi f)(-1)$ by Lemma~\ref{lem:integrate-t}.

(iv) The sheaf $\mathcal G_n$ is mixed by \cite[Th.~3.3.1]{Deligne}. Since $S$ is normal and connected, \cite[3.4.1]{Deligne} provides the unique filtration $\mathrm W_\bullet$ by lisse subsheaves with pointwise pure graded pieces. Let $h$ be a closed point of $S$, with residue field $\Fq$. The stalks at $\bar h$ of the subsheaves $\mathrm W_w\mathcal G_n$ form a $\Frob_h$-stable filtration of $(\mathcal G_n)_{\bar h}$ whose $w$-th graded piece is pure of weight $w$. Hence the rank of the lisse sheaf $\gr^{\mathrm W}_w\mathcal G_n$ is the number of eigenvalues of weight $w$, counted with multiplicity, of $\Frob_h$ on $(\mathcal G_n)_{\bar h}=H^{n+1}_c(X_h\otimes\Fqbar,\mathcal L)$. These numbers are given by Theorem~\ref{thm:fibre}(iii), a statement about the single fibre over $h$, proved in Section~\ref{sec:fibre} from Propositions~\ref{prop:fibration}, \ref{prop:weights-torus} and~\ref{prop:frob-invariants} and Lemma~\ref{lem:t0} alone. The assertions on $\mathcal H_n^{\le n}$ follow, as $\Qlb$ has weight $0\le n$.

(v) By Theorem~\ref{thm:fibre}(iv), $\Tr(\Frob_h\mid(\mathcal G_n)_h)=q\Tr\big(\Frob_q\mid H^{n-1}_c(Y_h,\Lpsi f)\big)=(-1)^{n-1}q\Sigma_n(h)$. Decomposing the trace along $\mathrm W_\bullet$ gives $A_q(h)+B_q(h)$, where $B_q(h)$ is the sum of the traces on the $\gr^{\mathrm W}_{2k}\mathcal G_n$. The bound on $A_q(h)$ follows from the weight and the rank in (iv). By Theorem~\ref{thm:fibre}(iii), $\Frob_h$ acts on the stalk of $\gr^{\mathrm W}_{2k}\mathcal G_n$ at $\bar h$ as multiplication by $q^k$, which gives the formula for $B_q(h)$.

(vi) Put $\mathcal E=\gr^{\mathrm W}_{2k}\mathcal G_n\otimes\Qlb(k)$, a lisse sheaf of rank $m=m_{n,k}$, pointwise pure of weight $0$. By the proof of~(v), $\Frob_h$ acts trivially on the stalk of $\mathcal E$ at $\bar h$, for every closed point $h$ of $S$. Hence $q^{-n}\sum_{h\in S(\Fq)}\Tr(\Frob_h\mid\mathcal E)=q^{-n}(q-1)^nm$, which tends to $m$ as $q\to\infty$ through powers of $p$. By Lemma~\ref{lem:averages}, applied with $T=S$ and $d=n$, the geometric fundamental group $\pi_1(S\otimes\Fpbar)$ acts trivially on the geometric generic stalk of $\mathcal E$. So $\pi_1(S)$ acts on it through $\mathrm{Gal}(\Fpbar/\Fp)$, which is topologically generated by the image of $\Frob_{h_0}$ for the rational point $h_0=(1,\dots,1)$, and $\Frob_{h_0}$ acts trivially. Hence $\mathcal E$ is constant, which is~(vi).
\end{proof}

\begin{remark}\label{rem:collision}
By Proposition~\ref{prop:crit}, the phase has $2^{n-1}$ critical points for $h\in\Omega_n$, but only $2^{n-1}-1$ at a point $h_0$ where exactly one pair $\{\pm\varepsilon_0\}$ satisfies $L_{\varepsilon_0}=0$. No two critical points collide there. Indeed, the critical locus of $F_n$ in $\Gm^n\times\Gm\times S$ is \'etale over $S$, because its relative Jacobian is the Hessian, which is invertible at every critical point for every $h$ (Proposition~\ref{prop:crit}). The drop in the number of points reflects a failure of finiteness. Instead, the critical point attached to $\{\pm\varepsilon_0\}$ escapes to infinity. As functions of the square roots $s_i$, its coordinates $u_i=-\lambda\varepsilon_{0,i}/(s_iL_{\varepsilon_0})$ have a pole and $t=L_{\varepsilon_0}^2/\lambda^2$ has a zero along $L_{\varepsilon_0}=0$. This is the degeneration of the phase on the face $\sigma^*$ of $P_n$ described in Remark~\ref{rem:AS}(b). Theorem~\ref{thm:lisse} shows that $\mathcal G_n$ is nevertheless lisse across these divisors; that is, $\mathcal G_n$ has no local monodromy along them. The proof never passes through the critical points: it only uses that the Swan conductors of $\mathcal N_h$ at $0$ and $\infty$ do not depend on $h$. This is consistent with Proposition~\ref{prop:nondeg}, by which $F_n$ is non-degenerate with respect to $\Delta_n$ for every $h$.
\end{remark}

\section{The coordinate boundary}\label{sec:boundary}

Throughout this section $n\ge2$, and $c=(c_2,\dots,c_n)\in(\Fq^\times)^{n-1}$ for some power $q$ of $p$. Every $c\in(\Fpbar^\times)^{n-1}$ is defined over a finite field, so all geometric statements hold for every such $c$. Let $C_c\subset S\otimes\Fq$ be the line $\{(h_1,c_2,\dots,c_n)\}$, identified with $\Gmn{h_1}$, and put
\[
\mathcal M_c=\Big(\bigotimes_{i=2}^n[\times c_i]^*\Kl\Big)\otimes\Lpsi{\lambda t}\qquad\text{on }\Gmn{t}.
\]
For $j\ge1$ let $\kappa_j$ be the number of Jordan blocks of size $j$ of $\Up{2}^{\otimes(n-1)}$. By Lemma~\ref{lem:CG}, $\sum_j\kappa_j=\binom{n-1}{\lfloor(n-1)/2\rfloor}$, and $\sum_jj\kappa_j=2^{n-1}$.

\subsection{Convolution}\label{sec:conv}
Let $\mu\colon\Gm\times\Gm\to\Gm$ be the multiplication. Following Katz \cite[Ch.~2]{Katz-Conv}, for perverse sheaves \cite{BBD} $\mathcal A,\mathcal B$ on $\Gm\otimes\Fpbar$ put
\[
\mathcal A*_!\mathcal B=R\mu_!(\mathcal A\boxtimes\mathcal B),\qquad\mathcal A*_*\mathcal B=R\mu_*(\mathcal A\boxtimes\mathcal B).
\]
Katz writes $\mathcal P$ for the class of perverse sheaves on $\Gm\otimes\Fpbar$ that have no perverse subobject or quotient of the form $\mathcal L_\chi[1]$. If $\mathcal A\in\mathcal P$, then $\mathcal A*_!\mathcal B$ and $\mathcal A*_*\mathcal B$ are perverse for every perverse sheaf $\mathcal B$ \cite[Ch.~2]{Katz-Conv}. The \emph{middle convolution} $\mathcal A*_{\mathrm{mid}}\mathcal B$ is the image of the natural morphism $\mathcal A*_!\mathcal B\to\mathcal A*_*\mathcal B$ in the abelian category of perverse sheaves \cite[Ch.~2]{Katz-Conv}. The same formulas apply to perverse sheaves on $\Gm$ over a finite field $\Fq$. Their formation commutes with extension of scalars to $\Fpbar$: for $R\mu_!$ by proper base change, for $R\mu_*$ because $\Fpbar$ is a filtered union of finite, hence \'etale, extensions of $\Fq$; and extension of scalars is t-exact for the perverse t-structure.

We use the following facts about perverse sheaves \cite{BBD} on a smooth connected curve $C$ over $\Fq$ or $\Fpbar$.
\begin{itemize}
\item[(P1)] If $F$ is lisse on $C$, then $F[1]$ is perverse, and $\Hom(F[1],G[1])=\Hom(F,G)$ for lisse $F$ and $G$. For a morphism $f$ of lisse sheaves, the kernel and the cokernel of $f[1]$ in the category of perverse sheaves are $(\ker f)[1]$ and $(\operatorname{coker}f)[1]$.
\item[(P2)] A complex is \emph{punctual} if it is supported on finitely many closed points. A complex $\mathcal A$ on $C$ is perverse if and only if $\mathcal H^i(\mathcal A)=0$ for $i\notin\{-1,0\}$, $\mathcal H^0(\mathcal A)$ is punctual, and $\mathcal H^{-1}(\mathcal A)$ has no nonzero section supported on a finite set. The truncation triangle then gives an exact sequence $0\to\mathcal H^{-1}(\mathcal A)[1]\to\mathcal A\to\mathcal H^0(\mathcal A)\to0$ of perverse sheaves. Every perverse sheaf $\mathcal A$ on $C$ is of the form $F[1]$, with $F$ lisse, on some dense open subset of $C$; the rank of $F$ is the \emph{generic rank} of $\mathcal A$. The generic rank is additive in short exact sequences, and a perverse sheaf of generic rank $0$ is punctual. Restriction to an open subset is exact on perverse sheaves.
\item[(P3)] Over $\Fpbar$, a punctual perverse sheaf $\mathcal N$ is a finite direct sum of sheaves $i_{x*}V$, with $V$ finite-dimensional. So if $\mathcal N\ne0$, there are nonzero morphisms $\delta_x\to\mathcal N$ and $\mathcal N\to\delta_x$, with $\delta_x=i_{x*}\Qlb$, for some closed point $x$. For lisse $G$, $\Hom(\delta_x,G[1])=H^1(i_x^!G)=H^1\big(G_x(-1)[-2]\big)=0$, and $\Hom(G[1],\delta_x)=\Hom(G,\delta_x[-1])=0$. Hence $G[1]$ has no nonzero punctual subobject or quotient. This holds over $\Fpbar$, and therefore also over $\Fq$, since extension of scalars is exact and faithful. In particular, if $G$ is lisse of rank one, then $G[1]$ is simple. Indeed, by what precedes, a nonzero perverse subobject $\mathcal N$ of $G[1]$ is not punctual, so by (P2) it has generic rank $1$. Then $G[1]/\mathcal N$ has generic rank $0$, so by (P2) it is a punctual quotient of $G[1]$, and it vanishes.
\item[(P4)] Let $j\colon U\hookrightarrow C$ be the inclusion of a dense open subset, $F_U$ a lisse sheaf on $U$, and $\mathcal K$ a perverse sheaf on $C$ with $\mathcal K|_U\cong F_U[1]$.
\begin{enumerate}[(a)]
\item If $\mathcal K$ has no nonzero punctual subobject or quotient, then $\mathcal K\cong(j_*F_U)[1]$.
\item Suppose that the base field is $\Fpbar$, that $\mathcal K=(j_*F_U)[1]$, and that $L$ is a lisse sheaf of rank one on $C$. If $\mathcal K$ has no subobject and no quotient isomorphic to $L[1]$, then $F_U$ has no subsheaf and no quotient sheaf isomorphic to $L|_U$.
\end{enumerate}
For (a), (P2) gives an exact sequence $0\to\mathcal H^{-1}(\mathcal K)[1]\to\mathcal K\to\mathcal H^0(\mathcal K)\to0$ with $\mathcal H^0(\mathcal K)$ punctual, so $\mathcal H^0(\mathcal K)=0$. Thus $\mathcal K=F[1]$, where $F=\mathcal H^{-1}(\mathcal K)$ has no nonzero section supported on a finite set and satisfies $F|_U=F_U$. The adjunction map $F\to j_*F_U$ is therefore injective, with punctual cokernel $C'$. The triangle $C'\to F[1]\to(j_*F_U)[1]\to C'[1]$ is then an exact sequence $0\to C'\to F[1]\to(j_*F_U)[1]\to0$ of perverse sheaves, so $C'=0$ by hypothesis, and $\mathcal K=(j_*F_U)[1]$.

For (b), a nonzero morphism $L|_U\to F_U$ corresponds by adjunction to a nonzero morphism $L\to j_*F_U$, that is, to a nonzero morphism $L[1]\to\mathcal K$. Since $L[1]$ is simple (P3), this is a monomorphism, which is excluded. A nonzero morphism $F_U\to L|_U$ corresponds by adjunction to a nonzero morphism $\phi\colon j_!F_U[1]\to L[1]$. The adjunction map $j_!F_U\to j_*F_U$ is injective with punctual cokernel $C''$, so as in (a) there is an exact sequence $0\to C''\to j_!F_U[1]\to\mathcal K\to0$ of perverse sheaves. The image of $C''$ under $\phi$ is a punctual subobject of $L[1]$, hence zero by (P3). So $\phi$ factors through a nonzero morphism $\mathcal K\to L[1]$, an epimorphism because $L[1]$ is simple, which is excluded. Hence there is no nonzero morphism $L|_U\to F_U$ and no nonzero morphism $F_U\to L|_U$; a fortiori $F_U$ has no subsheaf and no quotient sheaf isomorphic to $L|_U$.
\end{itemize}

\begin{lemma}\label{lem:conv-model}
Let $\mathcal N_c$ be the restriction of the sheaf $\mathcal N$ of Theorem~\ref{thm:lisse}(i) to $\Gmn{t}\times C_c$; thus $\mathcal N_c=[\times h_1]^*\Kl\otimes\pr_t^*\mathcal M_c$ with $[\times h_1](t,h_1)=h_1t$. Then $R\pr_{h_1!}\mathcal N_c\cong\mathcal H_n|_{C_c}[-1]$ and
\[
\mathcal H_n|_{C_c}[1]\cong\Kl[1]*_!\iota^*\mathcal M_c[1].
\]
\end{lemma}

\begin{proof}
By Theorem~\ref{thm:lisse}(i) and its proof, $R\pr_{S!}\mathcal N\cong R\pi^\circ_!\mathcal L[n]$, and by Theorem~\ref{thm:lisse}(ii) the complex $R\pi^\circ_!\mathcal L$ is $\mathcal H_n[-n-1]$; so $R\pr_{S!}\mathcal N\cong\mathcal H_n[-1]$. Restriction to $C_c$ by proper base change gives the first isomorphism. The morphism $\Phi\colon\Gmn{t}\times\Gmn{h_1}\to\Gm\times\Gm$, $\Phi(t,h_1)=(h_1t,t^{-1})$, is an isomorphism with inverse $(x,y)\mapsto(y^{-1},xy)$; it satisfies $\mu\circ\Phi=\pr_{h_1}$ and $\Phi^*(\Kl\boxtimes\iota^*\mathcal M_c)=\mathcal N_c$. Hence $R\pr_{h_1!}\mathcal N_c\cong R\mu_!(\Kl\boxtimes\iota^*\mathcal M_c)$. Since $\Kl[1]\boxtimes\iota^*\mathcal M_c[1]=(\Kl\boxtimes\iota^*\mathcal M_c)[2]$, we obtain $\Kl[1]*_!\iota^*\mathcal M_c[1]\cong R\pr_{h_1!}\mathcal N_c[2]\cong\mathcal H_n|_{C_c}[1]$.
\end{proof}

\begin{lemma}\label{lem:local-data}
\begin{enumerate}[(i)]
\item $\Kl$ is lisse on $\Gm$ and pure of weight $1$. It is tame at $0$, with $\Kl(0)\cong\Up{2}$, and totally wild at $\infty$, with $\Swan_\infty(\Kl)=1$.
\item $\iota^*\mathcal M_c$ is lisse on $\Gm$ and pure of weight $n-1$. It is totally wild at $0$, with all breaks equal to $1$ and $\Swan_0=2^{n-1}$, and tame at $\infty$, with $(\iota^*\mathcal M_c)(\infty)\cong\Up{2}^{\otimes(n-1)}$.
\item Neither $\Kl$ nor $\iota^*\mathcal M_c$ has, geometrically, a Kummer sheaf as a subsheaf, quotient sheaf or Jordan--H\"older constituent. Consequently $\Kl[1]$ and $\iota^*\mathcal M_c[1]$ have no Kummer subobject, quotient or constituent in the category of perverse sheaves.
\end{enumerate}
\end{lemma}

\begin{proof}
(i) is (K1)--(K3). For (ii), $\mathcal M_c$ is pure of weight $n-1$ by (K1). At $0$ it is tame with local monodromy $\Up{2}^{\otimes(n-1)}$, and at $\infty$ all its breaks equal $1$ by Lemma~\ref{lem:breaks}, exactly as in the proof of Proposition~\ref{prop:fibration}(ii). The inversion $\iota$ exchanges $0$ and $\infty$, preserves purity, and carries the local monodromy of $\mathcal M_c$ at one point to that of $\iota^*\mathcal M_c$ at the other. It maps a topological generator of tame inertia to the inverse of one, and a unipotent automorphism and its inverse have the same Jordan blocks.

(iii) A Kummer sheaf is tame at $0$ and at $\infty$. Let $F$ be $\Kl$ or $\iota^*\mathcal M_c$, and let $x$ be a point at which $F$ is totally wild: $x=\infty$ for $\Kl$ and $x=0$ for $\iota^*\mathcal M_c$. Since $\Iw_x$ acts semisimply, no nonzero subquotient of $F(x)$ has nonzero $\Iw_x$-invariants. Hence no nonzero subquotient of $F$, in particular no Kummer sheaf, is tame at $x$. For the perverse statement, let $G$ run over the Jordan--H\"older constituents of $F$. By (P1) and induction on the length of $F$, the perverse sheaf $F[1]$ is a successive extension of the $G[1]$. Each $G[1]$ is simple, being the intermediate extension of an irreducible lisse sheaf placed in degree $-1$ \cite{BBD}. So the Jordan--H\"older constituents of $F[1]$ are the $G[1]$, and by (P1) none of them is of the form $\mathcal L_\chi[1]$.
\end{proof}

\subsection{Tameness at the finite boundary point}

\begin{proposition}\label{prop:tame}
The sheaves $\mathcal H_n|_{C_c}$, $\mathcal G_n|_{C_c}$ and $\mathcal W_n|_{C_c}$ are tamely ramified at $h_1=0$.
\end{proposition}

\begin{proof}
\emph{Reduction.} By Theorem~\ref{thm:lisse}(iii),(iv), there is an exact sequence $0\to\Qlb\to\mathcal H_n\to\mathcal G_n\to0$ with constant first term, and $\mathcal W_n$ is a quotient of $\mathcal G_n$. Since $\Iw_0$ acts semisimply, a representation on which it acts trivially on a subrepresentation and on the quotient is trivial, and trivial actions pass to subquotients. It therefore suffices to show that $\mathcal G_n|_{C_c}$ is tame at $h_1=0$.

\emph{A Fourier model.} Write $u'=(u_2,\dots,u_n)$, $f'(u')=\sum_{i\ge2}u_i^{-1}$, $B(u')=-\lambda-\sum_{i\ge2}c_iu_i$, $U=\{u'\in\Gm^{n-1}:B(u')\ne0\}$ and $b=B^{-1}\colon U\to\Gmn{v}$. Let
\[
\mathcal Y=\Big\{(h_1,u)\in\Gm\times\Gm^n:\ h_1u_1+\sum_{i\ge2}c_iu_i=-\lambda\Big\},
\]
with projection $\pi_{\mathcal Y}$ to $\Gmn{h_1}$. By Lemma~\ref{lem:integrate-t} in its relative form over $C_c$, Theorem~\ref{thm:lisse}(iii) and proper base change,
\begin{equation}\label{eq:G-on-line}
R\pi_{\mathcal Y!}\Lpsi f\cong\mathcal G_n|_{C_c}(1)[-(n-1)].
\end{equation}
The map $(h_1,u)\mapsto(h_1,u')$ is an isomorphism $\mathcal Y\to\Gmn{h_1}\times U$, with inverse given by $u_1=B(u')/h_1$, and it carries $f$ to $h_1b(u')+f'(u')$. Put $\mathcal R=Rb_!\Lpsi{f'}$ on $\Gmn{v}$, let $j_v\colon\Gmn{v}\hookrightarrow\mathbb A^1_v$ be the inclusion, and let
\[
\FT_\psi(K)=R\pr_{2!}\big(\pr_1^*K\otimes\Lpsi{vh_1}\big)[1]
\]
be the Fourier transform on $\mathbb A^1_v\times\mathbb A^1_{h_1}$, in Laumon's normalization \cite[(1.2.1.1)]{Laumon}. The projection formula and proper base change for $b\times\mathrm{id}$ give $R\pi_{\mathcal Y!}\Lpsi f\cong\FT_\psi(j_{v!}\mathcal R)[-1]|_{\Gmn{h_1}}$. Therefore, by \eqref{eq:G-on-line},
\begin{equation}\label{eq:FT-model}
\FT_\psi(j_{v!}\mathcal R)|_{\Gmn{h_1}}\cong\mathcal G_n|_{C_c}(1)[2-n].
\end{equation}

\emph{The complex $\mathcal R$ near $v=\infty$.} Let $w=v^{-1}$, and let $D_w=\mathbb A^1_w\smallsetminus\{-\lambda\}$, a neighbourhood of $w=0$. On $D_w$ the map $a(w)=\big(\lambda c_i/(\lambda+w)\big)_{i\ge2}$ is a morphism $D_w\to\Gm^{n-1}$ with $a(0)=c$. In the coordinate $w$, the fibre of $b$ over $v=w^{-1}$ is $\{u'\in\Gm^{n-1}:\sum_{i\ge2}c_iu_i=-(\lambda+w)\}$. The scheme $\{(w,u'):w\in D_w,\ \sum_{i\ge2}c_iu_i=-(\lambda+w)\}$ is the inverse image under $a$ of the universal family
\[
\mathcal Y_{n-1}=\{(h',u')\in\Gm^{n-1}\times\Gm^{n-1}:\langle h',u'\rangle=-\lambda\},
\]
because $\sum_i\lambda c_i(\lambda+w)^{-1}u_i=-\lambda$ is the same equation; the phase is $f'$ on both. Suppose first $n\ge3$. By Lemma~\ref{lem:integrate-t} and Theorem~\ref{thm:lisse}(iii) for $n-1\ge2$ variables, the direct image of $\Lpsi{f'}$ along $\mathcal Y_{n-1}\to\Gm^{n-1}$ is $\mathcal G_{n-1}(1)[-(n-2)]$, a lisse sheaf placed in one degree. By proper base change along $a$, the restriction of $\mathcal R$ to the punctured neighbourhood $D_w\smallsetminus\{0\}$ of $v=\infty$ is the restriction of the lisse sheaf $a^*\mathcal G_{n-1}(1)$ on $D_w$, placed in degree $n-2$. If $n=2$, the fibre is the single point $u_2=-(\lambda+w)/c_2$, and the restriction of $\mathcal R$ to $D_w\smallsetminus\{0\}$ is that of $\Lpsi{-c_2/(\lambda+w)}$, which is lisse on $D_w$. In both cases $\mathcal R$ is, near $v=\infty$, a lisse sheaf in a single degree on which $I_\infty$ acts trivially.

\emph{Laumon's local Fourier transform.} We use three facts from \cite{Laumon}, valid for every perverse sheaf $P$ on $\mathbb A^1_v$. First, $\FT_\psi(P)$ is perverse \cite[(1.3.2.3)]{Laumon}. Second, the vanishing cycles of $\FT_\psi(P)$ at $h_1=0$ depend, as a representation of $I_0$, only on the representation $V$ of $I_\infty$ on $\mathcal H^{-1}(P)_{\bar\eta_\infty}$: there is an isomorphism of $I_0$-modules
\[
R^{-1}\Phi_0\big(\FT_\psi(P)\big)\cong\mathcal F^{(\infty,0')}(V),
\]
where $\mathcal F^{(\infty,0')}$ is Laumon's local Fourier transform \cite[Prop.~(2.3.2.1)(iii) and D\'ef.~(2.4.2.3)]{Laumon}. Third, if all slopes of $V$ are $<1$, then $\Swan\big(\mathcal F^{(\infty,0')}(V)\big)=\Swan(V)$ \cite[Th.~(2.4.3)(ii)(b)]{Laumon}. Shifts and Tate twists do not change the action of inertia, so the normalization of $\FT_\psi$ plays no further role. Now let $P_i={}^p\mathcal H^i(j_{v!}\mathcal R)$ be the perverse cohomology sheaves, so that $j_{v!}\mathcal R$ is a successive extension of the $P_i[-i]$. By the previous step, the $I_\infty$-module $V_i=\mathcal H^{-1}(P_i)_{\bar\eta_\infty}$ is trivial for every $i$. In particular it has slope $0$, so $\Swan\big(\mathcal F^{(\infty,0')}(V_i)\big)=\Swan(V_i)=0$, and $\Iw_0$ acts trivially on the vanishing cycles $R^{-1}\Phi_0\big(\FT_\psi(P_i)\big)$. On a small punctured neighbourhood of $0$, the perverse sheaf $\FT_\psi(P_i)$ is a lisse sheaf placed in degree $-1$; let $\mathcal V_i$ be its generic stalk, an $I_0$-module. The triangle formed by the stalk at $0$, the nearby cycles and the vanishing cycles gives an exact sequence of $I_0$-modules
\[
\mathcal H^{-1}\big(\FT_\psi(P_i)\big)_0\longrightarrow\mathcal V_i\longrightarrow R^{-1}\Phi_0\big(\FT_\psi(P_i)\big),
\]
in which $I_0$ acts trivially on the image of the first arrow. As $\Iw_0$ acts semisimply on $\mathcal V_i$, and trivially both on the image of the first arrow and on the third term, it acts trivially on $\mathcal V_i$. Thus every $\FT_\psi(P_i)$ is tame at $0$. Since $\FT_\psi$ is a triangulated functor, $\FT_\psi(j_{v!}\mathcal R)$ is a successive extension of the $\FT_\psi(P_i)[-i]$. On a small punctured neighbourhood of $0$ all cohomology sheaves involved are lisse, so the long exact cohomology sequences and the semisimplicity of $\Iw_0$-actions show that every cohomology sheaf of $\FT_\psi(j_{v!}\mathcal R)$ is tame at $0$. By \eqref{eq:FT-model}, $\mathcal G_n|_{C_c}$ is tame at $h_1=0$.
\end{proof}

\subsection{Unipotent parts of a middle convolution}

\begin{definition}\label{def:jordan-poly}
Let $\mathcal A$ be a perverse sheaf on $\Gm\otimes\Fpbar$, let $F$ be a lisse sheaf on a dense open subset $U\subseteq\Gm$ with $\mathcal A|_U\cong F[1]$, and let $\chi$ be a character of finite order prime to $p$. For $x\in\{0,\infty\}$ and $i\ge1$, let $b^x_i(\mathcal A,\chi)$ be the number of Jordan blocks of size $i$ in the unipotent part of $(\mathcal L_\chi\otimes F)(x)$. Define $\mathsf P_{\mathcal A,\chi}\in\mathbb Z[T,T^{-1}]$ by
\[
\mathsf P_{\mathcal A,\chi}(T)=a_0(\mathcal A,\chi)+\sum_{i\ge1}\big(b^0_i(\mathcal A,\chi)\,T^{-i}+b^\infty_i(\mathcal A,\chi)\,T^{i}\big),
\]
where $a_0(\mathcal A,\chi)$ is chosen so that $\mathsf P_{\mathcal A,\chi}(1)=\chi_c(\Gm\otimes\Fpbar,\mathcal A)$. This Euler characteristic does not change if $\mathcal A$ is replaced by $\mathcal L_\chi\otimes\mathcal A$, since a tame twist preserves all local invariants.
\end{definition}

If $\mathcal A=F[1]$ with $F$ lisse on all of $\Gm$, then \eqref{eq:GOS} gives $\mathsf P_{\mathcal A,\chi}(1)=\Swan_0(F)+\Swan_\infty(F)$.

\begin{lemma}\label{lem:weight-dictionary}
Let $U\subseteq\Gm$ be a dense open subset over a finite field, and let $F_U$ be a lisse sheaf on $U$, pure of weight $-1$, without geometrically trivial subsheaf or quotient sheaf. Let $F$ be the direct image of $F_U$ on $\Gm$, let $j_\infty\colon\Gm\hookrightarrow\mathbb P^1\smallsetminus\{0\}$ and $j_0\colon\mathbb P^1\smallsetminus\{0\}\hookrightarrow\mathbb P^1$ be the inclusions, and put
\[
\omega(F[1])=H^0\big(\mathbb P^1\otimes\Fpbar,\,j_{0!}Rj_{\infty*}F[1]\big).
\]
Then $\dim\omega(F[1])=\chi_c(\Gm\otimes\Fpbar,F[1])$. For $i\ge1$, the number of Frobenius eigenvalues of weight $-i$ on $\omega(F[1])$ is $b^0_i(F[1],\mathbf 1)$ and the number of weight $+i$ is $b^\infty_i(F[1],\mathbf 1)$; all other eigenvalues have weight $0$. Equivalently, $\sum_w\big(\text{number of eigenvalues of weight }w\big)\,T^w=\mathsf P_{F[1],\mathbf 1}(T)$.
\end{lemma}

\begin{proof}
Put $A=\mathbb P^1\smallsetminus\{0\}$, so that $\omega(F[1])=H^1_c\big(A\otimes\Fpbar,Rj_{\infty*}F\big)$. The complex $Rj_{\infty*}F$ has $j_{\infty*}F$ in degree $0$ and, in degree $1$, the skyscraper at $\infty$ with value $H^1(I_\infty,F(\infty))=F(\infty)_{I_\infty}(-1)$. The triangle $j_{\infty!}F\to Rj_{\infty*}F\to i_{\infty*}(Rj_{\infty*}F)_\infty$ gives an exact sequence
\[
\begin{aligned}
0&\to H^0_c(\Gm,F)\to H^0_c(A,Rj_{\infty*}F)\to F(\infty)^{I_\infty}\to H^1_c(\Gm,F)\\
&\to H^1_c(A,Rj_{\infty*}F)\to F(\infty)_{I_\infty}(-1)\to H^2_c(\Gm,F).
\end{aligned}
\]
The groups $H^0_c(\Gm,F)$ and $H^0_c(A,Rj_{\infty*}F)$ vanish, since $F$ has no nonzero sections supported at points and $A$ is affine. The group $H^2_c(\Gm,F)$, the geometric coinvariants of $F_U$ twisted by $(-1)$, vanishes by hypothesis. Next, there is an exact sequence
\[
0\to F(0)^{I_0}\oplus F(\infty)^{I_\infty}\to H^1_c(\Gm,F)\to H^1(\mathbb P^1,j_*F)\to0,
\]
obtained as in Proposition~\ref{prop:weights-torus} using $H^0(\Gm,F)=0$, whose right-hand term is pure of weight $0$ by \cite[3.2.3]{Deligne}. Combining the two sequences, $\omega(F[1])$ has a filtration with graded pieces $F(0)^{I_0}$, $H^1(\mathbb P^1,j_*F)$ and $F(\infty)_{I_\infty}(-1)$. In particular its dimension is $\dim H^1_c(\Gm,F)=\chi_c(\Gm,F[1])$. Invariants and coinvariants of $I_x$ come only from the unipotent part of $F(x)$. By \cite[1.8.4]{Deligne}, a Jordan block of size $i$ in the unipotent part of $F(0)$ contributes to $F(0)^{I_0}$ one eigenvalue of weight $-1-(i-1)=-i$. A block of size $i$ in the unipotent part of $F(\infty)$ contributes to $F(\infty)_{I_\infty}$ one eigenvalue of weight $-1+(i-1)$, hence to $F(\infty)_{I_\infty}(-1)$ one eigenvalue of weight $i$.
\end{proof}

Let $j'\colon\Gm\hookrightarrow\mathbb A^1=\mathbb P^1\smallsetminus\{\infty\}$ and $j''\colon\mathbb A^1\hookrightarrow\mathbb P^1$ be the inclusions. As $j_{0!}$ and $Rj_{\infty*}$ concern the distinct points $0$ and $\infty$, we have $j_{0!}Rj_{\infty*}\cong Rj''_*j'_!$, and therefore $\omega(F[1])=H^0\big(\mathbb A^1\otimes\Fpbar,j'_!F[1]\big)$. This is the functor that Katz, following Deligne, shows to be a fibre functor \cite[Th.~3.1]{Katz-Conv}.

\begin{theorem}[Rojas-Le\'on]\label{thm:RL}
Let $k$ be a finite field of characteristic $p$, and let $\mathcal A$ and $\mathcal B$ be pure perverse sheaves on $\Gm$ over $k$. Suppose that each of them is the middle extension of a lisse sheaf on a dense open subset, placed in degree $-1$, and that neither has, over $\bar k$, a subobject or quotient of the form $\mathcal L_\rho[1]$. Suppose moreover that $\mathcal A*_{\mathrm{mid}}\mathcal B$ has, over $\bar k$, no nonzero punctual subobject or quotient. Then for every character $\chi$ of finite order prime to $p$,
\[
\mathsf P_{\mathcal A*_{\mathrm{mid}}\mathcal B,\chi}=\mathsf P_{\mathcal A,\chi}\cdot\mathsf P_{\mathcal B,\chi}.
\]
\end{theorem}

\begin{proof}
A product formula of this kind was proved by Rojas-Le\'on \cite[Prop.~6.1]{RojasLeon-conv} with his local convolution functors. We give a proof based on Deligne's fibre functor, which also fixes the normalization.

\emph{Reductions.} Both sides depend only on $\mathcal A$ and $\mathcal B$ over $\bar k$, and the formation of $\mathcal A*_{\mathrm{mid}}\mathcal B$ commutes with extension of scalars (\S\ref{sec:conv}). The hypotheses are preserved by finite extensions of $k$: purity is, the formation of $j_*$ commutes with \'etale base change, and the remaining conditions are geometric. If $\chi$ has order $m$, we may therefore assume that $\mu_m\subset k$. Then $\chi$ is the Kummer character attached to a character of $k^\times$, again denoted $\chi$, and $\mathcal L_\chi$ is a lisse sheaf of rank one on $\Gm$ over $k$, pure of weight $0$. Next, for $\alpha\in\Qlb^\times$ let $\alpha^{\deg}$ be the geometrically constant sheaf of rank one on which $\Frob_k$ acts by $\alpha$. The projection formula gives $(\mathcal A\otimes\alpha^{\deg})*_{\mathrm{mid}}(\mathcal B\otimes\beta^{\deg})\cong(\mathcal A*_{\mathrm{mid}}\mathcal B)\otimes(\alpha\beta)^{\deg}$, and these twists change neither side of the formula, because the polynomials $\mathsf P$ are geometric. Choosing $\alpha$ and $\beta$ of suitable absolute values, we may assume that $\mathcal A$ and $\mathcal B$ are pure of weight $0$.

\emph{The fibre functor.} Over $\bar k$, $\mathcal A$ and $\mathcal B$ lie in Katz's class $\mathcal P$ (\S\ref{sec:conv}); that is, they are objects of the category $\mathcal P_{\mathrm{arith}}$ of \cite[Ch.~4]{Katz-Conv}. The objects of $\mathcal P_{\mathrm{arith}}$ that are pure of weight $0$ are stable under middle convolution \cite[Ch.~4, paragraph after Th.~4.1]{Katz-Conv}. So $\mathcal K=\mathcal A*_{\mathrm{mid}}\mathcal B$ lies in $\mathcal P_{\mathrm{arith}}$ and is pure of weight $0$; in particular, over $\bar k$ it has no subobject or quotient of the form $\mathcal L_\rho[1]$. Let $j'\colon\Gm\hookrightarrow\mathbb A^1$ be the inclusion, and for a perverse sheaf $\mathcal N$ on $\Gm$ over $k$ put
\[
\omega_\chi(\mathcal N)=H^0\big(\mathbb A^1\otimes\bar k,\,j'_!(\mathcal L_\chi\otimes\mathcal N)\big),
\]
on which $\Frob_k$ acts. By \cite[Th.~3.1 and Cor.~3.2]{Katz-Conv}, $\omega_\chi$ is a fibre functor, for middle convolution, on Katz's Tannakian category $\mathcal P$. By \cite[beginning of Ch.~5]{Katz-Conv}, for every character $\chi$ of $k^\times$ the action of $\Frob_k$ on $\omega_\chi$ is an element of the Tannakian group of the category generated by $\mathcal A\oplus\mathcal B$ in $\mathcal P_{\mathrm{arith}}$, that is, an automorphism of the fibre functor compatible with its tensor structure. (The tensor isomorphisms of $\omega_\chi$ are canonical, so they commute with transport of structure by $\mathrm{Gal}(\bar k/k)$, that is, with the action of $\Frob_k$, for every $\chi$. Katz's further condition that $\chi$ be good for $\mathcal A\oplus\mathcal B$ serves to make the eigenvalues unitary, which we do not need: the weights are computed below.) In particular there is an isomorphism $\omega_\chi(\mathcal K)\cong\omega_\chi(\mathcal A)\otimes\omega_\chi(\mathcal B)$ compatible with $\Frob_k$.

\emph{Weights.} Let $\mathcal N$ be one of $\mathcal A$, $\mathcal B$, $\mathcal K$, and put $\mathcal N'=\mathcal L_\chi\otimes\mathcal N$, a perverse sheaf pure of weight $0$. Over $\bar k$, $\mathcal N'$ has no subobject or quotient isomorphic to $\Qlb[1]$, because $\mathcal N$ has none isomorphic to $\mathcal L_{\chi^{-1}}[1]$. Next, there are a dense open subset $U\subseteq\Gm$, with inclusion $j\colon U\hookrightarrow\Gm$, and a lisse sheaf $G_U$ on $U$ with $\mathcal N'\cong(j_*G_U)[1]$. For $\mathcal N=\mathcal A$ or $\mathcal B$ this holds by hypothesis, with $G_U=\mathcal L_\chi|_U\otimes F_U$ if $\mathcal N=(j_*F_U)[1]$, because tensoring with the lisse sheaf $\mathcal L_\chi$ commutes with $j_*$. For $\mathcal N=\mathcal K$, the hypothesis on $\mathcal A*_{\mathrm{mid}}\mathcal B$ shows that $\mathcal N'=\mathcal L_\chi\otimes\mathcal K$ has no nonzero punctual subobject or quotient over $\bar k$, hence none over $k$, since extension of scalars is exact and faithful. By (P2), $\mathcal N'$ is of the form $G_U[1]$ on some dense open subset $U$, and (P4)(a) gives $\mathcal N'\cong(j_*G_U)[1]$. On $U$ we have $\mathcal N'=G_U[1]$ and $D(G_U[1])=G_U^\vee(1)[1]$, so purity of $\mathcal N'$ of weight $0$ means that $G_U$ is pointwise pure of weight $-1$. By (P4)(b), applied over $\bar k$ with $L=\Qlb$, the sheaf $G_U$ has no geometrically trivial subsheaf or quotient sheaf. Thus Lemma~\ref{lem:weight-dictionary} applies to $G_U$. By the remark following that lemma, its functor $\omega$ satisfies $\omega(\mathcal N')=\omega_\chi(\mathcal N)$, and $\mathsf P_{\mathcal N',\mathbf 1}=\mathsf P_{\mathcal N,\chi}$ by Definition~\ref{def:jordan-poly}. Hence
\[
\mathsf P_{\mathcal N,\chi}(T)=\sum_w\big(\text{number of eigenvalues of }\Frob_k\text{ on }\omega_\chi(\mathcal N)\text{ of weight }w\big)\,T^w .
\]
The eigenvalues of $\Frob_k$ on $\omega_\chi(\mathcal A)\otimes\omega_\chi(\mathcal B)$ are the products of those on the two factors, counted with multiplicity, and weights add. So the right-hand side is multiplicative in tensor products, and the isomorphism $\omega_\chi(\mathcal K)\cong\omega_\chi(\mathcal A)\otimes\omega_\chi(\mathcal B)$ gives the theorem.
\end{proof}

\begin{lemma}\label{lem:binomial}
For every $n\ge2$, $2r(n-1)+\kappa_1=r(n)$.
\end{lemma}

\begin{proof}
Since $r(m)=2^m-\binom m{\lfloor m/2\rfloor}$, the claim is that $2\binom{n-1}{\lfloor(n-1)/2\rfloor}-\kappa_1=\binom n{\lfloor n/2\rfloor}$. If $n=2m$, then $\kappa_1=0$ by Lemma~\ref{lem:CG}, and $\binom{2m}m=\binom{2m-1}{m-1}+\binom{2m-1}m=2\binom{2m-1}{m-1}$. If $n=2m+1$, then $\kappa_1=\binom{2m}m-\binom{2m}{m-1}$ by Lemma~\ref{lem:CG}, and $2\binom{2m}m-\kappa_1=\binom{2m}m+\binom{2m}{m-1}=\binom{2m+1}m$.
\end{proof}

\begin{theorem}\label{thm:boundary}
Let $\mathcal K=\Kl[1]*_{\mathrm{mid}}\iota^*\mathcal M_c[1]$.
\begin{enumerate}[(i)]
\item $\mathcal W_n|_{C_c}[1]\cong\mathcal K$.
\item $\mathcal W_n|_{C_c}$ is tame at $h_1=0$, and its local monodromy there is unipotent, with $r(n-1)$ Jordan blocks of size $2$, $\kappa_1$ blocks of size $1$, and no block of size $\ge3$.
\item The tame part of $(\mathcal W_n|_{C_c})(\infty)$ is unipotent, with $\kappa_j$ Jordan blocks of size $j-2$ for each $j\ge3$, and $\Swan_\infty(\mathcal W_n|_{C_c})=2^{n-1}$.
\end{enumerate}
In particular Theorem~\ref{thmC}(i) and~(iv) hold, as well as the assertions of Theorem~\ref{thmC}(ii) on the tame part and on the Swan conductor.
\end{theorem}

\begin{proof}
We work with perverse sheaves on $C_c$ over $\Fq$. Lemma~\ref{lem:conv-model} and the definition of $*_{\mathrm{mid}}$ give morphisms
\[
\mathcal H_n|_{C_c}[1]\cong\Kl[1]*_!\iota^*\mathcal M_c[1]\twoheadrightarrow\mathcal K\hookrightarrow\Kl[1]*_*\iota^*\mathcal M_c[1].
\]
By Theorem~\ref{thm:lisse}(iv) and (P1), there is also an epimorphism $\mathcal H_n|_{C_c}[1]\twoheadrightarrow\mathcal W_n|_{C_c}[1]$ with kernel $\mathcal H_n^{\le n}|_{C_c}[1]$. Step~0 shows that $\mathcal K$ has no nonzero punctual subobject or quotient, and Steps~1 and~2 compute the local monodromy of $\mathcal K$. Step~3 shows that $\mathcal H_n|_{C_c}[1]\twoheadrightarrow\mathcal K$ factors through $\mathcal W_n|_{C_c}[1]$, and Step~4 that the resulting epimorphism $\mathcal W_n|_{C_c}[1]\twoheadrightarrow\mathcal K$ is an isomorphism.

\emph{Step 0: no punctual subobject or quotient.} We argue over $\Fpbar$, using that the formation of $\mathcal K$ commutes with extension of scalars (\S\ref{sec:conv}). The sheaf $\mathcal H_n$ is lisse, so by (P3) the perverse sheaf $\mathcal H_n|_{C_c}[1]$ has no nonzero punctual quotient; hence neither has its quotient $\mathcal K$. For subobjects, since $\mathcal K\subseteq\Kl[1]*_*\iota^*\mathcal M_c[1]$, by (P3) it suffices to show that $\Hom(\delta_x,\Kl[1]*_*\iota^*\mathcal M_c[1])=0$ for every closed point $x$ of $\Gm\otimes\Fpbar$. Adjunction gives
\[
\Hom\big(\delta_x,R\mu_*(\Kl\boxtimes\iota^*\mathcal M_c)[2]\big)\cong H^0\big(\mu^{-1}(x),(\Kl\boxtimes\iota^*\mathcal M_c)|_{\mu^{-1}(x)}\big)(-1).
\]
This vanishes, because $\mu^{-1}(x)\cong\Gm$ and the restricted sheaf $y\mapsto\Kl(y)\otimes\mathcal M_c(y/x)$ is totally wild at $y=\infty$ by Lemma~\ref{lem:breaks}. Thus $\mathcal K$ has no nonzero punctual subobject or quotient over $\Fpbar$. Moreover $\Kl[1]$ and $\iota^*\mathcal M_c[1]$ are lisse sheaves on $\Gm$ placed in degree $-1$, pure as perverse sheaves by Lemma~\ref{lem:local-data}(i),(ii), and by Lemma~\ref{lem:local-data}(iii) neither has a subobject or quotient of the form $\mathcal L_\rho[1]$ over $\Fpbar$. So the pair $\Kl[1]$, $\iota^*\mathcal M_c[1]$ satisfies the hypotheses of Theorem~\ref{thm:RL}.

\emph{Step 1: the two polynomials.} Take first $\chi=\mathbf 1$. By Lemma~\ref{lem:local-data}(i), $b^0_2(\Kl[1],\mathbf 1)=1$ and all other $b$ vanish, while $\chi_c(\Gm,\Kl[1])=\Swan_\infty(\Kl)=1$; so $\mathsf P_{\Kl[1],\mathbf 1}=T^{-2}$. By Lemma~\ref{lem:local-data}(ii), $b^\infty_j(\iota^*\mathcal M_c[1],\mathbf 1)=\kappa_j$ and $b^0_j(\iota^*\mathcal M_c[1],\mathbf 1)=0$, while $\chi_c(\Gm,\iota^*\mathcal M_c[1])=\Swan_0(\iota^*\mathcal M_c)=2^{n-1}$. Since $2^{n-1}-\sum_j\kappa_j=2^{n-1}-\binom{n-1}{\lfloor(n-1)/2\rfloor}=r(n-1)$, this gives
\[
\mathsf P_{\iota^*\mathcal M_c[1],\mathbf 1}(T)=r(n-1)+\sum_{j\ge1}\kappa_jT^j.
\]
For $\chi\ne\mathbf 1$, twisting by $\mathcal L_\chi$ kills the unipotent parts at both points and does not change Swan conductors. Hence $\mathsf P_{\Kl[1],\chi}=1$ and $\mathsf P_{\iota^*\mathcal M_c[1],\chi}=2^{n-1}$.

\emph{Step 2: the product.} By Theorem~\ref{thm:RL},
\[
\mathsf P_{\mathcal K,\mathbf 1}(T)=r(n-1)\,T^{-2}+\kappa_1T^{-1}+\kappa_2+\sum_{j\ge3}\kappa_jT^{j-2},\qquad\mathsf P_{\mathcal K,\chi}=2^{n-1}\quad(\chi\ne\mathbf 1).
\]
Let $F_{\mathcal K}$ be the lisse sheaf on a dense open subset of $\Gm$ with $\mathcal K\cong F_{\mathcal K}[1]$ there. Let $x\in\{0,\infty\}$, and write the tame part of $F_{\mathcal K}(x)$ as $\bigoplus_\rho\mathcal L_\rho\otimes U_\rho$ (\S\ref{sec:conventions}). Every character $\rho\ne\mathbf 1$ of finite order prime to $p$ of the tame inertia at $x$ is the local monodromy at $x$ of a Kummer sheaf $\mathcal L_{\chi^{-1}}$ with $\chi\ne\mathbf 1$. The unipotent part of $(\mathcal L_\chi\otimes F_{\mathcal K})(x)$ is then $U_\rho$. Since $\mathsf P_{\mathcal K,\chi}$ is constant, all $b^x_i(\mathcal K,\chi)$ vanish, so $U_\rho=0$. Hence the tame parts of $F_{\mathcal K}(0)$ and $F_{\mathcal K}(\infty)$ are unipotent. By Definition~\ref{def:jordan-poly}, for $i\ge1$ the coefficient of $T^{-i}$ in $\mathsf P_{\mathcal K,\mathbf 1}$ is $b^0_i(\mathcal K,\mathbf 1)$ and that of $T^i$ is $b^\infty_i(\mathcal K,\mathbf 1)$. Comparing with the display:
\begin{itemize}
\item $b^0_2(\mathcal K,\mathbf 1)=r(n-1)$, $b^0_1(\mathcal K,\mathbf 1)=\kappa_1$, and $b^0_i(\mathcal K,\mathbf 1)=0$ for $i\ge3$;
\item $b^\infty_i(\mathcal K,\mathbf 1)=\kappa_{i+2}$ for $i\ge1$;
\item the constant term $\kappa_2$ is $a_0(\mathcal K,\mathbf 1)$.
\end{itemize}
Thus the tame part of $F_{\mathcal K}(0)$ consists of $r(n-1)$ Jordan blocks of size $2$, $\kappa_1$ blocks of size $1$ and no others, and that of $F_{\mathcal K}(\infty)$ of $\kappa_j$ blocks of size $j-2$ for each $j\ge3$ and no others. Finally $\chi_c(\Gm,\mathcal K)=\mathsf P_{\mathcal K,\mathbf 1}(1)=2^{n-1}$.

\emph{Step 3: $\mathcal K$ is a quotient of $\mathcal W_n|_{C_c}[1]$.} By Lemma~\ref{lem:conv-model}, $\mathcal K$ is a quotient of $\mathcal H_n|_{C_c}[1]$ in the category of perverse sheaves. We first determine the weights of $\mathcal K$ on a dense open subset. Put $\mathcal L_c=\Kl\boxtimes\iota^*\mathcal M_c$, a lisse sheaf on $\Gm^2$ over $\Fq$, pure of weight $n$, so that $\Kl[1]*_!\iota^*\mathcal M_c[1]=R\mu_!\mathcal L_c[2]$ and $\Kl[1]*_*\iota^*\mathcal M_c[1]=R\mu_*\mathcal L_c[2]$. Let $U\subseteq C_c$ be a dense open subset, defined over $\Fq$, on which these two perverse sheaves and $\mathcal K$ are of the form $A[1]$, $B[1]$ and $F_{\mathcal K}[1]$, with $A$, $B$ and $F_{\mathcal K}$ lisse. By (P1), on $U$ the epimorphism $A[1]\to F_{\mathcal K}[1]$ and the monomorphism $F_{\mathcal K}[1]\to B[1]$ come from a surjection $A\to F_{\mathcal K}$ and an injection $F_{\mathcal K}\to B$ of lisse sheaves. The sheaf $A=R^1\mu_!\mathcal L_c|_U$ is pointwise mixed of weights $\le n+1$ by \cite[Th.~3.3.1]{Deligne}. By Verdier duality, $D(R\mu_*\mathcal L_c[2])\cong R\mu_!\big(\mathcal L_c^\vee(2)[2]\big)$; restricting to $U$ gives $B^\vee(1)\cong R^1\mu_!(\mathcal L_c^\vee)(2)|_U$, which is pointwise mixed of weights $\le-n-3$, again by \cite[Th.~3.3.1]{Deligne}. So $B$ is pointwise mixed of weights $\ge n+1$. As $F_{\mathcal K}$ is a quotient of $A$ and a subsheaf of $B$, it is pointwise pure of weight $n+1$. Now consider the composite
\[
\phi\colon\mathcal H_n^{\le n}|_{C_c}[1]\longrightarrow\mathcal H_n|_{C_c}[1]\longrightarrow\mathcal K.
\]
It is a morphism of perverse sheaves on $C_c$ over $\Fq$: the inclusion $\mathcal H_n^{\le n}\subset\mathcal H_n$ is defined over $\Fp$ (Theorem~\ref{thm:lisse}(iv)), the isomorphism of Lemma~\ref{lem:conv-model} is defined over $\Fq$, and so is the natural morphism $R\mu_!\to R\mu_*$ whose image is $\mathcal K$. Let $U$ be as above. As $\mathcal H_n^{\le n}$ is lisse, $\phi|_U=\phi_0[1]$ for a morphism of lisse sheaves $\phi_0\colon\mathcal H_n^{\le n}|_U\to F_{\mathcal K}$. Let $x\in U$ be a closed point and $\bar x$ a geometric point above it. The stalk map $\phi_{0,\bar x}$ commutes with $\Frob_x$, so it maps the generalized eigenspace of $\Frob_x$ for an eigenvalue $\alpha$ into the generalized eigenspace for $\alpha$, because $(\Frob_x-\alpha)^N\phi_{0,\bar x}=\phi_{0,\bar x}(\Frob_x-\alpha)^N$ for all $N$. Every eigenvalue on the source has weight $\le n$ (Theorem~\ref{thm:lisse}(iv)), and every eigenvalue on the target has weight $n+1$. The two spectra are disjoint, so $\phi_{0,\bar x}=0$. Lisse sheaves on the connected scheme $U$ correspond to continuous representations of $\pi_1(U,\bar x)$ through their stalks at $\bar x$, and a morphism of lisse sheaves is determined by its stalk at $\bar x$. Hence $\phi_0=0$, that is, $\phi|_U=0$. Restriction to $U$ is exact on perverse sheaves (P2), so the image of $\phi$ is a perverse subobject of $\mathcal K$ whose restriction to $U$ vanishes, that is, a perverse sheaf supported on the finite set $C_c\smallsetminus U$. But $\mathcal K$ has no nonzero punctual subobject over $\Fpbar$ (Step~0), hence none over $\Fq$, since extension of scalars is exact and faithful. Hence the composite is zero, and $\mathcal K$ is a quotient of $\mathcal H_n|_{C_c}[1]/\mathcal H_n^{\le n}|_{C_c}[1]\cong\mathcal W_n|_{C_c}[1]$ (Theorem~\ref{thm:lisse}(iv)).

\emph{Step 4: comparison of ranks.} Let $\mathcal N$ be the kernel of the epimorphism $\mathcal W_n|_{C_c}[1]\twoheadrightarrow\mathcal K$ of Step~3. With $U$ as in Step~3, this epimorphism restricts on $U$ to $f[1]$ for a surjection $f\colon\mathcal W_n|_U\to F_{\mathcal K}$ of lisse sheaves, by (P1). By Proposition~\ref{prop:tame}, $\mathcal W_n|_{C_c}$ is tame at $h_1=0$. As $f$ is surjective on generic stalks and equivariant for $I_0$, the representation $F_{\mathcal K}(0)$ is a quotient of $(\mathcal W_n|_{C_c})(0)$, hence tame. By Step~2 it is therefore unipotent, with $r(n-1)$ blocks of size $2$ and $\kappa_1$ blocks of size $1$. So it has dimension $2r(n-1)+\kappa_1=r(n)$ by Lemma~\ref{lem:binomial}. Thus $F_{\mathcal K}$ and $\mathcal W_n|_U$ both have rank $r(n)$, and $f$, a surjection of lisse sheaves of the same rank, is an isomorphism. Hence $\mathcal N|_U=(\ker f)[1]=0$, so $\mathcal N$ is supported on the finite set $C_c\smallsetminus U$, that is, punctual. As $\mathcal N$ is a subobject of $\mathcal W_n|_{C_c}[1]$, it vanishes by (P3). Therefore $\mathcal W_n|_{C_c}[1]\to\mathcal K$ is an isomorphism, which is (i). Statements (ii) and (iii) now follow from Step 2. Moreover, by \eqref{eq:GOS}, $\Swan_\infty(\mathcal W_n|_{C_c})=\chi_c(\Gm,\mathcal K)-\Swan_0(\mathcal W_n|_{C_c})=2^{n-1}$.
\end{proof}

The Swan conductor at $\infty$ can also be obtained without the product formula, from an Euler characteristic computed in two ways.

\begin{proposition}\label{prop:euler-H}
One has $\chi_c(\Gm\otimes\Fpbar,\mathcal H_n|_{C_c})=-2^{n-1}$, and the Swan conductors at $\infty$ of $\mathcal H_n|_{C_c}$ and of $\mathcal W_n|_{C_c}$ are both equal to $2^{n-1}$.
\end{proposition}

\begin{proof}
Let $\mathcal N_c$ be the sheaf of Lemma~\ref{lem:conv-model} on $X_c=\Gmn{t}\times\Gmn{h_1}$, so that $\mathcal N_c(t,h_1)=\Kl(h_1t)\otimes\mathcal M_c(t)$. Compactly supported Euler characteristics are compatible with $R(-)_!$, by the Leray spectral sequence. By Lemma~\ref{lem:conv-model}, $R\pr_{h_1!}\mathcal N_c\cong\mathcal H_n|_{C_c}[-1]$, so $\chi_c(X_c\otimes\Fpbar,\mathcal N_c)=-\chi_c(\Gm\otimes\Fpbar,\mathcal H_n|_{C_c})$. On the other hand, the automorphism $(t,h_1)\mapsto(t,h_1t)$ of $X_c$ commutes with $\pr_t$ and carries $\mathcal N_c$ to $\mathcal M_c\boxtimes\Kl$. By the K\"unneth formula, $R\pr_{t!}\mathcal N_c\cong\mathcal M_c\otimes R\Gamma_c(\Gm\otimes\Fpbar,\Kl)$, and $R\Gamma_c(\Gm\otimes\Fpbar,\Kl)$ is a line placed in degree $1$, by (K1)--(K3) and~\eqref{eq:GOS}. Hence $\chi_c(X_c\otimes\Fpbar,\mathcal N_c)=-\chi_c(\Gm\otimes\Fpbar,\mathcal M_c)$. The sheaf $\mathcal M_c$ is tame at $0$, and all its $2^{n-1}$ breaks at $\infty$ equal $1$, as shown in the proof of Lemma~\ref{lem:local-data}(ii), so $\chi_c(\Gm\otimes\Fpbar,\mathcal M_c)=-2^{n-1}$ by~\eqref{eq:GOS}. This proves the first assertion. By~\eqref{eq:GOS} and Proposition~\ref{prop:tame}, $\Swan_\infty(\mathcal H_n|_{C_c})=2^{n-1}$. Finally, $\mathcal H_n|_{C_c}$ is a successive extension of $\mathcal W_n|_{C_c}$, of the constant sheaf $\Qlb$ and of the sheaves $\gr^{\mathrm W}_{2k}\mathcal G_n|_{C_c}$, which are geometrically constant by Theorem~\ref{thm:lisse}(vi). Swan conductors are additive in exact sequences, and they vanish on geometrically constant sheaves; so $\Swan_\infty(\mathcal W_n|_{C_c})=2^{n-1}$.
\end{proof}

\subsection{The point at infinity}

We keep the notation of the proof of Proposition~\ref{prop:tame}: $u'=(u_2,\dots,u_n)$, $f'(u')=\sum_{i\ge2}u_i^{-1}$, $B(u')=-\lambda-\sum_{i\ge2}c_iu_i$, $b=B^{-1}\colon U\to\Gmn{v}$ and $\mathcal R=Rb_!\Lpsi{f'}$. Put $s=-1/\lambda$.

\begin{lemma}\label{lem:R-global}
There is a lisse sheaf $F$ on $U_v=\Gmn{v}\smallsetminus\{s\}$ with $\mathcal R|_{U_v}\cong F[-(n-2)]$, and $F$ extends to a lisse sheaf on $\mathbb P^1_v\smallsetminus\{0,s\}$. In particular the inertia group at $v=\infty$ acts trivially on the generic stalk of $F$.
\end{lemma}

\begin{proof}
For $v\in U_v$ put $a(v)=\frac{\lambda v}{\lambda v+1}\,c\in\Gm^{n-1}$; this is defined because $v\ne0$ and $\lambda v+1\ne0$. For $u'\in\Gm^{n-1}$, the conditions $u'\in U$ and $b(u')=v$ mean $\sum_{i\ge2}c_iu_i=-(\lambda+v^{-1})$, and this equation is equivalent to $\langle a(v),u'\rangle=-\lambda$. Hence $b^{-1}(U_v)$ is the inverse image under $a$ of the universal family $\mathcal Y_{n-1}=\{(h',u'):\langle h',u'\rangle=-\lambda\}$ over $\Gm^{n-1}$ of the proof of Proposition~\ref{prop:tame}, with the same phase $f'$. In the coordinate $w=v^{-1}$ we have $a=\frac{\lambda}{\lambda+w}\,c$, which extends to $w=0$, that is, to $v=\infty$, with value $c$. Suppose $n\ge3$. By Lemma~\ref{lem:integrate-t} and Theorem~\ref{thm:lisse}(iii) for $n-1$ variables, the direct image of $\Lpsi{f'}$ along $\mathcal Y_{n-1}\to\Gm^{n-1}$ is $\mathcal G_{n-1}(1)[-(n-2)]$, and proper base change gives the assertions with $F=a^*\mathcal G_{n-1}(1)$. If $n=2$, the fibre of $b$ over $v\in U_v$ is the point $u_2=-(\lambda+v^{-1})/c_2$, and $F=\Lpsi{-c_2v/(\lambda v+1)}$, which is lisse on $\mathbb P^1_v\smallsetminus\{s\}$.
\end{proof}

We use Laumon's local Fourier transforms $\mathcal F^{(0,\infty')}$ and $\mathcal F^{(\infty,\infty')}$ \cite[D\'ef.~(2.4.2.3)]{Laumon} and his principle of stationary phase \cite[Prop.~(2.3.3.1)(iii)]{Laumon}, in the formulation of \cite[Th.~3]{Katz-Bourbaki}. Let $\Sigma\subset\mathbb A^1$ be a finite set, $j_\Sigma\colon\mathbb A^1\smallsetminus\Sigma\hookrightarrow\mathbb A^1$ the inclusion, and $F_\Sigma$ a lisse sheaf on $\mathbb A^1\smallsetminus\Sigma$. For $x\in\Sigma\cup\{\infty\}$ let $F_\Sigma(x)$ be the representation of the inertia group at $x$ on the generic stalk, transported to $0$ by translation if $x\in\Sigma$. Then, as representations of the inertia group at $\infty$ of the dual line,
\begin{equation}\label{eq:stationary}
\FT_\psi(j_{\Sigma!}F_\Sigma[1])(\infty)\cong\bigoplus_{x\in\Sigma}\Lpsi{xh_1}(\infty)\otimes\mathcal F^{(0,\infty')}\big(F_\Sigma(x)\big)\ \oplus\ \mathcal F^{(\infty,\infty')}\big(F_\Sigma(\infty)\big).
\end{equation}
Moreover all breaks of $\mathcal F^{(0,\infty')}(V)$ are $<1$ for every $V$ \cite[Th.~(2.4.3)(i)]{Laumon}, \cite[Prop.~6]{Katz-Bourbaki}, and $\mathcal F^{(\infty,\infty')}(V)=0$ if all breaks of $V$ are $\le1$ \cite[Th.~(2.4.3)]{Laumon}, \cite[Cor.~7]{Katz-Bourbaki}.

\begin{proposition}\label{prop:slopes-infinity}
Every break of $(\mathcal G_n|_{C_c})(\infty)$ is at most $1$. If $M$ is the part of $(\mathcal G_n|_{C_c})(\infty)$ of break $1$, then all breaks of $\Lpsi{h_1/\lambda}(\infty)\otimes M$ are $<1$. The same holds for every subquotient of $\mathcal G_n|_{C_c}$, in particular for $\mathcal W_n|_{C_c}$.
\end{proposition}

\begin{proof}
Let $\mathfrak C$ be the class of continuous representations $Q$ of the inertia group at $\infty$ such that all breaks of $Q$ are $\le1$ and all breaks of $\Lpsi{h_1/\lambda}(\infty)\otimes Q^{(1)}$ are $<1$, where $Q^{(1)}$ is the part of $Q$ of break $1$. The break decomposition is functorial and exact \cite[1.1]{Katz-Gauss}, so $\mathfrak C$ is stable under subquotients and extensions. It contains every representation all of whose breaks are $<1$, and every representation $\Lpsi{sh_1}(\infty)\otimes Y$ with all breaks of $Y$ less than $1$, by Lemma~\ref{lem:breaks}(ii); recall $s=-1/\lambda$.

Let $j_U\colon\mathbb A^1_v\smallsetminus\{0,s\}\hookrightarrow\mathbb A^1_v$ be the inclusion, let $F$ be as in Lemma~\ref{lem:R-global}, and put $K_1=j_{U!}F[1]$. The complex $j_{v!}\mathcal R$ vanishes at $v=0$ and agrees with $F[-(n-2)]$ on $U_v$. So excision for the point $s$ gives a distinguished triangle
\[
K_1[-(n-1)]\longrightarrow j_{v!}\mathcal R\longrightarrow i_{s*}\mathcal R_s\longrightarrow,
\]
where $\mathcal R_s$ is the stalk of $\mathcal R$ at $s$, a complex of finite-dimensional vector spaces. By proper base change, $\FT_\psi(i_{s*}\mathcal R_s)\cong\mathcal R_s\otimes\Lpsi{sh_1}[1]$, whose cohomology sheaves are of the form $\Lpsi{sh_1}\otimes(\text{constant})$. By~\eqref{eq:stationary}, applied with $\Sigma=\{0,s\}$ and $F_\Sigma=F$, the representation $\FT_\psi(K_1)(\infty)$ lies in $\mathfrak C$: the summand for $x=0$ has all breaks $<1$, the summand for $x=s$ is of the form $\Lpsi{sh_1}(\infty)\otimes Y$ with all breaks of $Y$ less than $1$, and the last summand vanishes, because $F(\infty)$ is trivial by Lemma~\ref{lem:R-global}. Here $\FT_\psi(K_1)(\infty)$ denotes the representation on the generic stalk of the lisse sheaf $\mathcal H^{-1}(\FT_\psi(K_1))$ near $\infty$; the sheaf $\mathcal H^0(\FT_\psi(K_1))$ is punctual, since $\FT_\psi(K_1)$ is perverse \cite[(1.3.2.3)]{Laumon}.

Apply $\FT_\psi$ to the triangle. By~\eqref{eq:FT-model}, $\FT_\psi(j_{v!}\mathcal R)$ is, on $\Gmn{h_1}$, the lisse sheaf $\mathcal G_n|_{C_c}(1)$ placed in degree $n-2$. On a punctured neighbourhood of $\infty$ all cohomology sheaves involved are lisse, $\mathcal H^{n-2}\big(\FT_\psi(K_1)[-(n-1)]\big)=\mathcal H^{-1}(\FT_\psi(K_1))$, and $\mathcal H^{n-1}\big(\FT_\psi(K_1)[-(n-1)]\big)=\mathcal H^0(\FT_\psi(K_1))$ vanishes there. So the long exact sequence of cohomology sheaves exhibits $(\mathcal G_n|_{C_c})(\infty)$ as an extension of a submodule of $\mathcal H^{n-2}\big(\FT_\psi(i_{s*}\mathcal R_s)\big)(\infty)$ by a quotient of $\FT_\psi(K_1)(\infty)$. Tate twists do not change the action of inertia. Since $\mathfrak C$ is stable under subquotients and extensions, $(\mathcal G_n|_{C_c})(\infty)$ and all its subquotients lie in $\mathfrak C$.
\end{proof}

\begin{theorem}\label{thm:infinity}
\begin{enumerate}[(i)]
\item $(\mathcal W_n|_{C_c})(\infty)=T\oplus M$, where $T$ is the tame part, unipotent of dimension $r(n)-2^{n-1}$ with $\kappa_j$ Jordan blocks of size $j-2$ for each $j\ge3$, and $M$ has dimension $2^{n-1}$ and is isoclinic of break $1$. Moreover $M\cong\Lpsi{-h_1/\lambda}(\infty)\otimes X$, where all breaks of $X$ are $<1$.
\item The sheaf $\det(\mathcal W_n|_{C_c})\otimes\Lpsi{2^{n-1}\lambda^{-1}h_1}$ on $C_c$ is geometrically constant.
\end{enumerate}
In particular Theorem~\ref{thmC}(ii) and~(iii) hold.
\end{theorem}

\begin{proof}
(i) By Theorem~\ref{thm:boundary}(iii), the tame part $T$ is unipotent with $\kappa_j$ blocks of size $j-2$ for $j\ge3$. Since $\sum_jj\kappa_j=2^{n-1}$ and $\sum_j\kappa_j=\binom{n-1}{\lfloor(n-1)/2\rfloor}$,
\[
\dim T=\sum_{j\ge3}(j-2)\kappa_j=2^{n-1}-2\binom{n-1}{\lfloor(n-1)/2\rfloor}+\kappa_1=r(n)-2^{n-1},
\]
the last equality by Lemma~\ref{lem:binomial}. So the wild part of $(\mathcal W_n|_{C_c})(\infty)$ has dimension $2^{n-1}$, and its breaks are positive and add up to $\Swan_\infty(\mathcal W_n|_{C_c})=2^{n-1}$ (Theorem~\ref{thm:boundary}(iii)). By Proposition~\ref{prop:slopes-infinity} they are at most $1$, so they all equal $1$. Thus the wild part is the part $M$ of break $1$, and $X=\Lpsi{h_1/\lambda}(\infty)\otimes M$ has all breaks $<1$ by Proposition~\ref{prop:slopes-infinity}.

(ii) Put $\mathcal D=\det(\mathcal W_n|_{C_c})\otimes\Lpsi{2^{n-1}\lambda^{-1}h_1}$, a lisse sheaf of rank one on $C_c\cong\Gm$, and work over $\Fpbar$. At $h_1=0$ the local monodromy of $\mathcal W_n|_{C_c}$ is unipotent (Theorem~\ref{thm:boundary}(ii)), so its determinant is trivial on $I_0$; as $\Lpsi{2^{n-1}\lambda^{-1}h_1}$ is lisse at $0$, the sheaf $\mathcal D$ extends to a lisse sheaf $\bar{\mathcal D}$ on $\mathbb A^1$. At $\infty$, by (i),
\[
\mathcal D(\infty)\cong\det T\otimes\det\big(\Lpsi{-h_1/\lambda}(\infty)\otimes X\big)\otimes\Lpsi{2^{n-1}\lambda^{-1}h_1}(\infty)=\det T\otimes\det X.
\]
The character $\det T$ is tame, and $\det X$, a direct summand of $X^{\otimes2^{n-1}}$, has break $<1$ by Lemma~\ref{lem:breaks}(i). So $\Swan_\infty(\bar{\mathcal D})<1$, and the Grothendieck--Ogg--Shafarevich formula on $\mathbb A^1$ \cite[Exp.~X]{SGA5} gives $\chi_c(\mathbb A^1\otimes\Fpbar,\bar{\mathcal D})=1-\Swan_\infty(\bar{\mathcal D})>0$. Since $H^0_c$ vanishes, $H^2_c(\mathbb A^1\otimes\Fpbar,\bar{\mathcal D})\ne0$; that is, the rank-one sheaf $\bar{\mathcal D}$ has nonzero geometric coinvariants, and it is geometrically constant.
\end{proof}

\section{Moments}\label{sec:fourth}

In this section $n\ge2$, $q$ is a power of $p$, and implied constants depend only on $n$. Recall that $K(0)=-1$ and $|K(a)|\le2q^{1/2}$ for $a\ne0$ \cite{Weil}. Moreover $K(a)$ is real: conjugating $\psi_q(x^{-1}+ax)$ and substituting $x\mapsto-x$ restores the summand.

\subsection{The first two moments}

\begin{lemma}\label{lem:V2}
Let $a,b\in\Fq$. Then $\sum_{x\in\Fq}K(ax)$ equals $-q$ if $a=0$ and $0$ otherwise. Moreover
\[
\sum_{x\in\Fq}K(ax)K(bx)=\begin{cases}q^2-q,&a=b\ne0,\\-q,&a\ne b,\ ab\ne0,\\0,&\text{exactly one of $a,b$ is }0,\\q,&a=b=0.\end{cases}
\]
\end{lemma}

\begin{proof}
We have $\sum_{y\in\Fq}K(y)=\sum_{x\in\Fq^\times}\psi_q(x^{-1})\sum_{y\in\Fq}\psi_q(xy)=0$, which gives the first assertion, as $K(0)=-1$. If $a,b\ne0$, expanding and summing over $x$ first gives
\[
\sum_xK(ax)K(bx)=q\sum_{\substack{u,v\in\Fq^\times\\au+bv=0}}\psi_q(u^{-1}+v^{-1})=q\sum_{u\in\Fq^\times}\psi_q\big((1-b/a)u^{-1}\big),
\]
which is $q(q-1)$ if $a=b$ and $-q$ otherwise. If $a=0\ne b$, the sum is $-\sum_xK(bx)=0$, and if $a=b=0$ it is $q$.
\end{proof}

\begin{proposition}\label{prop:low-moments}
For every $n\ge2$,
\begin{enumerate}[(i)]
\item $\sum_{h\in\Fq^n}\Sigma_n(h)=(-1)^nq^{n-1}$;
\item $\sum_{h\in\Fq^n}|\Sigma_n(h)|^2=q^{n-2}\big((q-1)^{n+1}+1-(-1)^n(q-2)\big)$;
\item $\sum_{h\in\Fq^n}\Sigma_n(h)^2=q^{n-2}\big(1-(q-1)^n+2(-1)^n\big)$.
\end{enumerate}
In particular Theorem~\ref{thmD}(i) holds.
\end{proposition}

\begin{proof}
Expand each factor $\Sigma_n(h)$ by~\eqref{eq:intro-factor} and sum over $h\in\Fq^n$; the sum over $h$ factors over the coordinates. For (i), this gives $q\sum_h\Sigma_n(h)=\sum_t\psi_q(\lambda t)\big(\sum_xK(tx)\big)^n=(-q)^n$, by Lemma~\ref{lem:V2}. For (ii) and (iii), put $V_2(t_1,t_2)=\sum_xK(t_1x)K(t_2x)$. Since $K$ is real,
\[
\begin{aligned}
q^2\sum_h|\Sigma_n(h)|^2&=\sum_{t_1,t_2}\psi_q\big(\lambda(t_1-t_2)\big)V_2(t_1,t_2)^n,\\
q^2\sum_h\Sigma_n(h)^2&=\sum_{t_1,t_2}\psi_q\big(\lambda(t_1+t_2)\big)V_2(t_1,t_2)^n.
\end{aligned}
\]
By Lemma~\ref{lem:V2}, the pair $t_1=t_2=0$ contributes $q^n$, pairs with exactly one zero coordinate contribute $0$, pairs with $t_1=t_2\ne0$ contribute $(q^2-q)^n$ times $\psi_q(0)=1$, respectively times $\psi_q(2\lambda t_1)$, and pairs with $t_1\ne t_2$ and $t_1t_2\ne0$ contribute $(-q)^n$ times the additive character. Since $\sum_{t\ne0}\psi_q(c\,t)=-1$ for $c\ne0$, and $2\lambda\ne0$,
\[
\sum_{\substack{t_1\ne t_2\\t_1t_2\ne0}}\psi_q\big(\lambda(t_1-t_2)\big)=1-(q-1)=2-q,\qquad\sum_{\substack{t_1\ne t_2\\t_1t_2\ne0}}\psi_q\big(\lambda(t_1+t_2)\big)=1-(-1)=2.
\]
Hence $q^2\sum_h|\Sigma_n(h)|^2=q^n+(q-1)(q^2-q)^n+(2-q)(-q)^n$ and $q^2\sum_h\Sigma_n(h)^2=q^n-(q^2-q)^n+2(-q)^n$, which are (ii) and (iii).
\end{proof}

The substitution $u\mapsto-u$ transforms the condition defining $\Sigma_n(h)$ into the one defining $\Sigma_n(-h)$ and conjugates the summand, so $\Sigma_n(-h)=\overline{\Sigma_n(h)}$. Thus (iii) is the correlation of $\Sigma_n$ with its reflection $h\mapsto-h$; it is used in Corollary~\ref{cor:M2}.

\subsection{A four-parameter identity}
For $t=(t_1,t_2,t_3,t_4)\in\Fq^4$ put
\[
V(t)=\sum_{x\in\Fq}\prod_{j=1}^4K(t_jx),\qquad\Lambda(t)=\lambda(t_1-t_2+t_3-t_4).
\]

\begin{proposition}\label{prop:master}
For every $n\ge2$,
\[
q^4\sum_{h\in\Fq^n}|\Sigma_n(h)|^4=\sum_{t\in\Fq^4}\psi_q(\Lambda(t))\,V(t)^n.
\]
\end{proposition}

\begin{proof}
Write $|\Sigma_n(h)|^4=\Sigma_n(h)\overline{\Sigma_n(h)}\Sigma_n(h)\overline{\Sigma_n(h)}$ and expand each factor by \eqref{eq:intro-factor}, with summation variables $t_1,\dots,t_4$. Since $K$ is real, conjugation only replaces $\psi_q(\lambda t_j)$ by $\psi_q(-\lambda t_j)$, so the additive phases combine to $\Lambda(t)$. The remaining product $\prod_{i,j}K(t_jh_i)$ factors over $i$, and summing over $h\in\Fq^n$ gives $V(t)^n$.
\end{proof}

We partition $\Fq^4$ as follows. Let $\mathcal T_0=\{t:t_1t_2t_3t_4=0\}$. Let $\mathcal D_1=\{(a,a,b,b)\}$ and $\mathcal D_2=\{(a,b,b,a)\}$, with $a,b\in\Fq^\times$, and $\mathcal T_D=\mathcal D_1\cup\mathcal D_2$. Let $\mathcal T_P=\{(a,b,a,b):a,b\in\Fq^\times,\ a\ne b\}$, and let $\mathcal T_G$ be the complement of $\mathcal T_0\cup\mathcal T_D\cup\mathcal T_P$. The two pairings defining $\mathcal T_D$ are those compatible with the signs in $\Lambda$, which vanishes identically on $\mathcal T_D$.

\subsection{The degenerate strata}

\begin{lemma}\label{lem:diag}
Let $a,b\in\Fq^\times$, $\rho=a/b$, and let $\chi_2$ be the quadratic character of $\Fq^\times$, extended by $\chi_2(0)=0$. Then
\[
\sum_{x\in\Fq}K(ax)^2K(bx)^2=\begin{cases}2q^3-3q^2-3q,&\rho=1,\\ q^3-(2+\chi_2(\rho))q^2-3q,&\rho\ne1.\end{cases}
\]
\end{lemma}

\begin{proof}
For $s\in\Fq$ let $\tau(s)$ be the sum of $\psi_q(x_1^{-1}+x_2^{-1})$ over $x_1,x_2\in\Fq^\times$ with $x_1+x_2=s$. Expanding the squares gives $K(y)^2=\sum_s\tau(s)\psi_q(ys)$, and orthogonality in $x$ yields
\[
\sum_xK(ax)^2K(bx)^2=q\sum_{s\in\Fq}\tau(s)\tau(-\rho s).
\]
Clearly $\tau(0)=q-1$. For $s\ne0$ write $x_1=sy$ and $x_2=s(1-y)$ with $y\notin\{0,1\}$; then $x_1^{-1}+x_2^{-1}=\big(sy(1-y)\big)^{-1}$. For $m\in\Fq^\times$ the equation $y(1-y)=m$ has $1+\chi_2(1-4m)$ solutions. Hence, writing $\mu=m^{-1}$ and noting $\chi_2(1-4m)=\chi_2(\mu(\mu-4))$,
\[
\tau(s)=\sum_{m\ne0}\big(1+\chi_2(1-4m)\big)\psi_q\Big(\frac1{sm}\Big)=-1+R(s),\qquad R(s)=\sum_{\mu\ne0}\chi_2\big(\mu(\mu-4)\big)\psi_q(\mu/s).
\]
We use Jacobsthal's evaluation: $\sum_{\mu\in\Fq}\chi_2\big((\mu-\alpha)(\mu-\beta)\big)=-1$ for distinct $\alpha,\beta\in\Fq$. Indeed, for $\mu\ne\beta$ the summand is $\chi_2(z)$ with $z=(\mu-\alpha)/(\mu-\beta)$, which runs over $\Fq\smallsetminus\{1\}$. Since $p$ is odd, $0\ne4$, so $\sum_{\mu\ne0}\chi_2(\mu(\mu-4))=-1$, and therefore $\sum_{s\ne0}R(s)=\sum_{s\ne0}R(-\rho s)=1$. Moreover
\[
\begin{aligned}
\sum_{s\ne0}R(s)R(-\rho s)&=\sum_{\mu,\mu'\ne0}\chi_2\big(\mu(\mu-4)\big)\chi_2\big(\mu'(\mu'-4)\big)\big(q\,\mathbf 1_{\mu'=\rho\mu}-1\big)\\
&=q\,\chi_2(\rho)\sum_{\mu\ne0}\chi_2\big((\mu-4)(\rho\mu-4)\big)-1.
\end{aligned}
\]
If $\rho\ne1$, Jacobsthal's evaluation with the roots $4$ and $4/\rho$ gives $\sum_{\mu\ne0}\chi_2\big((\mu-4)(\rho\mu-4)\big)=-\chi_2(\rho)-1$, so the last expression equals $-q-q\chi_2(\rho)-1$. If $\rho=1$ the sum is $q-2$, and the expression equals $q^2-2q-1$. Finally
\[
\sum_s\tau(s)\tau(-\rho s)=(q-1)^2+(q-1)-2+\sum_{s\ne0}R(s)R(-\rho s),
\]
which is $q^2-2q-3-q\chi_2(\rho)$ if $\rho\ne1$ and $2q^2-3q-3$ if $\rho=1$. Multiplying by $q$ gives the lemma.
\end{proof}

\begin{proposition}\label{prop:diag-main}
For $n\ge2$, $\displaystyle\sum_{t\in\mathcal T_D}\psi_q(\Lambda(t))V(t)^n=2q^{3n+2}+O(q^{3n+1})$.
\end{proposition}

\begin{proof}
On $\mathcal T_D$ we have $\Lambda=0$ and $V(a,a,b,b)=V(a,b,b,a)=\sum_xK(ax)^2K(bx)^2$. By Lemma~\ref{lem:diag} this is $q^3+O(q^2)$ for the $(q-1)(q-2)$ pairs with $a\ne b$, and $2q^3+O(q^2)$ for the $q-1$ pairs with $a=b$. Hence each of $\mathcal D_1$ and $\mathcal D_2$ contributes $q^{3n+2}+O(q^{3n+1})$, while their intersection, the diagonal, contributes $O(q^{3n+1})$. Inclusion--exclusion gives the claim.
\end{proof}

\begin{proposition}\label{prop:pair}
For $n\ge2$, $\displaystyle\sum_{t\in\mathcal T_P}\psi_q(\Lambda(t))V(t)^n=O(q^{3n+1})$.
\end{proposition}

\begin{proof}
For $t=(a,b,a,b)$ with $a\ne b$, Lemma~\ref{lem:diag} gives $V(t)=q^3+O(q^2)$, hence $V(t)^n=q^{3n}+O(q^{3n-1})$, while $\Lambda(t)=2\lambda(a-b)$. Since $2\lambda\ne0$,
\[
\sum_{\substack{a,b\ne0\\a\ne b}}\psi_q\big(2\lambda(a-b)\big)=\Big|\sum_{a\ne0}\psi_q(2\lambda a)\Big|^2-(q-1)=2-q.
\]
So the main term contributes $O(q^{3n+1})$. The error term, summed trivially over $O(q^2)$ points, contributes $O(q^{3n+1})$ as well.
\end{proof}

\begin{proposition}\label{prop:zero}
For $n\ge2$, $\displaystyle\sum_{t\in\mathcal T_0}\psi_q(\Lambda(t))V(t)^n=O\big(q^{(5n+6)/2}\big)$.
\end{proposition}

\begin{proof}
Suppose exactly $a\le3$ coordinates of $t$ are nonzero. The term $x=0$ of $V(t)$ has absolute value $1$, and Weil's bound gives $|V(t)|\le1+(q-1)2^aq^{a/2}\le2^{a+1}q^{1+a/2}$. There are at most $\binom4aq^a$ such $t$. The resulting bound $O\big(q^{a+n(1+a/2)}\big)$ is largest for $a=3$, where it equals $O\big(q^{(5n+6)/2}\big)$.
\end{proof}

\subsection{The Cayley cubic and the generic stratum}
For $t\in\Fq^4$ with all $t_i\ne0$ let $H_t\subset\mathbb P^3$ be the plane $\sum_it_iu_i=0$. Let
\[
\mathcal Q=\{e_3(u)=0\}\subset\mathbb P^3,\qquad e_3(u)=\sum_{i<j<k}u_iu_ju_k,
\]
be the Cayley cubic surface, and put $\Gamma_t=H_t\cap\mathcal Q$ and $a_t=q+1-\#\Gamma_t(\Fq)$.

\begin{lemma}\label{lem:plane}
If all $t_i\ne0$, then $V(t)=-q^2\big(a_t+2+3q^{-1}\big)$.
\end{lemma}

\begin{proof}
Orthogonality in $x$ gives $V(t)=q\,\Sigma_0(t)$, where $\Sigma_0(t)$ is the sum of $\psi_q(u_1^{-1}+\cdots+u_4^{-1})$ over the $u\in(\Fq^\times)^4$ with $\sum_it_iu_i=0$. This set is stable under scalar multiplication, and summing $\psi_q(\nu^{-1}f(u))$ over $\nu\in\Fq^\times$ gives $q-1$ if $f(u)=0$ and $-1$ otherwise. Hence $\Sigma_0(t)=qN_0-N$, where $N$ is the number of points of $H_t(\Fq)$ with no zero coordinate and $N_0$ is the number of those with $f(u)=0$. The four lines $H_t\cap\{u_i=0\}$ are distinct. They meet pairwise in the six points of $H_t$ with exactly two zero coordinates, and no three of them are concurrent, since $H_t$ contains no coordinate point. Inclusion--exclusion gives $N=(q^2+q+1)-4(q+1)+6=q^2-3q+3$. For $u$ with no zero coordinate, $f(u)=0$ if and only if $e_3(u)=0$. The restriction of $e_3$ to $\{u_i=0\}$ is the product of the other three coordinates. So a point of $\Gamma_t$ with a zero coordinate has exactly two zero coordinates, and each of the six such points of $H_t$ lies on $\mathcal Q$. Hence $N_0=\#\Gamma_t(\Fq)-6=q-5-a_t$, so $\Sigma_0(t)=-qa_t-2q-3$, and the formula follows.
\end{proof}

\begin{proposition}\label{prop:lines}
The lines of $\mathbb P^3$ over $\Fqbar$ contained in $\mathcal Q$ are exactly the six coordinate edges $\{u_i=u_j=0\}$ and the three lines $\{u_i+u_j=u_k+u_l=0\}$, $\{i,j\}\sqcup\{k,l\}=\{1,2,3,4\}$. For $t\in\mathcal T_G$ the curve $\Gamma_t$ is a geometrically irreducible and reduced plane cubic.
\end{proposition}

\begin{proof}
Throughout the proof, points and lines are taken over $\Fqbar$. The nine lines lie on $\mathcal Q$. Every monomial of $e_3$ involves $u_i$ or $u_j$, so $e_3$ vanishes on $\{u_i=u_j=0\}$. For a partition $\{i,j\}\sqcup\{k,l\}$ one has $e_3=u_iu_j(u_k+u_l)+u_ku_l(u_i+u_j)$, which vanishes on $\{u_i+u_j=u_k+u_l=0\}$. Conversely, a line through the point $(1:0:0:0)$ consists of the points $(1:sa:sb:sc)$ together with $(0:a:b:c)$. It lies on $\mathcal Q$ if and only if $s^2(ab+ac+bc)+s^3abc$ vanishes identically, that is, if and only if two of $a,b,c$ vanish. By symmetry, the lines on $\mathcal Q$ through a coordinate point are the coordinate edges. A line contained in a coordinate plane $u_i=0$ lies on $\mathcal Q$ only if the product of the other three coordinates vanishes on it, so it is again a coordinate edge.

Let $\Pi\subset\mathcal Q$ be a line through no coordinate point and not contained in a coordinate plane. Each $u_i$ restricts to a nonzero linear form $\theta_i$ on $\Pi\cong\mathbb P^1$, with a single zero $z_i$. Choose an affine coordinate $z$ on $\Pi$ for which no $z_i$ lies at infinity, and write $\theta_i=\alpha_i(z-z_i)$ with $\alpha_i\ne0$. Since $e_3=\theta_1\theta_2\theta_3\theta_4\sum_i\theta_i^{-1}$ on $\Pi$, the rational function $\sum_i\alpha_i^{-1}(z-z_i)^{-1}$ vanishes identically. A zero $z_i$ distinct from all the others would be a simple pole with residue $\alpha_i^{-1}\ne0$, so the zeros coincide at least in pairs. Three of them cannot coincide, for $\Pi$ would then pass through a coordinate point. Thus, after renumbering, $z_1=z_2\ne z_3=z_4$, and $\theta_2=\gamma\theta_1$, $\theta_4=\gamma'\theta_3$ with $\gamma,\gamma'\ne0$. Then
\[
e_3|_\Pi=\theta_1\theta_3\big(\gamma(1+\gamma')\theta_1+\gamma'(1+\gamma)\theta_3\big),
\]
and since $\theta_1,\theta_3$ are linearly independent, $\gamma=\gamma'=-1$. So $\Pi$ is the line $u_1+u_2=u_3+u_4=0$.

Now let $t\in\mathcal T_G$. The plane $H_t$ contains no coordinate edge, since $\{u_i=u_j=0\}\subset H_t$ would force $t_k=t_l=0$. It contains the line $u_i+u_j=u_k+u_l=0$, which is spanned by $e_i-e_j$ and $e_k-e_l$, if and only if $t_i=t_j$ and $t_k=t_l$. For the three partitions these conditions describe exactly $\mathcal D_1$, $\mathcal D_2$ and $\mathcal T_P\cup(\mathcal D_1\cap\mathcal D_2)$. Hence $H_t$ contains no line of $\mathcal Q$. In particular $H_t\not\subset\mathcal Q$, so $\Gamma_t$ is a plane cubic curve. A plane cubic that is reducible or non-reduced over $\Fqbar$ contains a line, which would be a line of $\mathcal Q$ lying in $H_t$. Hence $\Gamma_t$ is geometrically irreducible and reduced.
\end{proof}

\begin{corollary}\label{cor:hasse}
For $t\in\mathcal T_G$, $|a_t|\le2q^{1/2}$.
\end{corollary}

\begin{proof}
The six points of $H_t$ with exactly two zero coordinates are $\Fq$-rational points of $\Gamma_t$, and they are smooth points of $\Gamma_t$. For instance, at $(t_2:-t_1:0:0)$ the gradient of $e_3$ is $(0,0,u_1u_2,u_1u_2)$ with $u_1u_2\ne0$, and it is not proportional to the gradient $(t_1,t_2,t_3,t_4)$ of the equation of $H_t$, because $t_1\ne0$. If $\Gamma_t$ is smooth, it is therefore an elliptic curve over $\Fq$, and Hasse's theorem \cite[Ch.~V, Th.~1.1]{Silverman} gives $|a_t|\le2q^{1/2}$. Otherwise $\Gamma_t$ is, by Proposition~\ref{prop:lines}, a geometrically irreducible singular plane cubic. A line through two singular points, or through a point of multiplicity $3$ and a further point of $\Gamma_t$, would meet $\Gamma_t$ with total multiplicity at least $4$. By B\'ezout's theorem it would then be a component of $\Gamma_t$, which is impossible. Hence $\Gamma_t$ has a unique singular point $P_0$, and it is a double point. Being unique, $P_0$ is fixed by $\mathrm{Gal}(\Fqbar/\Fq)$ and is $\Fq$-rational. Let $\nu\colon\tilde\Gamma\to\Gamma_t$ be the normalization. The arithmetic genus of a plane cubic is $1$, and the double point lowers it by at least $1$, so $\tilde\Gamma$ is a smooth projective curve of genus $0$. It has an $\Fq$-point, lying above a smooth rational point of $\Gamma_t$, so it is isomorphic to $\mathbb P^1$. The map $\nu$ is an isomorphism away from $P_0$, and $\nu^{-1}(P_0)$ consists of one or two geometric points, one for each branch of $\Gamma_t$ at $P_0$. Hence $\nu^{-1}(P_0)$ has $0$, $1$ or $2$ rational points. Therefore $\#\Gamma_t(\Fq)=(q+1-\#\nu^{-1}(P_0)(\Fq))+1$, and $a_t=\#\nu^{-1}(P_0)(\Fq)-1$ satisfies $|a_t|\le1$.
\end{proof}

\begin{lemma}\label{lem:ray}
For every integer $k\ge0$,
\[
\Big|\sum_{t\in\mathcal T_G}\psi_q(\Lambda(t))\,a_t^k\Big|\le7\cdot2^kq^{3+k/2}.
\]
\end{lemma}

\begin{proof}
Let $\nu\in\Fq^\times$. The sets $\mathcal T_0$, $\mathcal D_1$, $\mathcal D_2$ and $\mathcal T_P$ are defined by conditions invariant under $t\mapsto\nu t$: the vanishing of a coordinate, and the patterns $(a,a,b,b)$, $(a,b,b,a)$ and $(a,b,a,b)$ with $a\ne b$. Hence $\mathcal T_G$ is stable under $t\mapsto\nu t$. Moreover $H_{\nu t}=H_t$, so $\Gamma_{\nu t}=\Gamma_t$ and $a_{\nu t}=a_t$; and $\Lambda(\nu t)=\nu\Lambda(t)$, as $\Lambda$ is linear. All coordinates of a point $t\in\mathcal T_G$ are nonzero, so $\nu t=t$ forces $\nu=1$. Thus $\Fq^\times$ acts freely on $\mathcal T_G$, and $\mathcal T_G$ is the disjoint union of the orbits $\{\nu t:\nu\in\Fq^\times\}$, each with $q-1$ elements. Let $\overline{\mathcal T}_G\subset\mathbb P^3(\Fq)$ be the set of orbits, that is, the image of $\mathcal T_G$. The condition $\Lambda(t)=0$ depends only on the orbit $[t]$, and
\[
\sum_{\nu\in\Fq^\times}\psi_q(\nu\Lambda(t))=\begin{cases}q-1,&\Lambda(t)=0,\\-1,&\Lambda(t)\ne0.\end{cases}
\]
Summing first over each orbit, we obtain
\[
\sum_{t\in\mathcal T_G}\psi_q(\Lambda(t))\,a_t^k=(q-1)\sum_{\substack{[t]\in\overline{\mathcal T}_G\\ \Lambda(t)=0}}a_t^k-\sum_{\substack{[t]\in\overline{\mathcal T}_G\\ \Lambda(t)\ne0}}a_t^k=q\sum_{\substack{[t]\in\overline{\mathcal T}_G\\ \Lambda(t)=0}}a_t^k-\sum_{[t]\in\overline{\mathcal T}_G}a_t^k.
\]
The orbits with $\Lambda(t)=0$ lie on the plane $\Lambda=0$ of $\mathbb P^3$, which has $q^2+q+1\le3q^2$ rational points, and $\#\overline{\mathcal T}_G\le\#\mathbb P^3(\Fq)=q^3+q^2+q+1\le4q^3$. By Corollary~\ref{cor:hasse}, $|a_t|^k\le(2q^{1/2})^k$ for every $[t]\in\overline{\mathcal T}_G$. So the right-hand side is at most $(q\cdot3q^2+4q^3)(2q^{1/2})^k=7\cdot2^kq^{3+k/2}$ in absolute value.
\end{proof}

\begin{proposition}\label{prop:generic}
For $n\ge2$, $\displaystyle\sum_{t\in\mathcal T_G}\psi_q(\Lambda(t))V(t)^n=O\big(q^{(5n+6)/2}\big)$.
\end{proposition}

\begin{proof}
By Lemma~\ref{lem:plane}, $V(t)^n=(-1)^nq^{2n}\sum_{k=0}^n\binom nk(2+3q^{-1})^{n-k}a_t^k$. Apply Lemma~\ref{lem:ray} to each term; the largest contribution, $O(q^{2n+3+n/2})$, comes from $k=n$.
\end{proof}

\subsection{The fourth moment}

\begin{theorem}\label{thm:fourth}
For every $n\ge3$,
\[
\sum_{h\in\Fq^n}|\Sigma_n(h)|^4=2q^{3n-2}+O\big(q^{3n-2-\delta(n)}\big).
\]
In particular Theorem~\ref{thmD}(ii) holds.
\end{theorem}

\begin{proof}
The four strata partition $\Fq^4$, and Table~\ref{tab:strata} collects their contributions to the right-hand side of Proposition~\ref{prop:master}. Adding them, the left-hand side multiplied by $q^4$ equals $2q^{3n+2}+O(q^{3n+1})+O\big(q^{(5n+6)/2}\big)$. After division by $q^4$ the main term is $2q^{3n-2}$, and the two error terms are $O(q^{A_1})$ and $O(q^{A_2})$ with
\[
A_1=3n-3=(3n-2)-1,\qquad A_2=\frac{5n-2}2=(3n-2)-\frac{n-2}2.
\]
So the total error is $O(q^{3n-2-\delta(n)})$ with $\delta(n)=\min\big(1,\frac{n-2}2\big)$. For $n=3$ one has $A_1=6$ and $A_2=\frac{13}2$, so $\delta(3)=\frac12$; for $n\ge4$, $A_2\le A_1$ and $\delta(n)=1$.
\end{proof}

\begin{table}[ht]
\centering
\small
\caption{The strata of $\Fq^4$ in the proof of Theorem~\ref{thm:fourth}, for $n\ge2$. The last column is the contribution of the stratum to $\sum_t\psi_q(\Lambda(t))V(t)^n$, given by Propositions~\ref{prop:diag-main}, \ref{prop:pair}, \ref{prop:zero} and~\ref{prop:generic}. On the diagonal $t_1=t_2=t_3=t_4$ of $\mathcal T_D$ one has $V(t)=2q^3+O(q^2)$; on $\mathcal T_0$, $a\le3$ is the number of nonzero coordinates of $t$; on $\mathcal T_G$, $|a_t|\le2q^{1/2}$ by Corollary~\ref{cor:hasse}.}\label{tab:strata}
\begin{tabular}{@{}llll@{}}
\toprule
stratum & number of points & $V(t)$ & contribution\\
\midrule
$\mathcal T_D$ & $2(q-1)^2-(q-1)$ & $q^3+O(q^2)$ & $2q^{3n+2}+O(q^{3n+1})$\\
$\mathcal T_P$ & $(q-1)(q-2)$ & $q^3+O(q^2)$ & $O(q^{3n+1})$\\
$\mathcal T_0$ & $q^4-(q-1)^4\le4q^3$ & $O(q^{1+a/2})$ & $O\big(q^{(5n+6)/2}\big)$\\
$\mathcal T_G$ & $\le(q-1)^4$ & $-q^2\big(a_t+2+3q^{-1}\big)$ & $O\big(q^{(5n+6)/2}\big)$\\
\bottomrule
\end{tabular}
\end{table}

For $n=3$ the saving is $q^{1/2}$. It comes from the averaging over projective rays in Lemma~\ref{lem:ray}, which reduces the generic stratum from the trivial size $q^{4+5n/2}$ to $q^{3+5n/2}$.

\section{The geometric monodromy group}\label{sec:monodromy}

Let $V$ be the geometric generic stalk of $\mathcal W_n$, and let $G=\Ggeom(\mathcal W_n)\subseteq\GL(V)$ be the Zariski closure of the image of $\pi_1(S\otimes\Fpbar)$. For a Zariski-closed subgroup $H\subseteq\GL(V)$ put
\[
M_4(H,V)=\dim\big(V^{\otimes2}\otimes(V^\vee)^{\otimes2}\big)^H=\dim\End_H(V\otimes V).
\]
Since $\mathcal W_n$ is pure, it is geometrically semisimple \cite[3.4.1(iii)]{Deligne}, and $G^0$ is semisimple \cite[Cor.~1.3.9]{Deligne}. In particular $G^0$ is reductive, and every finite-dimensional representation of $G$ is semisimple by Lemma~\ref{lem:semisimple}(i). For $h\in S(\Fq)$ put $A^0_q(h)=q^{-(n+1)/2}A_q(h)$, so that $|A^0_q(h)|\le r(n)$. By Theorem~\ref{thm:lisse}(v),
\begin{equation}\label{eq:normalized}
\Sigma_n(h)=(-1)^{n-1}\big(q^{(n-1)/2}A^0_q(h)+c_n(q)\big),\qquad c_n(q)=\sum_{k=1}^{\lfloor n/2\rfloor}m_{n,k}q^{k-1}\ll q^{(n-2)/2}.
\end{equation}
By Corollary~\ref{cor:all-h}, the points of $\Fq^n\smallsetminus S(\Fq)$ with exactly $a$ nonzero coordinates, $0\le a\le n-1$, number at most $\binom naq^a$, and for each of them $|\Sigma_n(h)|\le(2^a-1)q^{(a-1)/2}$ if $a\ge1$ and $\Sigma_n(h)=0$ if $a=0$.

\subsection{Moments and irreducibility}

\begin{lemma}\label{lem:lower}
Let $n\ge3$. As $q\to\infty$ through powers of $p$,
\[
q^{-n}\sum_{h\in S(\Fq)}|A^0_q(h)|^4=2+O\big(q^{-1/2}\big).
\]
\end{lemma}

\begin{proof}
By~\eqref{eq:normalized}, $\big||\Sigma_n(h)|^4-q^{2n-2}|A^0_q(h)|^4\big|\ll q^{3(n-1)/2}q^{(n-2)/2}=q^{2n-5/2}$, and summing over the at most $q^n$ points of $S(\Fq)$,
\[
\sum_{h\in S(\Fq)}|\Sigma_n(h)|^4=q^{2n-2}\sum_{h\in S(\Fq)}|A^0_q(h)|^4+O(q^{3n-5/2}).
\]
The points of $\Fq^n\smallsetminus S(\Fq)$ contribute $O(q^{3n-5})$ to the fourth moment. Since $\delta(n)\ge\frac12$ for $n\ge3$, Theorem~\ref{thm:fourth} yields $q^{2n-2}\sum_{h\in S(\Fq)}|A^0_q(h)|^4=2q^{3n-2}+O(q^{3n-5/2})$.
\end{proof}

\begin{proposition}\label{prop:moment}
Let $T$ be a smooth geometrically connected variety of dimension $d$ over a finite field $\mathbb F_{q_0}$, and let $\mathcal F$ be a lisse sheaf on $T$, pure of weight $w$, with geometric generic stalk $V_{\mathcal F}$ and geometric monodromy group $G_{\mathcal F}$. Put $M=M_4(G_{\mathcal F},V_{\mathcal F})$. Then there are complex numbers $\beta_1,\dots,\beta_M$ of absolute value $1$ such that, for $q=q_0^k$,
\[
q^{-d}\sum_{x\in T(\Fq)}\big|q^{-w/2}\Tr(\Frob_x\mid\mathcal F)\big|^4=\sum_{m=1}^M\beta_m^k+O\big(q^{-1/2}\big).
\]
Consequently, if the left-hand side tends to a limit $L$ as $k\to\infty$, then $L=M$.
\end{proposition}

\begin{proof}
Put $\mathcal E=\mathcal F^{\otimes2}\otimes(\mathcal F^\vee)^{\otimes2}$, a lisse sheaf pure of weight $0$. At $x\in T(\Fq)$ the eigenvalues of $\Frob_x$ on $\mathcal F^\vee$ are the inverses $\alpha^{-1}=q^{-w}\bar\alpha$ of the eigenvalues $\alpha$ on $\mathcal F$, so $\Tr(\Frob_x\mid\mathcal E)=\big|q^{-w/2}\Tr(\Frob_x\mid\mathcal F)\big|^4$. The invariants of $\pi_1(T\otimes\Fqbar)$ on $V_{\mathcal F}^{\otimes2}\otimes(V_{\mathcal F}^\vee)^{\otimes2}$ are those of its Zariski closure $G_{\mathcal F}$, and they have dimension $M$. Lemma~\ref{lem:averages} gives the assertions.
\end{proof}

\begin{corollary}\label{cor:M4}
For $n\ge3$ one has $M_4(G,V)=2$. Consequently $V$ is an irreducible representation of $G$, and it is not self-dual.
\end{corollary}

\begin{proof}
Lemma~\ref{lem:lower} and Proposition~\ref{prop:moment}, applied to $\mathcal W_n$ on $T=S$ with $d=n$ and $w=n+1$, give $M_4(G,V)=2$. Since $V$ is semisimple, if it were reducible we could write $V=V_1\oplus V_2$ with $V_1,V_2\ne0$ stable under $G$; then $\dim\End_G(V)\ge2$, and the injection $\End_G(V)\otimes\End_G(V)\hookrightarrow\End_G(V\otimes V)$ gives $M_4\ge4$. Suppose $V\cong V^\vee$. Since $V$ is irreducible, $(V\otimes V)^G\cong\Hom_G(V^\vee,V)$ is one-dimensional by Schur's lemma. Write $V\otimes V=\Sym^2V\oplus\bigwedge^2V$; both summands are $G$-stable. The two components of an invariant vector are invariant, so the invariant line lies in $\Sym^2V$ or in $\bigwedge^2V$, say in $X$, and let $Y$ be the other summand. Put $N=\dim V=r(n)\ge5$. Then $\dim\Sym^2V=N(N+1)/2$ and $\dim\bigwedge^2V=N(N-1)/2$ both exceed $1$. Hence $X=\mathbf 1\oplus X'$ with $X'\ne0$, and $Y\ne0$. So $V\otimes V=\mathbf 1\oplus X'\oplus Y$ with three nonzero $G$-stable summands. The three projections are linearly independent elements of $\End_G(V\otimes V)$, so $M_4\ge3$.
\end{proof}

The exact moments of Proposition~\ref{prop:low-moments} give these two properties for every $n\ge2$, independently of the fourth moment.

\begin{corollary}\label{cor:M2}
Let $n\ge2$. As $q\to\infty$ through powers of $p$,
\[
q^{-n}\sum_{h\in S(\Fq)}|A^0_q(h)|^2=1+O\big(q^{-1/2}\big),\qquad q^{-n}\sum_{h\in S(\Fq)}A^0_q(h)^2=O\big(q^{-1/2}\big).
\]
Consequently $\dim\End_G(V)=1$ and $(V\otimes V)^G=0$: the representation $V$ of $G$ is irreducible and not self-dual.
\end{corollary}

\begin{proof}
By~\eqref{eq:normalized}, for $h\in S(\Fq)$,
\[
|\Sigma_n(h)|^2=q^{n-1}|A^0_q(h)|^2+2c_n(q)q^{(n-1)/2}\operatorname{Re}A^0_q(h)+c_n(q)^2,
\]
and similarly $\Sigma_n(h)^2=q^{n-1}A^0_q(h)^2+2c_n(q)q^{(n-1)/2}A^0_q(h)+c_n(q)^2$. Summed over the at most $q^n$ points of $S(\Fq)$, the last two terms contribute $O(q^{(n-2)/2+(n-1)/2+n})=O(q^{2n-3/2})$. The points outside $S(\Fq)$ contribute $O(q^{2n-3})$ to both moments. By Proposition~\ref{prop:low-moments}(ii) and~(iii), $\sum_{h\in\Fq^n}|\Sigma_n(h)|^2=q^{2n-1}+O(q^{2n-2})$ and $\sum_{h\in\Fq^n}\Sigma_n(h)^2=O(q^{2n-2})$. Dividing by $q^{2n-1}$ gives the two estimates. Now apply Lemma~\ref{lem:averages} with $T=S$ and $d=n$. For the first estimate take $\mathcal E=\mathcal W_n\otimes\mathcal W_n^\vee$, which is pure of weight $0$ with trace $|A^0_q(h)|^2$ at $h$ and with $\dim\big(V\otimes V^\vee\big)^G=\dim\End_G(V)$. For the second take $\mathcal E=\mathcal W_n^0\otimes\mathcal W_n^0$, where $\mathcal W_n^0=\mathcal W_n\otimes\gamma^{\deg}$ and $\gamma\in\Qlb$ is the element with $\iota_{\mathbb C}(\gamma)=p^{-(n+1)/2}>0$; its trace at $h\in S(\Fq)$ is $A^0_q(h)^2$, and $G$ acts on its geometric generic stalk $V\otimes V$. Hence $\dim\End_G(V)=1$ and $(V\otimes V)^G=0$. As $V$ is semisimple, the first equality means that $V$ is irreducible. The second means $\Hom_G(V^\vee,V)=0$.
\end{proof}

\subsection{The identity component}

\begin{proposition}\label{prop:unipotent}
For every $n\ge2$ the group $G$ contains a unipotent element $g\ne1$ with $\rk(g-1)=r(n-1)$. In particular $G$ is infinite.
\end{proposition}

\begin{proof}
Choose $c\in(\Fpbar^\times)^{n-1}$. The inclusion $C_c\subset S$ induces a homomorphism $\pi_1(C_c\otimes\Fpbar)\to\pi_1(S\otimes\Fpbar)$. Through it, the inertia group at $h_1=0$ acts on $V$ by the local monodromy of $\mathcal W_n|_{C_c}$. By Theorem~\ref{thm:boundary}(ii), a topological generator of tame inertia acts by a unipotent $g$ with $\rk(g-1)=r(n-1)\ge1$. In characteristic $0$ such an element has infinite order, since $g^m=\exp(m\log g)$ and $\log g\ne0$.
\end{proof}

\begin{proposition}\label{prop:det}
For every $n\ge2$, the character $\det\mathcal W_n$ of $\pi_1(S\otimes\Fpbar)$ has finite image $\Gamma$, and $\det(G)=\Gamma$. Consequently $G\cap\{\text{scalars}\}$ is finite.
\end{proposition}

\begin{proof}
Let $\rho\colon\pi_1(S\otimes\Fpbar)\to\GL(V)$ be the representation defined by $\mathcal W_n$, so that $\Gamma$ is the image of $\det\circ\rho$.

\emph{Step 1: inertia.} By symmetry in the coordinates, Theorem~\ref{thm:boundary} applies to every coordinate line in every direction. Fix $i$ and a line $C\subset S\otimes\Fpbar$ along which the coordinates $h_j$, $j\ne i$, are constant. Through the homomorphism $\pi_1(C)\to\pi_1(S\otimes\Fpbar)$, the inertia group $I_0$ at $h_i=0$ acts on $V$ by the local monodromy of $\mathcal W_n|_C$. By Theorem~\ref{thm:boundary}(ii) the wild inertia acts trivially, so the action factors through the procyclic tame quotient of $I_0$, a topological generator of which acts by a unipotent element $g$. The image of $I_0$ therefore lies in the closure of the group generated by $g$, on which the continuous character $\det$ is identically $1$, because $\det g=1$. Thus $\det\circ\rho$ is trivial on $I_0$.

\emph{Step 2: characters of order prime to $p$.} Let $m\ge1$ be prime to $p$. As $S\otimes\Fpbar$ is connected, the continuous homomorphisms $\pi_1(S\otimes\Fpbar)\to\mathbb Z/m$ are the elements of $H^1(S\otimes\Fpbar,\mathbb Z/m)$, which we identify with $H^1(S\otimes\Fpbar,\mu_m)$ by choosing a primitive $m$-th root of unity in $\Fpbar$. The Kummer sequence $1\to\mu_m\to\Gm\xrightarrow{\,m\,}\Gm\to1$, exact for the \'etale topology because $p\nmid m$, gives an exact sequence
\[
0\to\mathcal O(S\otimes\Fpbar)^\times\big/\mathcal O(S\otimes\Fpbar)^{\times m}\to H^1(S\otimes\Fpbar,\mu_m)\to\mathrm{Pic}(S\otimes\Fpbar)[m]\to0.
\]
The ring $\mathcal O(S\otimes\Fpbar)=\Fpbar[h_1^{\pm1},\dots,h_n^{\pm1}]$ is factorial, so $\mathrm{Pic}(S\otimes\Fpbar)=0$ and the first map is an isomorphism. The units of this ring are $\Fpbar^\times\cdot h_1^{\mathbb Z}\cdots h_n^{\mathbb Z}$, and $\Fpbar^\times$ is $m$-divisible. Hence
\[
H^1(S\otimes\Fpbar,\mu_m)\cong(\mathbb Z/m)^n,
\]
with basis the Kummer classes of $h_1,\dots,h_n$. Let $\chi$ be the character attached to the class of $h_1^{e_1}\cdots h_n^{e_n}$, and let $C$ be as in Step~1. On $C$ the functions $h_j$, $j\ne i$, are constants, hence $m$-th powers in $\Fpbar^\times$, so $\chi|_{\pi_1(C)}$ is the Kummer character of $h_i^{e_i}$. Its restriction to $I_0$ is the $e_i$-th power of the character $\sigma\mapsto\sigma(h_i^{1/m})/h_i^{1/m}$. This character maps $I_0$ onto $\mu_m$, because $h_i$ is a uniformizer at $0$, so that adjoining $h_i^{1/m}$ gives a totally ramified extension of degree $m$. Hence $\chi$ is trivial on $I_0$ if and only if $m\mid e_i$. Consequently a character of $\pi_1(S\otimes\Fpbar)$ of finite order prime to $p$ that is trivial on the inertia groups of Step~1 for $i=1,\dots,n$ is trivial.

\emph{Step 3: $\Gamma$ is finite.} The group $\Gamma$ is compact, since $\pi_1(S\otimes\Fpbar)$ is. The group $\Qlb^\times$ is the union of the closed subgroups $E^\times$, where $E\subset\Qlb$ runs over the countably many finite extensions of $\mathbb Q_\ell$. By the Baire category theorem some $\Gamma\cap E^\times$ has nonempty interior in $\Gamma$; it is then an open subgroup, of finite index. Enlarging $E$ by a finite set of coset representatives, we may assume $\Gamma\subseteq E^\times$. The valuation $v$ of $E$ maps $\Gamma$ onto a compact subgroup of the discrete torsion-free group $v(E^\times)$, which is finite and hence trivial; so $\Gamma\subseteq\mathcal O_E^\times$. Now $\mathcal O_E^\times$ is the product of a finite group and the pro-$\ell$ group $1+\mathfrak m_E$. The profinite abelian group $\Gamma$ is the product of its pro-$p$ part $\Gamma_p$ and its prime-to-$p$ part $\Gamma'$. Indeed, $\Gamma$ is the inverse limit of its finite quotients, each finite abelian group is canonically the product of its $p$-part and its prime-to-$p$ part, and these decompositions are compatible with the transition maps. As $\ell\ne p$, the group $\Gamma_p$ meets $1+\mathfrak m_E$ trivially, so it embeds into the finite group $\mathcal O_E^\times/(1+\mathfrak m_E)$ and is finite. If $\Gamma'\ne1$, then $\Gamma'$ has a nontrivial finite quotient, whose order is prime to $p$, and hence a nontrivial continuous character $\chi_0$ of finite order prime to $p$. Composing $\det\circ\rho$, which maps onto $\Gamma$, with the projection $\Gamma\to\Gamma'$ and with $\chi_0$ gives a nontrivial character of $\pi_1(S\otimes\Fpbar)$ of order prime to $p$. By Step~1 it is trivial on all the inertia groups considered there, which contradicts Step~2. Hence $\Gamma'=1$, and $\Gamma=\Gamma_p$ is finite.

\emph{Step 4.} The subset $\det^{-1}(\Gamma)$ of $\GL(V)$ is Zariski closed and contains the image of $\rho$, so it contains $G$. Therefore $\det(G)=\Gamma$. If $z\cdot1\in G$, then $z^{r(n)}=\det(z\cdot1)\in\Gamma$, which leaves finitely many possibilities for $z$.
\end{proof}

\begin{theorem}[Larsen's alternative]\label{thm:larsen}
Let $V$ be a vector space of dimension $N\ge2$ over an algebraically closed field $k$ of characteristic $0$, and let $H\subseteq\GL(V)$ be a Zariski-closed subgroup with $H^0$ reductive and $M_4(H,V)=2$. Then either $H^0\supseteq\SL(V)$, or $H^0$ consists of scalars. In the second case $H/(H\cap\{\text{scalars}\})$ is finite.
\end{theorem}

This is \cite[Th.~1.1.6(1)]{Katz-Larsen}. We include the short proof, which uses no classification of finite groups.

\begin{proof}
By Lemma~\ref{lem:semisimple}(i), every finite-dimensional representation of $H$ is semisimple. Let $H$ act on $\End(V)=V\otimes V^\vee$ by conjugation. A permutation of the tensor factors identifies $V^{\otimes2}\otimes(V^\vee)^{\otimes2}$ with $\End(V)\otimes\End(V)^\vee=\End_k(\End(V))$, compatibly with the action of $H$, so $\dim\End_H(\End(V))=M_4(H,V)=2$. Since $k$ has characteristic $0$, we have $N\ne0$ in $k$, and the trace gives a decomposition $\End(V)=k\cdot1\oplus\mathfrak{sl}(V)$ into $H$-stable subspaces. Accordingly,
\[
\End_H(\End(V))=\End_H(k\cdot1)\oplus\Hom_H(k\cdot1,\mathfrak{sl}(V))\oplus\Hom_H(\mathfrak{sl}(V),k\cdot1)\oplus\End_H(\mathfrak{sl}(V)).
\]
The first and the last summands contain the identity maps and are nonzero, as $\mathfrak{sl}(V)\ne0$. The total dimension is $2$, so both are one-dimensional and the two cross terms vanish. Thus the semisimple $H$-module $\mathfrak{sl}(V)$ has a one-dimensional endomorphism algebra, and it is irreducible.

Let $\mathfrak h\subseteq\End(V)$ be the Lie algebra of $H$. For $g\in H$, conjugation by $g$ preserves $H$, so its differential, which is conjugation by $g$ on $\End(V)$, preserves $\mathfrak h$. Thus $\mathfrak h\cap\mathfrak{sl}(V)$ is an $H$-stable subspace of $\mathfrak{sl}(V)$, equal to $0$ or to $\mathfrak{sl}(V)$. In characteristic $0$ every group scheme of finite type over $k$ is smooth, by Cartier's theorem. Hence the Lie algebra of the intersection of two closed subgroups of $\GL(V)$ is the intersection of their Lie algebras, and every algebraic group has the dimension of its Lie algebra.

If $\mathfrak h\cap\mathfrak{sl}(V)=\mathfrak{sl}(V)$, then $(H\cap\SL(V))^0$ is a closed connected subgroup of $\SL(V)$ of dimension $\dim\mathfrak{sl}(V)=\dim\SL(V)$. Since $\SL(V)$ is connected, hence irreducible as a variety, a closed irreducible subset of the same dimension is all of $\SL(V)$. Thus $(H\cap\SL(V))^0=\SL(V)$, and $\SL(V)\subseteq H^0$.

If $\mathfrak h\cap\mathfrak{sl}(V)=0$, then $\dim\mathfrak h\le1$. The projection $X\mapsto X-N^{-1}\Tr(X)\cdot1$ of $\End(V)$ onto $\mathfrak{sl}(V)$ is $H$-equivariant, and it maps $\mathfrak h$ onto an $H$-stable subspace of $\mathfrak{sl}(V)$ of dimension at most $1<N^2-1$. This subspace is therefore $0$, that is, $\mathfrak h\subseteq k\cdot1$. Let $Z\subseteq\GL(V)$ be the group of scalars. Then $H^0\cap Z$ has Lie algebra $\mathfrak h\cap k\cdot1=\mathfrak h$, hence the same dimension as the irreducible variety $H^0$, and $H^0=(H^0\cap Z)^0\subseteq Z$. Finally $H^0\subseteq H\cap Z$, so $H/(H\cap Z)$ is a quotient of the finite group $H/H^0$.
\end{proof}

\begin{theorem}\label{thm:monodromy}
For every $n\ge2$ and every odd prime $p$, $G^0=\SL_{r(n)}$.
\end{theorem}

\begin{proof}
\emph{Step 1: $G^0\subseteq\SL(V)$ and $\dim G^0\ge1$, for every $n\ge2$.} By Proposition~\ref{prop:det}, $\det(G^0)$ is a subgroup of the finite group $\det(G)$. It is connected, being the image of a connected group under a morphism of algebraic groups, and is therefore trivial. Thus $G^0\subseteq\SL(V)$. By Proposition~\ref{prop:unipotent} the group $G$ is infinite; since it has finitely many connected components, $\dim G^0\ge1$.

\emph{Step 2: the case $n\ge3$.} By Corollary~\ref{cor:M4}, $M_4(G,V)=2$, and $G^0$ is reductive, being semisimple. So Theorem~\ref{thm:larsen} applies to $H=G$. The group $G^0$ does not consist of scalars, for it would then be contained in $G\cap\{\text{scalars}\}$, which is finite by Proposition~\ref{prop:det}, whereas $\dim G^0\ge1$. Hence $G^0\supseteq\SL(V)$, and therefore $G^0=\SL(V)$.

\emph{Step 3: the case $n=2$.} Here $\dim V=r(2)=2$, and the fourth moment gives no information: for $n=2$ the bounds of Propositions~\ref{prop:zero} and~\ref{prop:generic} are of the same order $q^{3n+2}$ as the main term of Proposition~\ref{prop:diag-main}. We argue directly instead. Let $\mathfrak h\subseteq\mathfrak{sl}_2$ be the Lie algebra of the connected reductive group $G^0$. In characteristic $0$ the Lie algebra of a connected reductive group is reductive, so $\mathfrak h=\mathfrak z(\mathfrak h)\oplus[\mathfrak h,\mathfrak h]$ with $[\mathfrak h,\mathfrak h]$ semisimple. If $[\mathfrak h,\mathfrak h]\ne0$, it is therefore a nonzero semisimple subalgebra of $\mathfrak{sl}_2$, hence of dimension at least $3$. So $\mathfrak h=\mathfrak{sl}_2$, and $G^0=\SL_2$, because $\SL_2$ is connected of dimension $3$. If $[\mathfrak h,\mathfrak h]=0$, then $G^0$ is commutative, since in characteristic $0$ a connected algebraic group with abelian Lie algebra is commutative; being reductive, it is a torus, conjugate into the diagonal torus of $\SL_2$; it contains no unipotent element other than $1$. But for the element $g$ of Proposition~\ref{prop:unipotent} some power $g^m$ with $m\ge1$ lies in $G^0$, and $g^m=1+m(g-1)\ne1$, because $(g-1)^2=0$ in dimension two. Hence $G^0=\SL_2$.
\end{proof}

\begin{remark}\label{rem:no-classification}
Larsen's alternative leaves two possibilities, $G^0\supseteq\SL(V)$ or $G$ finite modulo scalars, and the second is excluded by the unipotent element of Proposition~\ref{prop:unipotent}, which has infinite order. So no list of finite subgroups of $\GL(V)$, and in particular neither the classification of finite simple groups nor the results of Guralnick and Tiep on finite primitive groups, enters the proof. The same holds for the determination of the component group below, which uses only the local monodromy at infinity and Artin--Schreier theory.
\end{remark}

\subsection{The determinant and the component group}

\begin{lemma}\label{lem:AS-lines}
Let $k$ be an algebraically closed field of characteristic $p$, let $S_k=\Gmn{k}^n$, and let $\chi\colon\pi_1(S_k)\to\Qlb^\times$ be a continuous character whose image is a finite $p$-group. Suppose that for every $i$ and every $c\in(k^\times)^{n-1}$ the restriction of $\chi$ to the fundamental group of the coordinate line $L_{i,c}=\{h:h_j=c_j\ \text{for}\ j\ne i\}$ is trivial. Then $\chi$ is trivial.
\end{lemma}

\begin{proof}
The image of $\chi$ is cyclic of order $p^e$. We argue by induction on $e$, the case $e=0$ being trivial. The character $\chi^p$ satisfies the same hypothesis and has image of order $p^{e-1}$, so it is trivial by induction; hence $e\le1$. Suppose $e=1$. Choosing a primitive $p$-th root of unity in $\Qlb$, we may regard $\chi$ as an element of $H^1(S_k,\mathbb F_p)=\Hom(\pi_1(S_k),\mathbb F_p)$. Let $A=k[h_1^{\pm1},\dots,h_n^{\pm1}]$ and $\wp(x)=x^p-x$. The Artin--Schreier sequence $0\to\mathbb F_p\to\mathcal O\xrightarrow{\wp}\mathcal O\to0$ on the \'etale site, together with $H^1(S_k,\mathcal O)=0$, identifies $H^1(S_k,\mathbb F_p)$ with $A/\wp(A)$, compatibly with restriction to the lines $L_{i,c}$, for which the same identification holds with $A$ replaced by $k[h_i^{\pm1}]$.

Let $\mathsf M=\mathbb Z^n\smallsetminus p\mathbb Z^n$. Every class in $A/\wp(A)$ is represented by a Laurent polynomial $g=\sum_{m\in\mathsf M}a_mh^m$: constants lie in $\wp(k)$ since $k$ is algebraically closed, and $a\,h^{pm}\equiv a^{1/p}h^m$ modulo $\wp(A)$. This representative is unique. Indeed, if $\sum_{m\in\mathsf M}a_mh^m=\wp(b)$ with $b=\sum_mb_mh^m$ nonconstant, choose $m_0\ne0$ in the support of $b$ with $\max_i|m_{0,i}|$ maximal; then $pm_0$ is not in the support of $b$, so the coefficient of $h^{pm_0}$ in $\wp(b)$ is $b_{m_0}^p\ne0$, although $pm_0\notin\mathsf M$. So $b$ is constant, and then all $a_m$ vanish. The same holds in one variable.

Let $g=\sum_{m\in\mathsf M}a_mh^m$ represent $\chi$, fix $i$, and for $c\in(k^\times)^{n-1}$ write $m=(m',m_i)$ and $c^{m'}=\prod_{j\ne i}c_j^{m_j}$. The restriction of $g$ to $L_{i,c}$ is $\sum_{e\in\mathbb Z}\alpha_e(c)h_i^e$ with $\alpha_e(c)=\sum_{m_i=e}a_mc^{m'}$. Its reduced representative in one variable has, for $e\notin p\mathbb Z$, the coefficient $\sum_{s\ge0}\alpha_{p^se}(c)^{1/p^s}$ of $h_i^e$. By hypothesis these coefficients vanish for all $c$. Raising to a large power $p^N$, we find that the Laurent polynomial $\sum_{s\ge0}\sum_{m_i=p^se}a_m^{p^{N-s}}c^{p^{N-s}m'}$ vanishes on $(k^\times)^{n-1}$, hence identically. Distinct pairs $(s,m)$ give distinct monomials $c^{p^{N-s}m'}$. Indeed, for fixed $s$ the exponent $m$ is determined by $m'$, as $m_i=p^se$. And if $p^{N-s}m'=p^{N-\tilde s}\tilde m'$ with $s<\tilde s$, then $\tilde m=p^{\tilde s-s}m\in p\mathbb Z^n$, which is impossible. So $a_m=0$ whenever $m_i\ne0$. As every $m\in\mathsf M$ has a nonzero coordinate, $g=0$, and $\chi$ is trivial.
\end{proof}

\begin{theorem}\label{thm:det}
For every $n\ge2$ there is $\alpha\in\Qlb^\times$ with $\det\mathcal W_n\cong\Lpsi{-2^{n-1}\lambda^{-1}(h_1+\cdots+h_n)}\otimes\alpha^{\deg}$, where $\alpha^{\deg}$ is the geometrically constant sheaf of rank one on which $\Frob_p$ acts by $\alpha$. The image $\Gamma$ of $\det\mathcal W_n$ on $\pi_1(S\otimes\Fpbar)$ is the group $\mu_p$ of $p$-th roots of unity.
\end{theorem}

\begin{proof}
Put $\mathcal D=\det\mathcal W_n\otimes\Lpsi{2^{n-1}\lambda^{-1}(h_1+\cdots+h_n)}$, and let $\chi$ be the character of $\pi_1(S\otimes\Fpbar)$ that it defines. By Proposition~\ref{prop:det} and its proof, the image of $\det\mathcal W_n$ on $\pi_1(S\otimes\Fpbar)$ is a finite $p$-group, and the image of the Artin--Schreier sheaf lies in $\mu_p$; so the image of $\chi$ is a finite $p$-group. The family is symmetric in the coordinates, so Theorem~\ref{thm:infinity}(ii) applies to the coordinate lines $L_{i,c}$ in every direction $i$. On $L_{i,c}$ the sheaf $\Lpsi{2^{n-1}\lambda^{-1}(h_1+\cdots+h_n)}$ is $\Lpsi{2^{n-1}\lambda^{-1}h_i}$ tensored with a geometrically constant sheaf, so $\chi$ is trivial on $\pi_1(L_{i,c})$. By Lemma~\ref{lem:AS-lines}, $\chi$ is trivial, that is, $\mathcal D$ is geometrically constant. A lisse sheaf of rank one on $S$ that is geometrically constant is of the form $\alpha^{\deg}$, which proves the first assertion. Finally $2^{n-1}\lambda^{-1}\ne0$ in $\Fp$, as $p$ is odd. So the restriction of $\Lpsi{-2^{n-1}\lambda^{-1}(h_1+\cdots+h_n)}$ to a coordinate line is a nontrivial Artin--Schreier sheaf, of Swan conductor $1$ at $\infty$; its geometric image is a nontrivial subgroup of $\mu_p$, hence equal to $\mu_p$.
\end{proof}

The constant $\alpha$ is determined in Theorem~\ref{thm:alpha} below: $\alpha=p^{(n+1)r(n)/2}$.

\begin{theorem}\label{thm:Ggeom}
For every $n\ge2$ and every odd prime $p$,
\[
\Ggeom(\mathcal W_n)=\big\{g\in\GL(V):\det(g)^p=1\big\}.
\]
In particular Theorem~\ref{thmE}(ii) holds.
\end{theorem}

\begin{proof}
By Theorem~\ref{thm:monodromy}, $G^0=\SL(V)$, and by Proposition~\ref{prop:det} and Theorem~\ref{thm:det}, $\det(G)=\Gamma=\mu_p$. So $G\subseteq\{g:\det(g)\in\mu_p\}$. Conversely, let $\det(g)\in\mu_p$. Choose $g'\in G$ with $\det g'=\det g$. Then $g'^{-1}g\in\SL(V)=G^0\subseteq G$, so $g\in G$. The component group $G/G^0=G/\SL(V)$ is isomorphic to $\mu_p$ through $\det$.
\end{proof}

\begin{remark}\label{rem:n2-consistency}
For $n=2$, Corollary~\ref{cor:M2} shows independently that $\det$ is nontrivial on $G$. Indeed, if $\det$ were trivial on $G$, then $G=\SL_2$, and the standard representation of $\SL_2$ is self-dual, contradicting $(V\otimes V)^G=0$. For $n\ge3$ the non-self-duality in Corollary~\ref{cor:M2} carries no information on the determinant, since the standard representation of $\SL_r$ is not self-dual for $r\ge3$.
\end{remark}

\subsection{The constant \texorpdfstring{$\alpha$}{alpha} and the arithmetic monodromy group}\label{sec:alpha}

We now determine the constant $\alpha$ of Theorem~\ref{thm:det}. The argument is an induction on $n$. On a coordinate line, the determinant of a Frobenius element at $h_1=0$, acting on the generic stalk of $\mathcal G_n$, is computed in two ways: from Theorem~\ref{thm:det}, and from the Fourier model~\eqref{eq:FT-model}, through the vanishing cycles of a Fourier transform. Put $b_n=\binom n{\lfloor n/2\rfloor}$, so that $r(n)=2^n-b_n$ for every $n\ge1$, and
\[
S_1(n)=\sum_{k=0}^{\lfloor n/2\rfloor}k\,m_{n,k},\qquad e(n)=(n-2)2^{n-1}+1 .
\]

\begin{lemma}\label{lem:numerology}
Let $n\ge1$.
\begin{enumerate}[(i)]
\item $2S_1(n)=(n+1)b_n-2^n$.
\item $e(1)=0$, and $e(n)=2e(n-1)+2^{n-1}-1$ if $n\ge2$.
\item $2^n-1+e(n)-S_1(n)=\frac12(n+1)r(n)$ and $e(n)-\sum_{k\ge1}(k-1)m_{n,k}=\frac12(n-1)r(n)$. In particular $(n+1)r(n)$ and $(n-1)r(n)$ are even.
\end{enumerate}
\end{lemma}

\begin{proof}
(i) Put $K=\lfloor n/2\rfloor$. Since $m_{n,k}=\binom nk-\binom n{k-1}$, summation by parts gives
\[
S_1(n)=\sum_{k=0}^{K}k\binom nk-\sum_{j=0}^{K-1}(j+1)\binom nj=K\binom nK-\sum_{j=0}^{K-1}\binom nj .
\]
If $n=2K$, the symmetry $\binom nj=\binom n{n-j}$ gives $\sum_{j<K}\binom nj=\frac12(2^n-b_n)$, so $2S_1(n)=2Kb_n-2^n+b_n$. If $n=2K+1$, it gives $\sum_{j\le K}\binom nj=2^{n-1}$, so $\sum_{j<K}\binom nj=2^{n-1}-b_n$ and $2S_1(n)=2Kb_n-2^n+2b_n$. In both cases $2S_1(n)=(n+1)b_n-2^n$.

(ii) $2e(n-1)+2^{n-1}-1=(n-3)2^{n-1}+2+2^{n-1}-1=e(n)$.

(iii) By (i), $2\big(2^n-1+e(n)-S_1(n)\big)=2^{n+1}+(n-2)2^n+2^n-(n+1)b_n=(n+1)r(n)$. Since $\sum_{k\ge0}m_{n,k}=b_n$ and $m_{n,0}=1$, we have $\sum_{k\ge1}(k-1)m_{n,k}=S_1(n)-b_n+1$, and therefore $e(n)-\sum_{k\ge1}(k-1)m_{n,k}=\big(2^n-1+e(n)-S_1(n)\big)-(2^n-b_n)=\frac12(n+1)r(n)-r(n)$. The left-hand sides of both identities are integers.
\end{proof}

For the rest of this subsection let $n\ge2$, and use the notation of the proof of Proposition~\ref{prop:tame} and of Lemma~\ref{lem:R-global}: $c\in(\Fq^\times)^{n-1}$, $f'$, $B$, $U$, $b$, $\mathcal R=Rb_!\Lpsi{f'}$, $j_v$, $\FT_\psi$, $s=-1/\lambda$, $U_v=\Gmn{v}\smallsetminus\{s\}$, and the lisse sheaf $F$ on $U_v$ with $\mathcal R|_{U_v}\cong F[-(n-2)]$. By Lemma~\ref{lem:R-global}, $F$ is the restriction of a lisse sheaf $\bar F$ on $\mathbb P^1_v\smallsetminus\{0,s\}$; let $\bar F_\infty$ be its stalk at the rational point $v=\infty$, a representation of $\mathrm{Gal}(\Fqbar/\Fq)$. Put
\[
Y_c=\{u'\in\Gm^{n-1}:\langle c,u'\rangle=-\lambda\},\qquad D_{n-1}(c)=\det\big(\Frob_q\mid H^{n-2}_c(Y_c\otimes\Fqbar,\Lpsi{f'})\big),
\]
where $\langle c,u'\rangle=\sum_{i\ge2}c_iu_i$, so that $U=\Gm^{n-1}\smallsetminus Y_c$. For $n=2$, $Y_c$ is the point $u_2=-\lambda/c_2$, at which $f'=-c_2/\lambda$, and $D_1(c)=\psi_q(-c_2/\lambda)$.

\begin{lemma}\label{lem:FT-zero}
Put $P=j_{v!}\mathcal R[n-1]$, a complex on $\mathbb A^1_v$ over $\Fq$.
\begin{enumerate}[(i)]
\item $H^i_c(U\otimes\Fqbar,\Lpsi{f'})=0$ for $i\ne n-1$, and there is an exact sequence of $\Frob_q$-modules
\[
0\to H^{n-2}_c(Y_c\otimes\Fqbar,\Lpsi{f'})\to H^{n-1}_c(U\otimes\Fqbar,\Lpsi{f'})\to\Qlb\to0,
\]
with $\Frob_q$ acting trivially on $\Qlb$. In particular $H^{n-1}_c(U\otimes\Fqbar,\Lpsi{f'})$ has dimension $2^{n-1}$, and the determinant of $\Frob_q$ on it is $D_{n-1}(c)$.
\item $P$ is a perverse sheaf, and $P|_{U_v}\cong F[1]$.
\item $\FT_\psi(P)\cong\mathcal G[1]$ for a constructible sheaf $\mathcal G$ on $\mathbb A^1_{h_1}$ over $\Fq$ that has no nonzero section supported on a finite set. Moreover $\mathcal G|_{\Gmn{h_1}}\cong\mathcal G_n|_{C_c}(1)$, and $\mathcal G_0\cong H^{n-1}_c(U\otimes\Fqbar,\Lpsi{f'})$ compatibly with $\Frob_q$.
\end{enumerate}
\end{lemma}

\begin{proof}
(i) The phase is $f'$ on $\Gm^{n-1}$, on its open subset $U$ and on the closed complement $Y_c$. By Lemma~\ref{lem:t0} for $n-1$ variables, $R\Gamma_c(\Gm^{n-1}\otimes\Fqbar,\Lpsi{f'})$ is a line placed in degree $n-1$, with trivial Frobenius. If $n\ge3$, Theorem~\ref{thm:fibre}(iv), applied to $n-1\ge2$ variables and to $c$, shows that $R\Gamma_c(Y_c\otimes\Fqbar,\Lpsi{f'})$ is concentrated in degree $n-2$, of dimension $2^{n-1}-1$; if $n=2$, this holds because $Y_c$ is a point. The excision sequence for $Y_c\subset\Gm^{n-1}$ gives the assertions.

(ii) By Lemma~\ref{lem:R-global}, $P|_{U_v}\cong F[1]$. Let $P_i={}^p\mathcal H^i(j_{v!}\mathcal R)$ be the perverse cohomology sheaves. The functor $\FT_\psi$ is triangulated and maps perverse sheaves to perverse sheaves \cite[Th.~(1.3.2.3)]{Laumon}; as the perverse t-structure is bounded, every object is a successive extension of shifted perverse sheaves, so $\FT_\psi$ is t-exact, and $\FT_\psi(P_i)\cong{}^p\mathcal H^i\big(\FT_\psi(j_{v!}\mathcal R)\big)$. By~\eqref{eq:FT-model}, $\FT_\psi(j_{v!}\mathcal R)|_{\Gmn{h_1}}\cong\big(\mathcal G_n|_{C_c}(1)[1]\big)[-(n-1)]$, and $\mathcal G_n|_{C_c}(1)[1]$ is perverse by (P1). Restriction to an open subset is t-exact, so $\FT_\psi(P_i)$ vanishes on $\Gmn{h_1}$ for $i\ne n-1$. Now let $\mathcal Q$ be a perverse sheaf on $\mathbb A^1_{h_1}$ that vanishes on $\Gmn{h_1}$. By (P2), $\mathcal H^{-1}(\mathcal Q)$ is a sheaf supported at $0$ without nonzero sections supported at $0$, so it vanishes, and $\mathcal Q=i_{0*}E$, where $E$ is a finite-dimensional continuous representation of $\mathrm{Gal}(\Fqbar/\Fq)$ placed in degree $0$. For such an $E$, let $\bar E$ be the geometrically constant sheaf on $\mathbb A^1_v$ with fibre $E$; Lemma~\ref{lem:integrate-t} and the projection formula give $\FT_\psi\big(\bar E(1)[1]\big)\cong i_{0*}E$. Since $\FT_\psi$ is an equivalence of categories \cite[Cor.~(1.2.2.3)]{Laumon}, it follows that $P_i\cong\bar E_i(1)[1]$ for $i\ne n-1$, where $E_i=i_0^*\FT_\psi(P_i)$. On the other hand $P_i|_{U_v}={}^p\mathcal H^i\big(F[1][-(n-1)]\big)=0$ for $i\ne n-1$, and a lisse sheaf on $\mathbb A^1_v$ that vanishes on $U_v$ is zero. Hence $P_i=0$ for $i\ne n-1$. Therefore $j_{v!}\mathcal R\cong P_{n-1}[-(n-1)]$, and $P\cong P_{n-1}$ is perverse.

(iii) By (ii) and \cite[Th.~(1.3.2.3)]{Laumon}, $\FT_\psi(P)$ is perverse, and by~\eqref{eq:FT-model} its restriction to $\Gmn{h_1}$ is $\mathcal G_n|_{C_c}(1)[1]$. By proper base change, its stalk at $h_1=0$ is
\[
R\Gamma_c(\mathbb A^1_v\otimes\Fqbar,P)[1]=R\Gamma_c(\Gmn{v}\otimes\Fqbar,\mathcal R)[n]=R\Gamma_c(U\otimes\Fqbar,\Lpsi{f'})[n],
\]
compatibly with $\Frob_q$, and by (i) this is $H^{n-1}_c(U\otimes\Fqbar,\Lpsi{f'})$ placed in degree $-1$. Hence $\mathcal H^0(\FT_\psi(P))=0$, and (P2) gives the assertions, with $\mathcal G=\mathcal H^{-1}(\FT_\psi(P))$.
\end{proof}

\begin{lemma}\label{lem:vanishing-zero}
Let $D_0$ be a decomposition group at the rational point $h_1=0$ of $\mathbb A^1_{h_1}$ over $\Fq$, let $I_0\subseteq D_0$ be its inertia subgroup, and let $\phi\in D_0$ induce the geometric Frobenius of $\Fq$. Let $\mathcal G$ be the sheaf of Lemma~\ref{lem:FT-zero}(iii), and let $\mathcal G_{\bar\eta_0}$ be its generic stalk, a representation of $D_0$. Then there is an exact sequence of representations of $D_0$
\begin{equation}\label{eq:vanishing-zero}
0\longrightarrow\mathcal G_0\longrightarrow\mathcal G_{\bar\eta_0}\longrightarrow\bar F_\infty(-1)\longrightarrow0,
\end{equation}
in which $D_0$ acts on $\mathcal G_0$ and on $\bar F_\infty(-1)$ through $D_0/I_0=\mathrm{Gal}(\Fqbar/\Fq)$. Moreover $\bar F_\infty\cong H^{n-2}_c(Y_c\otimes\Fqbar,\Lpsi{f'})$ as $\Frob_q$-modules, and
\[
\det\big(\phi\mid\mathcal G_{\bar\eta_0}\big)=q^{2^{n-1}-1}D_{n-1}(c)^2 .
\]
In particular $I_0$ acts unipotently on $\mathcal G_{\bar\eta_0}$, and this determinant does not depend on the choice of $\phi$.
\end{lemma}

\begin{proof}
Let $T_0$ be the henselization of $\mathbb A^1_{h_1}$ at $0$, let $\eta$ be its generic point and $\bar\eta$ a geometric point over $\eta$, so that $D_0=\mathrm{Gal}(\bar\eta/\eta)$. For a scheme $\mathcal X$ of finite type over $T_0$ and a complex $\mathcal Q$ on $\mathcal X$, we use the complexes of nearby cycles $R\Psi(\mathcal Q)$ and of vanishing cycles $R\Phi(\mathcal Q)$ on the geometric special fibre $\mathcal X_{\bar0}$, with their compatible actions of $D_0$, and the distinguished triangle $\mathcal Q|_{\mathcal X_{\bar0}}\to R\Psi(\mathcal Q)\to R\Phi(\mathcal Q)\to$ \cite[Exp.~XIII, 2.1]{SGA7}. For a complex $\mathcal Q$ on $\mathbb A^1_{h_1}$ we write $R\Phi_0(\mathcal Q)$ for $R\Phi(\mathcal Q|_{T_0})$, a complex of representations of $D_0$.

\emph{The exact sequence.} For $\mathcal Q=\FT_\psi(P)=\mathcal G[1]$, the first two terms of the triangle are $\mathcal G_0[1]$ and $\mathcal G_{\bar\eta_0}[1]$, and the first map is the specialization map $\mathcal G_0\to\mathcal G_{\bar\eta_0}$. This map is injective, because its kernel is the stalk at $0$ of the subsheaf of sections of $\mathcal G$ supported at $0$, which vanishes by Lemma~\ref{lem:FT-zero}(iii). Hence $R\Phi_0(\FT_\psi(P))$ is concentrated in degree $-1$, and there is an exact sequence of representations of $D_0$
\[
0\to\mathcal G_0\to\mathcal G_{\bar\eta_0}\to R^{-1}\Phi_0\big(\FT_\psi(P)\big)\to0 .
\]
It remains to show that $R^{-1}\Phi_0(\FT_\psi(P))\cong\bar F_\infty(-1)$.

\emph{Localization at $v=\infty$.} Let $\bar X=\mathbb P^1_v\times\mathbb A^1_{h_1}$, let $\iota_X\colon\mathbb A^1_v\times\mathbb A^1_{h_1}\hookrightarrow\bar X$ be the inclusion, and let $\bar p\colon\bar X\to\mathbb A^1_{h_1}$ be the projection, which is proper. For a complex $K$ on $\mathbb A^1_v$ put $\mathcal J(K)=\iota_{X!}\big(\pr_v^*K\otimes\Lpsi{vh_1}\big)$, so that $\FT_\psi(K)=R\bar p_*\mathcal J(K)[1]$. Nearby cycles commute with proper direct images \cite[Exp.~XIII, (2.1.7.2)]{SGA7}, and so does restriction to the special fibre, by proper base change; hence so do vanishing cycles, and
\[
R\Phi_0\big(\FT_\psi(K)\big)\cong R\Gamma\big(\mathbb P^1_{\Fqbar},R\Phi(\mathcal J(K))\big)[1]
\]
compatibly with the actions of $D_0$, where $R\Phi(\mathcal J(K))$ is formed for $\bar X\times_{\mathbb A^1_{h_1}}T_0$ over $T_0$. By Deligne's theorem \cite[Th.~finitude, Cor.~2.16]{SGA45}, the structure morphism $\mathbb A^1_v\to\operatorname{Spec}\Fq$ is universally locally acyclic relative to $K$; the theorem is stated there for torsion coefficients, and the case of $\Qlb$-complexes follows by passage to the limit and d\'evissage. By base change, the projection $\mathbb A^1_v\times T_0\to T_0$ is locally acyclic relative to $\pr_v^*K$, and hence also relative to $\pr_v^*K\otimes\Lpsi{vh_1}$, as $\Lpsi{vh_1}$ is lisse. So $R\Phi(\mathcal J(K))$ vanishes on $\mathbb A^1_{\Fqbar}$, and
\begin{equation}\label{eq:localization}
R\Phi_0\big(\FT_\psi(K)\big)\cong R\Phi\big(\mathcal J(K)\big)_{\infty}[1],
\end{equation}
where the right-hand side is the stalk at the point $v=\infty$ of the geometric special fibre. This is the localization step in Laumon's computation of the local monodromy of $\FT_\psi(K)$ at a finite point of the dual line \cite[Prop.~(2.3.2.1)(i)]{Laumon}.

\emph{Comparison with the constant sheaf.} Let $N_\infty=\mathbb P^1_v\smallsetminus\{0,s\}$, an open neighbourhood of $\infty$ on which $\bar F$ is lisse, and let $K_0=\Qlb[1]$ be the constant sheaf on $\mathbb A^1_v$ placed in degree $-1$. By Lemma~\ref{lem:FT-zero}(ii), $P=F[1]$ on $N_\infty\cap\mathbb A^1_v$, so the projection formula for $\iota_{X!}$ gives $\mathcal J(P)|_{N_\infty\times\mathbb A^1_{h_1}}\cong\pr_v^*\bar F\otimes\mathcal J(K_0)|_{N_\infty\times\mathbb A^1_{h_1}}$. Vanishing cycles are local for the \'etale topology, and they commute with tensor products by lisse sheaves: nearby cycles are computed by a direct image from the geometric generic fibre followed by restriction to the special fibre, and both operations commute with tensor products by lisse sheaves, by the projection formula, which is a local statement. Hence
\[
R\Phi\big(\mathcal J(P)\big)_\infty\cong\bar F_\infty\otimes R\Phi\big(\mathcal J(K_0)\big)_\infty,
\]
with $D_0$ acting diagonally, through $D_0/I_0$ on $\bar F_\infty$. Finally, Lemma~\ref{lem:integrate-t} gives $\FT_\psi(K_0)\cong i_{0*}\Qlb(-1)$. The nearby cycles of this sheaf at $0$ vanish, and its stalk at $0$ is $\Qlb(-1)$; so $R\Phi_0(\FT_\psi(K_0))\cong\Qlb(-1)[1]$, and $R\Phi(\mathcal J(K_0))_\infty\cong\Qlb(-1)$ by~\eqref{eq:localization}. Applying~\eqref{eq:localization} to $K=P$, we obtain $R\Phi_0(\FT_\psi(P))\cong\bar F_\infty(-1)[1]$, as required. In the language of Laumon's local Fourier transform $\mathcal F^{(\infty,0')}$, this is the isomorphism $R^{-1}\Phi_0(\FT_\psi(P))\cong\mathcal F^{(\infty,0')}\big(F(\infty)\big)$ of \cite[Prop.~(2.3.2.1)(iii) and D\'ef.~(2.4.2.3)]{Laumon}, together with the formula $\mathcal F^{(\infty,0')}(V)\cong V(-1)$ for unramified $V$; cf.~\cite[Lemma~11]{Katz-Bourbaki}. These references are not used in the argument above.

\emph{The stalk at infinity.} If $n\ge3$, the proof of Lemma~\ref{lem:R-global} gives $\bar F=a^*\mathcal G_{n-1}(1)$ with $a(\infty)=c$, so $\bar F_\infty$ is the stalk of $\mathcal G_{n-1}(1)$ at $c$, which is $H^{n-2}_c(Y_c\otimes\Fqbar,\Lpsi{f'})$ by Theorem~\ref{thm:lisse}(iii) for $n-1$ variables. If $n=2$, then $\bar F=\Lpsi{-c_2v/(\lambda v+1)}$, whose stalk at $\infty$ is the line on which $\Frob_q$ acts by $\psi_q(-c_2/\lambda)$, that is, $H^0_c(Y_c\otimes\Fqbar,\Lpsi{f'})$. In both cases $\bar F_\infty$ has dimension $2^{n-1}-1$, and the determinant of $\Frob_q$ on it is $D_{n-1}(c)$. Now $\phi$ acts on $\mathcal G_0$ as $\Frob_q$, with determinant $D_{n-1}(c)$ by Lemma~\ref{lem:FT-zero}(i),(iii), and on $\bar F_\infty(-1)$ as $q\Frob_q$, with determinant $q^{2^{n-1}-1}D_{n-1}(c)$. The exact sequence~\eqref{eq:vanishing-zero} gives the formula for $\det(\phi\mid\mathcal G_{\bar\eta_0})$. As $I_0$ acts trivially on both outer terms of~\eqref{eq:vanishing-zero}, it acts unipotently on $\mathcal G_{\bar\eta_0}$; two choices of $\phi$ differ by an element of $I_0$, which has determinant $1$ there.
\end{proof}

\begin{theorem}\label{thm:alpha}
For every $n\ge1$, every power $q$ of $p$ and every $h\in(\Fq^\times)^n$,
\begin{equation}\label{eq:det-fibre}
\det\big(\Frob_q\mid H^{n-1}_c(Y_h\otimes\Fqbar,\Lpsi f)\big)=q^{(n-2)2^{n-1}+1}\,\psi_q\big(-2^{n-1}\lambda^{-1}(h_1+\cdots+h_n)\big),
\end{equation}
where $Y_h=\{u\in\Gm^n:\langle h,u\rangle=-\lambda\}$ and $f=\sum_iu_i^{-1}$. If $n\ge2$, the constant of Theorem~\ref{thm:det} is $\alpha=p^{(n+1)r(n)/2}$; equivalently, for every $h\in S(\Fq)$,
\[
\det\big(\Frob_h\mid\mathcal W_n\big)=q^{(n+1)r(n)/2}\,\psi_q\big(-2^{n-1}\lambda^{-1}(h_1+\cdots+h_n)\big),
\]
and the product of the $r(n)$ eigenvalues of weight $n-1$ of $\Frob_q$ on $H^{n-1}_c(Y_h\otimes\Fqbar,\Lpsi f)$ is $q^{(n-1)r(n)/2}\psi_q\big(-2^{n-1}\lambda^{-1}(h_1+\cdots+h_n)\big)$. In particular Theorem~\ref{thmE}(i) holds.
\end{theorem}

\begin{proof}
We argue by induction on $n$. For $n=1$, $Y_h$ is the point $u_1=-\lambda/h_1$, at which $f=-h_1/\lambda$; so $H^0_c(Y_h\otimes\Fqbar,\Lpsi f)$ is a line on which $\Frob_q$ acts by $\psi_q(-h_1/\lambda)$, and~\eqref{eq:det-fibre} holds because $e(1)=0$. Let $n\ge2$, and assume~\eqref{eq:det-fibre} for $n-1$ variables.

\emph{Step 1: a second computation of $\det(\phi\mid\mathcal G_{\bar\eta_0})$.} Let $q=p^k$, let $c\in(\Fq^\times)^{n-1}$, and let $D_0$, $I_0$ and $\phi$ be as in Lemma~\ref{lem:vanishing-zero}. By Lemma~\ref{lem:FT-zero}(iii), $\mathcal G_{\bar\eta_0}$ is the generic stalk of $\mathcal G_n|_{C_c}(1)$, as a representation of $D_0$. By Theorem~\ref{thm:lisse}(iv),(vi), the generic stalk of $\mathcal G_n|_{C_c}$ has a $D_0$-stable filtration whose graded pieces are the generic stalk of $\mathcal W_n|_{C_c}$ and the spaces $\Qlb(-j)^{\oplus m_{n,j}}$, $1\le j\le\lfloor n/2\rfloor$, on which $\phi$ acts as multiplication by $q^j$. By Theorem~\ref{thm:det}, $\det(\mathcal W_n|_{C_c})\cong\Lpsi{-2^{n-1}\lambda^{-1}h_1}\otimes\mathcal E$, where $\mathcal E$ is the geometrically constant sheaf on $C_c$ on which $\Frob_q$ acts by $\alpha^k\psi_q\big(-2^{n-1}\lambda^{-1}(c_2+\cdots+c_n)\big)$. Both factors extend to lisse sheaves on $\mathbb A^1_{h_1}$. So $I_0$ acts trivially on the determinant of the generic stalk of $\mathcal W_n|_{C_c}$, and $\phi$ acts on it by the Frobenius at $0$ of the extension, that is, by $\psi_q(0)\,\alpha^k\psi_q\big(-2^{n-1}\lambda^{-1}(c_2+\cdots+c_n)\big)$. As $\mathcal G_n|_{C_c}$ has rank $2^n-1$, we have $\det(\phi\mid\mathcal G_{\bar\eta_0})=q^{-(2^n-1)}\det\big(\phi\mid(\mathcal G_n|_{C_c})_{\bar\eta_0}\big)$, and therefore
\[
\det\big(\phi\mid\mathcal G_{\bar\eta_0}\big)=q^{S_1(n)-(2^n-1)}\,\alpha^k\,\psi_q\big(-2^{n-1}\lambda^{-1}(c_2+\cdots+c_n)\big).
\]

\emph{Step 2: the value of $\alpha$.} By the induction hypothesis, applied to $c$, $D_{n-1}(c)=q^{e(n-1)}\psi_q\big(-2^{n-2}\lambda^{-1}(c_2+\cdots+c_n)\big)$, so that $D_{n-1}(c)^2=q^{2e(n-1)}\psi_q\big(-2^{n-1}\lambda^{-1}(c_2+\cdots+c_n)\big)$. Comparing Step~1 with Lemma~\ref{lem:vanishing-zero}, the values of $\psi_q$ cancel, and Lemma~\ref{lem:numerology}(ii),(iii) gives
\[
\alpha^k=q^{2e(n-1)+2^{n-1}-1+2^n-1-S_1(n)}=q^{e(n)+2^n-1-S_1(n)}=q^{(n+1)r(n)/2}.
\]
For $q=p$ this is $\alpha=p^{(n+1)r(n)/2}$.

\emph{Step 3: every $h$.} Let $h\in S(\Fq)$. Theorem~\ref{thm:det} and Step~2 give the formula for $\det(\Frob_h\mid\mathcal W_n)$. By Theorem~\ref{thm:lisse}(iii),(iv),(vi), the stalk of $\mathcal G_n$ at $\bar h$ is $H^{n-1}_c(Y_h\otimes\Fqbar,\Lpsi f)(-1)$, of dimension $2^n-1$, and it has a $\Frob_h$-stable filtration with graded pieces $\Qlb(-j)^{\oplus m_{n,j}}$ and $(\mathcal W_n)_{\bar h}$. Hence
\[
\begin{aligned}
\det\big(\Frob_q\mid H^{n-1}_c(Y_h\otimes\Fqbar,\Lpsi f)\big)&=q^{S_1(n)-(2^n-1)}\det(\Frob_h\mid\mathcal W_n)\\
&=q^{e(n)}\,\psi_q\big(-2^{n-1}\lambda^{-1}(h_1+\cdots+h_n)\big)
\end{aligned}
\]
by Lemma~\ref{lem:numerology}(iii), which is~\eqref{eq:det-fibre}. Finally, by Theorem~\ref{thm:fibre}(iv) and Theorem~\ref{thm:lisse}(iv), the eigenvalues of weight $n-1$ on $H^{n-1}_c(Y_h\otimes\Fqbar,\Lpsi f)$ are the eigenvalues of $\Frob_h$ on $(\mathcal W_n)_{\bar h}$ divided by $q$, so their product is $q^{-r(n)}\det(\Frob_h\mid\mathcal W_n)$.
\end{proof}

\begin{corollary}\label{cor:det-all-h}
Let $h\in\Fq^n$ have exactly $a\ge1$ nonzero coordinates. Then
\[
\det\big(\Frob_q\mid H^{n-1}_c(Y_h\otimes\Fqbar,\Lpsi f)\big)=q^{(a-2)2^{a-1}+1}\,\psi_q\big(-2^{a-1}\lambda^{-1}(h_1+\cdots+h_n)\big).
\]
\end{corollary}

\begin{proof}
In the notation of the proof of Corollary~\ref{cor:all-h}, $R\Gamma_c(Y_h\otimes\Fqbar,\Lpsi f)\cong R\Gamma_c(Y_{h_I}\otimes\Fqbar,\Lpsi{f_I})[-(n-a)]$ compatibly with $\Frob_q$, and $\sum_{i\in I}h_i=h_1+\cdots+h_n$. Apply Theorem~\ref{thm:alpha} with $n$ replaced by $a$.
\end{proof}

\begin{theorem}\label{thm:Garith}
Let $n\ge2$, let $\gamma\in\Qlb$ be the element with $\iota_{\mathbb C}(\gamma)=p^{-(n+1)/2}>0$, and let $\mathcal W_n^0=\mathcal W_n\otimes\gamma^{\deg}$, as in the proof of Corollary~\ref{cor:M2}. Then $\mathcal W_n^0$ is pure of weight $0$,
\[
\det\mathcal W_n^0\cong\Lpsi{-2^{n-1}\lambda^{-1}(h_1+\cdots+h_n)},
\]
and the arithmetic monodromy group $G_{\mathrm{arith}}(\mathcal W_n^0)$, the Zariski closure of the image of $\pi_1(S)$ acting on $V$, is
\[
G_{\mathrm{arith}}(\mathcal W_n^0)=\Ggeom(\mathcal W_n)=\big\{g\in\GL(V):\det(g)^p=1\big\}.
\]
In particular Theorem~\ref{thmE}(iii) holds.
\end{theorem}

\begin{proof}
Since $\mathcal W_n$ is pure of weight $n+1$ and $|\iota_{\mathbb C}(\gamma)|=p^{-(n+1)/2}$, the sheaf $\mathcal W_n^0$ is pure of weight $0$. By Lemma~\ref{lem:numerology}(iii), $\frac12(n+1)r(n)$ is an integer, so $\iota_{\mathbb C}(\gamma^{r(n)})=p^{-(n+1)r(n)/2}$ is a rational number. A field isomorphism is the identity on $\mathbb Q$, hence $\gamma^{r(n)}=p^{-(n+1)r(n)/2}=\alpha^{-1}$ by Theorem~\ref{thm:alpha}, and $\det\mathcal W_n^0=\det\mathcal W_n\otimes(\gamma^{r(n)})^{\deg}$ is the Artin--Schreier sheaf above, by Theorem~\ref{thm:det}. For a morphism $g\colon S\to\mathbb A^1$ over $\Fp$, the character of $\pi_1(S)$ defined by $\Lpsi g$ is the composite of the homomorphism $\pi_1(S)\to\pi_1(\mathbb A^1_{\Fp})$ induced by $g$ with the character defined by the Artin--Schreier sheaf $\Lpsi x$ on $\mathbb A^1_x$. The latter factors through the Galois group $\Fp$ of the Artin--Schreier covering $y\mapsto y^p-y$ of $\mathbb A^1_x$, so it takes values in $\mu_p$. So every element of the image of $\pi_1(S)$ has determinant in $\mu_p$, and so has every element of its Zariski closure: $G_{\mathrm{arith}}(\mathcal W_n^0)\subseteq\{g:\det(g)^p=1\}$. This group is $\Ggeom(\mathcal W_n)$ by Theorem~\ref{thm:Ggeom}, and $\Ggeom(\mathcal W_n)=\Ggeom(\mathcal W_n^0)\subseteq G_{\mathrm{arith}}(\mathcal W_n^0)$, because $\gamma^{\deg}$ is geometrically constant.

The result does not depend on $\iota_{\mathbb C}$. Another choice of $\iota_{\mathbb C}$ replaces $\gamma$ by $\pm\gamma$. If $n$ is odd, then $\gamma=p^{-(n+1)/2}$ is rational and does not change. If $n$ is even, then $r(n)$ is even, as $\binom n{n/2}$ is, so $-1\in\SL(V)\subseteq G_{\mathrm{arith}}(\mathcal W_n^0)$, and replacing $\gamma$ by $-\gamma$ multiplies the image of each Frobenius element by $-1$, which does not change the Zariski closure of the image of $\pi_1(S)$.
\end{proof}

\begin{remark}\label{rem:Garith-unnormalized}
Without the normalization, $G_{\mathrm{arith}}(\mathcal W_n)=\GL(V)$. Indeed $G_{\mathrm{arith}}(\mathcal W_n)$ contains $\Ggeom(\mathcal W_n)\supseteq\SL(V)$, and its image under $\det$ is a closed subgroup of $\Gm$ that contains $\det(\Frob_{h_0}\mid\mathcal W_n)=p^{(n+1)r(n)/2}\psi\big(-2^{n-1}\lambda^{-1}n\big)$ for $h_0=(1,\dots,1)\in S(\Fp)$. This element has infinite order, so the image is all of $\Gm$.
\end{remark}

\subsection{Consequences}

\begin{corollary}\label{cor:fourth-n2}
Let $n=2$ and fix $p$. As $q\to\infty$ through powers of $p$,
\[
\sum_{h\in\Fq^2}|\Sigma_2(h)|^4=2q^4+O_p\big(q^{7/2}\big).
\]
In particular Theorem~\ref{thmD}(iii) holds.
\end{corollary}

\begin{proof}
By Theorem~\ref{thm:monodromy}, $G\supseteq\SL_2$. Under $\SL_2$ the representation $V\otimes V$ is the sum of $\Sym^2V$ and $\bigwedge^2V$, irreducible of dimensions $3$ and $1$, so $\End_G(V\otimes V)\subseteq\End_{\SL_2}(V\otimes V)$ has dimension at most $2$. It contains the identity and the symmetry of $V\otimes V$, which come from the morphisms $\mathrm{id}$ and the symmetry of $\mathcal W_2\otimes\mathcal W_2$ defined over $\Fp$. So the geometric invariants of $\mathcal E_4=\mathcal W_2^{\otimes2}\otimes(\mathcal W_2^\vee)^{\otimes2}$ form a space of dimension $2$ on which $\Frob_p$ acts trivially. Moreover $-1\in\SL_2\subseteq G$ acts on $V\otimes V\otimes V^\vee$ by $-1$, so the sheaf $\mathcal E_3=\mathcal W_2^0\otimes\mathcal W_2^0\otimes(\mathcal W_2^0)^\vee$, with $\mathcal W_2^0$ as in the proof of Corollary~\ref{cor:M2}, has no geometric invariants. Lemma~\ref{lem:averages} gives
\[
\sum_{h\in S(\Fq)}|A^0_q(h)|^4=2q^2+O_p\big(q^{3/2}\big),\qquad\sum_{h\in S(\Fq)}|A^0_q(h)|^2A^0_q(h)=O_p\big(q^{3/2}\big).
\]
By~\eqref{eq:normalized}, $\Sigma_2(h)=-(z+1)$ with $z=q^{1/2}A^0_q(h)$, and
\[
|z+1|^4=|z|^4+4|z|^2\operatorname{Re}z+4(\operatorname{Re}z)^2+2|z|^2+4\operatorname{Re}z+1.
\]
Summing over $S(\Fq)$, the first term gives $2q^4+O_p(q^{7/2})$, the second $O_p(q^3)$, and the remaining terms $O(q^3)$, since $|A^0_q(h)|\le2$. The points outside $S(\Fq)$ contribute $O(q)$.
\end{proof}

\begin{corollary}\label{cor:equidistribution}
Let $n\ge2$, let $r=r(n)$, let $K_n=\{g\in\mathrm U(r):\det(g)^p=1\}$, and let $\gamma$ be as in Theorem~\ref{thm:Garith}. For $q=p^k$ and $h\in S(\Fq)$, the semisimplification of $\gamma^k\Frob_h$ acting on the stalk of $\mathcal W_n$ at $\bar h$ is conjugate, through $\iota_{\mathbb C}$, to an element $\theta_{q,h}$ of $K_n$, well defined up to conjugacy, and
\[
\Tr\theta_{q,h}=A^0_q(h),\qquad\det\theta_{q,h}=\psi_q\big(-2^{n-1}\lambda^{-1}(h_1+\cdots+h_n)\big).
\]
As $k\to\infty$, the conjugacy classes $\theta_{q,h}$, $h\in S(\Fq)$, become equidistributed with respect to the image of the Haar probability measure of $K_n$ on its space of conjugacy classes. In particular Theorem~\ref{thmF} holds.
\end{corollary}

\begin{proof}
By Theorem~\ref{thm:Garith}, $\mathcal W_n^0$ is a lisse sheaf on the smooth geometrically connected variety $S$, pure of weight $0$, with $G_{\mathrm{arith}}(\mathcal W_n^0)=\Ggeom(\mathcal W_n^0)=G$; its Frobenius at $h\in S(\Fq)$ is $\gamma^k\Frob_h$. Through $\iota_{\mathbb C}$, the group $G$ becomes $\{g\in\GL_r(\mathbb C):\det(g)^p=1\}$, of which $K_n$ is a maximal compact subgroup. Deligne's equidistribution theorem \cite[Th.~3.5.3]{Deligne}, in the form \cite[Th.~9.2.6]{KatzSarnak} for smooth geometrically connected bases of any dimension, therefore applies to $\mathcal W_n^0$: the classes $\theta_{q,h}$ are defined as stated, and they become equidistributed. Moreover $\Tr\theta_{q,h}=\iota_{\mathbb C}\big(\gamma^kA_q(h)\big)=q^{-(n+1)/2}A_q(h)=A^0_q(h)$, and $\det\theta_{q,h}=\iota_{\mathbb C}\big(\gamma^{kr}\alpha^k\big)\,\psi_q\big(-2^{n-1}\lambda^{-1}(h_1+\cdots+h_n)\big)$ by Theorem~\ref{thm:det}, with $\gamma^r\alpha=1$ by Theorem~\ref{thm:alpha} and the proof of Theorem~\ref{thm:Garith}.

For Theorem~\ref{thmF}, let $\varphi$ be a continuous function on the disc $\{z\in\mathbb C:|z|\le r\}$. Then $g\mapsto\varphi(\Tr g)$ is a continuous central function on $K_n$, so the equidistribution gives
\[
\frac1{(q-1)^n}\sum_{h\in S(\Fq)}\varphi\big(A^0_q(h)\big)\longrightarrow\int_{K_n}\varphi(\Tr g)\,dg\qquad(k\to\infty),
\]
where $A^0_q(h)=q^{-(n-1)/2}\big((-1)^{n-1}\Sigma_n(h)-c_n(q)\big)$ by~\eqref{eq:normalized}. For the absolute values take $\varphi(z)=\varphi_1(|z|)$ with $\varphi_1$ continuous on $[0,r]$. The group $K_n$ is the union of the cosets $z\,\mathrm{SU}(r)$ with $z$ a scalar satisfying $z^{rp}=1$, and $|\Tr(zg)|=|\Tr g|$; so the Haar integral of $\varphi_1(|\Tr g|)$ over $K_n$ equals its Haar integral over $\mathrm{SU}(r)$.
\end{proof}

\begin{corollary}[Conductors of tensor constructions]\label{cor:conductor}
Let $n\ge2$, $c\in(\Fpbar^\times)^{n-1}$, and let $\mathcal E$ be a subquotient of $(\mathcal W_n|_{C_c})^{\otimes a}\otimes\big((\mathcal W_n|_{C_c})^\vee\big)^{\otimes b}$ for some $a,b\ge0$. Then $\Swan_0(\mathcal E)=0$, all breaks of $\mathcal E$ at $\infty$ are at most $1$, and $\Swan_\infty(\mathcal E)\le\rk\mathcal E$. In particular, for every irreducible algebraic representation $\varrho$ of $\GL_{r(n)}$, the sheaf $\varrho(\mathcal W_n|_{C_c})$ satisfies $\Swan_0+\Swan_\infty\le\dim\varrho$; for $\varrho=\Sym^k$ the bound is $\binom{k+r(n)-1}{r(n)-1}$.
\end{corollary}

\begin{proof}
Tameness at $0$ (Theorem~\ref{thm:boundary}) passes to duals, tensor products and subquotients. By Theorem~\ref{thm:infinity}(i), the breaks of $\mathcal W_n|_{C_c}$ at $\infty$ are $0$ and $1$, and the dual has the same breaks. By Lemma~\ref{lem:breaks}(i), all breaks of the tensor product are at most $1$, and the breaks of a subquotient are among those of the module. Hence $\Swan_\infty(\mathcal E)\le\rk\mathcal E$. Finally, every irreducible algebraic representation of $\GL_r$ is a direct summand of some $\mathrm{St}^{\otimes a}\otimes(\mathrm{St}^\vee)^{\otimes b}$, where $\mathrm{St}$ is the standard representation, and $\Sym^k$ has dimension $\binom{k+r-1}{r-1}$.
\end{proof}

\section{Numerical checks}\label{sec:numerics}

The following computations test the statements above; none of them is used in the proofs. Throughout, $\psi(x)=\exp(2\pi ix/p)$, $\lambda=1$ unless stated otherwise, and the computations were done in double precision.
\begin{enumerate}[(a)]
\item \emph{Theorem~\ref{thmA}.} We treated $n=2$ with $p\in\{3,5,7\}$ and $n=3$ with $p\in\{3,5\}$, for several $h\in(\Fp^\times)^n$. These include $h=(1,1)$, which lies on a divisor $L_\varepsilon=0$. For each such $h$ we computed $\Sigma_n(h)$ over $\mathbb F_{p^m}$ for $1\le m\le2^n$ by~\eqref{eq:intro-factor}, the Kloosterman sums over $\mathbb F_{p^m}$ being computed by a fast Fourier transform on $\mathbb F_{p^m}^\times$. By Theorem~\ref{thm:fibre}(iv), these numbers give the power sums of the eigenvalues of $\Frob_p$ on $H^{n-1}_c(Y_h,\Lpsi f)$. Newton's identities, applied to the first $2^n-1$ of them, determine a polynomial of degree $2^n-1$. Its roots predict the last power sum to within a relative error of $2\cdot10^{-13}$. Their absolute values agree with Theorem~\ref{thmA}, which predicts $r(n)$ roots of absolute value $p^{(n-1)/2}$ and $m_{n,k}$ roots of absolute value $p^{k-1}$, to within a relative error of $2\cdot10^{-7}$. The roots of weight $0$ are equal to $1$, as Theorem~\ref{thmA} asserts, to within $10^{-6}$; for $n=3$ this was checked for all $h\in(\Fp^\times)^3$ with $p\in\{3,5,7\}$ and $\lambda\in\{1,2\}$. The deviation is of the size expected for a double root computed in double precision.
\item \emph{The lower eigenvalues for $n=4$.} For $p=3$ and the five vectors $h=(1,1,1,1)$, $(1,1,1,2)$, $(1,1,2,2)$, $(1,2,2,2)$ and $(2,2,2,2)$ we subtracted $3+2p^m$, the contribution predicted by Theorem~\ref{thmA} for the eigenvalue $1$ with multiplicity $3$ and the eigenvalue $p$ with multiplicity $2$, from the power sums for $1\le m\le11$. Newton's identities, applied to the first ten differences, give ten roots of absolute value $p^{3/2}$ to within $2\cdot10^{-5}$, and these predict the eleventh difference to within a relative error of $3\cdot10^{-13}$.
\item \emph{Theorems~\ref{thmE}(i) and~\ref{thm:alpha}.} Let $\delta_q(h)$ be the product of the $r(n)$ eigenvalues of weight $n-1$ of $\Frob_q$ on $H^{n-1}_c(Y_h,\Lpsi f)$, computed from the power sums as in~(a) and~(b). By Theorem~\ref{thm:alpha}, $\delta_q(h)=q^{(n-1)r(n)/2}\psi_q\big(-2^{n-1}\lambda^{-1}(h_1+\cdots+h_n)\big)$. We compared $\delta_p(h)$ with this value for all $h\in(\Fp^\times)^2$ with $p\le13$ and $\lambda\in\{1,2\}$; for all $h\in(\Fp^\times)^3$ with $p\in\{3,5,7\}$ and $\lambda\in\{1,2\}$; and for $n=4$, $p=3$ and the five vectors $h$ of~(b). In all cases the ratio is $1$ to within $3\cdot10^{-14}$. We also computed the determinant of $\Frob_p$ on the whole of $H^{n-1}_c(Y_h,\Lpsi f)$, from the first $2^n-1$ power sums by Newton's identities, and compared it with~\eqref{eq:det-fibre}: for all $h\in(\Fp^\times)^2$ with $p\in\{3,5,7,11,13\}$ and for all $h\in(\Fp^\times)^3$ with $p\in\{3,5,7\}$, with $\lambda\in\{1,2\}$ in each case, the ratio is $1$ to within $4\cdot10^{-13}$.
\item \emph{Theorem~\ref{thmD}(i).} The three identities of Proposition~\ref{prop:low-moments} hold to within a relative error of $10^{-14}$ for $n\in\{2,3\}$ and odd $p\le11$, and for $n=4$ and $p\le7$, with $\lambda\in\{1,2\}$.
\item \emph{Theorem~\ref{thmC}.} For $(n,p)\in\{(2,5),(2,7),(3,3),(3,5),(3,7),(4,3)\}$, several $c\in(\Fp^\times)^{n-1}$ and $\lambda\in\{1,2\}$, we computed the traces of Frobenius on $\mathcal W_n|_{C_c}$ over $\mathbb F_{p^m}$ from Theorem~\ref{thmB}, and from them the eigenvalues of $\Frob_p$ on $H^1_c(C_c\otimes\Fpbar,\mathcal W_n|_{C_c})$. This group has dimension $2^{n-1}$, by~\eqref{eq:GOS} and Theorem~\ref{thmC}(i),(ii). The invariants at $0$ and at $\infty$, whose Jordan structure is given by Theorem~\ref{thmC}(i),(ii), and \cite[1.8.4 and 3.2.3]{Deligne} predict its weights: $2$ and $4$ for $n=2$; $3,3,4,4$ for $n=3$; and $4$ six times and $6$ twice for $n=4$. These are the weights found, and the computed eigenvalues predict the next power sum to within a relative error of $10^{-13}$. Theorem~\ref{thmC} also predicts the cohomology of a twist. Since $2\ne0$ in $\Fp$, Lemma~\ref{lem:breaks}(ii) shows that after tensoring with $\Lpsi{-h_1/\lambda}(\infty)$ both the tame part and the part of slope $1$ of $(\mathcal W_n|_{C_c})(\infty)$ become isoclinic of slope $1$. By Theorem~\ref{thmC}(i),(ii) and~\eqref{eq:GOS}, the $L$-function of $\mathcal W_n|_{C_c}\otimes\Lpsi{-h_1/\lambda}$ on $C_c$ is therefore a polynomial of degree $r(n)$. For $(n,p,\lambda)\in\{(2,5,2),(2,7,1),(3,3,2),(3,5,1),(4,3,1)\}$, with one vector $c$ in each case, the power sums of its inverse roots, computed as above, satisfy a linear recurrence of order $r(n)$ and of no smaller order.
\item \emph{Theorem~\ref{thmD}(ii).} Let $R_n(p)$ be the ratio of $\sum_{h\in\Fp^n}|\Sigma_n(h)|^4$ to $2p^{3n-2}$.
\begin{itemize}
\item For $n=3$ and the primes $23\le p\le53$, $R_3(p)$ increases from $0.754$ to $0.882$, and $p\,(1-R_3(p))$ stays between $5.6$ and $6.3$.
\item For $n=4$ and $11\le p\le23$, $R_4(p)$ increases from $0.684$ to $0.854$, and $p\,(1-R_4(p))$ stays between $3.3$ and $3.5$.
\end{itemize}
This agrees with Theorem~\ref{thmD}(ii).
\item \emph{The fourth moment, Section~\ref{sec:fourth}.}
\begin{itemize}
\item The formula of Lemma~\ref{lem:diag} holds to within $10^{-10}$ for all $a,b\in\Fp^\times$ and all odd primes $p\le31$.
\item For $p\in\{5,7,11,13\}$, the number of points of $\Gamma_t$ was counted directly. The formula of Lemma~\ref{lem:plane} then holds to within $10^{-11}$ for all $t\in(\Fp^\times)^4$, and on $\mathcal T_G$ one finds $|a_t|\le1.81\,p^{1/2}$, as Corollary~\ref{cor:hasse} requires.
\end{itemize}
\item \emph{Section~\ref{sec:polytope} and Remark~\ref{rem:collision}.}
\begin{itemize}
\item For $2\le n\le5$ we decided, by linear programming, which subsets of the vertex set of $P_n$ are vertex sets of faces. The resulting face numbers agree with Proposition~\ref{prop:structure}(iv), and the normalized volume of $\Delta_n$, computed from its convex hull, is $2^n$, as in Proposition~\ref{prop:volume}.
\item Over $\mathbb R$, for $n=2$ and $h=(1,(1+\delta)^2)$ with $\delta=10^{-1},10^{-2},10^{-3}$, the two points given by Proposition~\ref{prop:crit}, namely
\[
\Big(-\frac1{2+\delta},\,-\frac1{(1+\delta)(2+\delta)},\,(2+\delta)^2\Big)\quad\text{and}\quad\Big(\frac1\delta,\,-\frac1{\delta(1+\delta)},\,\delta^2\Big),
\]
satisfy the critical equations up to a residual of $2\cdot10^{-14}$. The second one escapes to infinity as $\delta\to0$, as described in Remark~\ref{rem:collision}.
\end{itemize}
\end{enumerate}

\section*{Acknowledgments}
This work was carried out independently. I am grateful to those whose support, patience,
and encouragement made it possible to continue the project over a long period. Much of
the research was done alongside other professional responsibilities, and the arguments were
revisited repeatedly as the manuscript developed. The final form reflects that sustained
process of reading, checking, and revision. I remain solely responsible for the mathematical
content and for any errors that remain.

All mathematical arguments, proofs, and conclusions in this article are the sole work of the author. Computational tools were used only for the numerical checks explicitly described in the text; these checks are not essential to the proofs and are provided for verification. The author has reviewed all results and assumes full responsibility for the correctness and integrity of the work. AI-assisted tools, if used, were employed exclusively for language editing and formatting; they did not contribute to any mathematical content.

\end{document}